\documentclass[11pt]{article}

\usepackage[top=50mm, bottom=50mm, left=50mm, right=50mm]{geometry}
\usepackage{lineno}

\usepackage{amssymb}
\usepackage{mathrsfs}
\usepackage[final]{showkeys}
\usepackage{amsmath}
\usepackage{amsthm}
\usepackage{dsfont}
\usepackage{epsfig}
\usepackage{graphicx}
\usepackage{graphics}
\usepackage{float}
\usepackage{caption}
\usepackage{subcaption}
\usepackage{multirow}
\usepackage{color}
\usepackage{lineno}
\usepackage{fullpage}
\usepackage[normalem]{ulem} 
\usepackage{makeidx}
\usepackage{xspace}
\usepackage{wrapfig}
\usepackage{mathtools}
\usepackage{stmaryrd}
\usepackage{authblk}
\usepackage{etoolbox}
\usepackage{truncate}
\usepackage{xkeyval}
\usepackage{xstring}
\usepackage{soul}
\sethlcolor{yellow}

\usepackage{hyperref}
\hypersetup{
colorlinks = true,
linkcolor = {red},
urlcolor = {blue},
citecolor = {blue}
}

\usepackage[draft]{changes}
\definechangesauthor[name={Quan}, color=orange]{quan}
\usepackage{appendix}

\makeindex

\newtheorem{theorem}{Theorem}

\newtheorem{corollary}[theorem]{Corollary}
\newtheorem{proposition}[theorem]{Proposition}
\newtheorem{lemma}[theorem]{Lemma}

\newtheorem{observation}[theorem]{Observation}

\newtheorem{example}{Example}

\theoremstyle{definition}
\newtheorem{hypothesis}{Hypothesis}
\newtheorem{definition}[theorem]{Definition}
\newtheorem*{demasi}{Illustration by Example $1$}
\theoremstyle{remark}
\newtheorem{remark}{Remark}

\newcommand{\Ebb}{\mathbb{E}}

\newcommand{\Rbb}{\mathbb{R}}

\newcommand{\Zbb}{\mathbb{Z}}
\newcommand{\Pbb}{\mathbb{P}}

\newcommand{\oneds}{\mathds{1}}

\newcommand{\Acal}{\mathcal{A}}

\newcommand{\Ccal}{\mathcal{C}}

\newcommand{\Ecal}{\mathcal{E}}
\newcommand{\Fcal}{\mathcal{F}}
\newcommand{\Gcal}{\mathcal{G}}

\newcommand{\Rcal}{\mathcal{R}}

\newcommand{\Zcal}{\mathcal{Z}}
\newcommand{\Lcal}{\mathcal{L}}
\newcommand{\Ocal}[1]{\mathcal{O}\left(#1\right)}
\newcommand{\Xcal}{\mathcal{X}}
\newcommand{\Tcal}{\mathcal{T}}
\newcommand{\Ucal}{\mathcal{U}}
\newcommand{\Ycal}{\mathcal{Y}}

\newcommand{\efrak}{\mathfrak{e}}
\newcommand{\sfrak}{\mathfrak{s}}

\newcommand{\Lfrak}{\mathfrak{L}}

\newcommand{\Yfrak}{\mathcal{M}}
\newcommand{\drm}{\mathrm{d}}
\newcommand{\Zbar}{\Bar{Z}}
\newcommand{\Ubar}{\Bar{U}}

\newcommand{\Ftilde}{\tilde{F}}
\newcommand{\pivtilde}{\widetilde{\mathsf{Piv}}}
\newcommand{\WW}{\widetilde{W}}
\newcommand{\sft}{\tilde{\sfrak}}
\newcommand{\Ttilde}{\tilde{T}}
\newcommand{\ftilde}{\tilde{f}}
\newcommand{\Utilde}{\widetilde{U}}
\newcommand{\utilde}{\check{u}}
\newcommand{\Tcaltilde}{\tilde{\Tcal}}

\newcommand{\Ytilde}{\tilde{Y}}
\newcommand{\Ucaltilde}{\Check{\mathcal{U}}}

\newcommand{\wsf}{\mathsf{w}}

\newcommand\iid{i.i.d. }

\newcommand{\ie}{i.e. }
\newcommand{\eg}{e.g. }

\newcommand{\ddt}{\dfrac{d}{dt}}
\newcommand{\esperance}[1]{\Ebb\left[ #1 \right]}
\newcommand{\variance}[1]{\textrm{Var}\left[ #1 \right]}

\newcommand{\covariance}[2]{\textrm{Cov}\left( #1, #2 \right)}
\newcommand{\esperancewithstartingpoint}[2]{\Ebb_{#1}\left[ #2 \right]}

\newcommand{\proba}[1]{\Pbb\left[ #1 \right]}
\newcommand{\probawithstartingpoint}[2]{\Pbb_{#1}\left[ #2 \right]}

\newcommand{\norm}[1]{\left\| #1 \right\|}
\newcommand{\normtv}[1]{\left\| #1 \right\|_{TV}}

\newcommand{\indicator}[1]{\mathds{1}_{\left\{#1\right\}}}
\newcommand{\eventindicator}[1]{\mathds{1}_{#1}}

\newcommand{\tmixxepsilon}{\textrm{t}_\textsc{mix}(x;\epsilon)}

\newcommand{\restr}[2]{{
  \left.\kern-\nulldelimiterspace 
  #1 
  \vphantom{\big|} 
  \right|_{#2} 
  }}

\newcommand{\ceil}[1]{\left\lceil #1 \right\rceil}
\newcommand{\floor}[1]{\left\lfloor #1 \right\rfloor}
\newcommand{\tmix}{\textrm{t}_\textsc{mix}(\epsilon)}
\newcommand{\minset}[1]{\inf \left\{ #1  \right\}}
\newcommand{\maxset}[1]{\max \left\{ #1  \right\}}
\newcommand{\infset}[1]{\inf \left\{ #1  \right\}}
\newcommand{\dtv}[2]{\textrm{d}_\textsc{tv}\left( #1, #2 \right)}
\newcommand{\whp}{w.h.p}
\newcommand{\wrt}{w.r.t }
\newcommand{\absolutevalue}[1]{\left|#1\right|}

\newcommand{\piv}[1]{\mathsf{Piv}\left(#1\right)}
\newcommand{\etal}{\textit{et al}}

\newcommand{\rade}[1]{\mathfrak{R}_{#1}}

\newcommand{\sigmaalgebra}{$\sigma$-algebra }

\newcommand{\xirefresh}{\Xi^{\text{refresh}}}
\newcommand{\xistarrefresh}{\Xi^{\text{refresh,*}}}
\newcommand{\xiglauber}{\Xi^{\text{Glauber}}}
\newcommand{\xistarglauber}{\Xi^{\text{Glauber,*}}}
\newcommand{\xiblack}{\Xi^{\text{black}}}
\newcommand{\xiblue}{\Xi^{\text{blue}}}
\newcommand{\xiexclusion}{\Xi^{\text{exclusion}}}
\newcommand{\xistarexclusion}{\Xi^{\text{exclusion,*}}}

\newcommand{\rhostar}{\rho_*}
\newcommand{\rhobar}{\Bar{\rho}}
\newcommand{\badset}{\mathrm{BAD}}

\newcommand{\setvalue}{:= }
\newcommand{\mutilde}{\tilde{\mu}}

\newcommand{\generator}{\Lcal}
\newcommand{\generatorglauber}{\Lcal_G}
\newcommand{\generatorexclusion}{\Lcal_E}
\newcommand{\redset}{\mathrm{Red}}
\newcommand{\blueset}{\mathrm{Blue}}
\newcommand{\greenset}{\mathrm{Green}}

\newcommand{\cardinalred}{\Rcal}

\newcommand{\lattice}{\Lambda_L}
\newcommand{\latticedimensiond}{\Lambda^d_L}

\newcommand{\integerinterval}[2]{\left\llbracket #1, #2 \right\rrbracket}
\newcommand{\distance}[1]{\mathsf{dist}\left(#1\right)}

\newcommand{\plusminusone}{\{-1,1\}}
\newcommand{\rhoplus}{\rho_+}
\newcommand{\rhominus}{\rho_-}
\newcommand{\gap}{\texttt{gap}}

\newcommand{\longeur}{L}
\newcommand{\gw}{\texttt{Gw}}

\newcommand{\tauzerotoone}{\tau^{0\to 1}}
\newcommand{\tauonetozero}{\tau^{1\to 0}}
\newcommand{\tautwotoone}{\tau^{2\to 1}}
\newcommand{\tautildezerotoone}{\tilde{\tau}^{0\to 1}}
\newcommand{\tautildeonetozero}{\tilde{\tau}^{1\to 0}}
\newcommand{\taucheckzerotoone}{\check{\tau}^{0\to 1}}
\newcommand{\taucheckonetozero}{\check{\tau}^{1\to 0}}
\newcommand{\Tcalcheck}{\check{\Tcal}}
\newcommand{\Ccalexclusion}{\Ccal^{\text{exclusion}}}
\newcommand{\Ccalglauber}{\Ccal^{\text{glauber}}}

\begin{document}

\title{\LARGE Cutoff for the Glauber-Exclusion process in the full high-temperature regime: an information percolation approach}
 
\author{\Large Hong-Quan Tran\thanks{tran@ceremade.dauphine.fr}}
\affil{\large CEREMADE, CNRS, Université Paris-Dauphine, PSL University 75016 Paris, France}
\maketitle
\begin{abstract}
    The Glauber-Exclusion process is a superposition of a Glauber dynamics and the Symmetric Simple Exclusion Process (SSEP) on the lattice. The model was shown to admit a reaction-diffusion equation as the hydrodynamic limit. In this article, we define a notion of temperature regimes via the reaction function in the equation and prove cutoff in the full high-temperature regime for the attractive model in dimensions $1$ and $2$ with periodic boundary condition. Our results show that the equation in the hydrodynamic limit reflects the mixing behavior of the large but finite system. Besides, cutoff is proved under the lack of reversibility and an explicit formula for the invariant measure. We also provide the spectral gap and prove pre-cutoff in all dimensions. Our proof involves a new interpretation of attractiveness, the information percolation framework introduced by Lubetzky and Sly, anti-concentration of simple random walk on the lattice, and a coupling inspired by excursion theory. We hope that this approach can find new applications in the future.  
\end{abstract}
\tableofcontents
\section{Introduction}
The Glauber-Exclusion process, introduced in \cite{DeMasi1985, DeMasi1986}, is a superposition of a Glauber dynamics and an appropriately accelerated symmetric simple exclusion process (SSEP) on a lattice. The model is also called Glauber-Kawasaki dynamics in \cite{DeMasi2019}. Here, we use the term Glauber-Exclusion as in \cite{Tanaka2022}. The model can be described as follows.
\paragraph{Notations.} Fix a number $d \in \Zbb_+$. For any $L \in \Zbb_+$, let $\latticedimensiond := (\Zbb/L\Zbb)^d$ be the torus of side-length $L$ in dimension $d$. For any finite set $\Omega$, we call the elements of $\{-1,1\}^\Omega$ the spin configurations on $\Omega$, and we denote by $\leq$ the usual coordinate-wise order on $\{-1,1\}^\Omega$. Our state space is $\Xcal := \plusminusone^{\latticedimensiond}$, the set of all spin configurations on $\latticedimensiond$. For any $x \in \Xcal$, $u,u' \in \latticedimensiond$, we denote by $x^{u \leftrightarrow u'}$ the configuration obtained from $x$ by exchanging the values at two sites $u$ and $u'$, and $x^{u,1}$ (resp. $x^{u, -1}$) the configuration obtained from $x$ by replacing the value at site $u$ by $1$ (resp. $-1$). We also denote by $\exp(\theta)$ the exponential distribution with parameter $\theta$, \ie the distribution with density $\indicator{s>0}\theta e^{-\theta s} \drm s$. 

The generator of the symmetric simple exclusion process (SSEP) $\generatorexclusion$ acts on an observable $\varphi: \Xcal \to \Rbb$ by  
\begin{equation}
    (\generatorexclusion \varphi)(x) := \sum_{u\sim u'} \left(\varphi(x^{u\leftrightarrow u'}) - \varphi(x)\right),
\end{equation}
where the sum is taken over all pairs of neighbors $\{u,u'\}$ on the lattice. This means that $\generatorexclusion$ exchanges the contents of two neighboring sites $u, u'$ at rate $1$. 

To define the Glauber dynamics, we first need the following definition.
\paragraph{Flip-rate function.} Throughout the paper, we fix a number $m \in \Zbb_+$ and a function $c: \{-1,1\}^{B(0,m)} \to \Rbb_{>0}$, where $B(0,m)$ is the subset of $\Zbb^d$ consisting of all sites of distance at most $m$ from $0$. We call this function $c$ the \emph{local flip-rate function}.   

For any $L$ large enough, we identify $B(0,m)$ with the subset of $\latticedimensiond$ consisting of all sites of distance at most $m$ from site $0$. The function $c$ extends to a \emph{global flip-rate function} $\hat{c}: \latticedimensiond \times \Xcal \to \Rbb_{>0}$ as follows.
\begin{enumerate}
    \item $\hat{c}(0, \cdot)$ is local and given by $c$: 
    \begin{align*}
        \forall x \in \Xcal, \;\hat{c}(0,x) := c\left(x|_{B(0,m)}\right).
    \end{align*}
    \item $\hat{c}$ is invariant by translation: 
    \begin{align*}
         \forall x \in \Xcal, u \in \latticedimensiond,\; \hat{c}(u,x) := \hat{c}(0, x_{u+\cdot}),
    \end{align*}
    where 
    \begin{align*}
        x_{u + \cdot}(u') := x(u + u'),\;\forall x \in \Xcal, \, u, u' \in \latticedimensiond.
    \end{align*}
\end{enumerate}
The generator $\generatorglauber$ of the Glauber dynamics associated with the local flip-rate function $c(\cdot)$ acts on an observable $\varphi: \Xcal \to \Rbb$ by  
\begin{equation}
    (\generatorglauber \varphi) (x) := \sum_{u \in \latticedimensiond} \hat{c}(u, x)(\varphi(x^{u, -x(u)}) - \varphi(x)).
\end{equation}
This means that $\generatorglauber$ flips the spin on site $u$ of configuration $x$ at rate $\hat{c}(u, x)$, for any $(u,x) \in \latticedimensiond \times \Xcal$.
We suppose that the Glauber dynamics is attractive.  
\begin{hypothesis}[Attractiveness]\label{hyp:attractive}
    For any $x, y \in \{-1,1\}^{B(0,m)}$ such that $x \leq y$,
            \begin{align}
                c(x) &\geq c(y) \text{ if } x(0) = y(0) = 1 ,\\
                c(x) &\leq c(y) \text{ if } x(0) = y(0) = -1.
            \end{align}
\end{hypothesis}
Hypothesis \ref{hyp:attractive} implies that, for any $x,y \in \Xcal$ such that $x \leq y$, for any $u \in \latticedimensiond$,
\begin{align}
    \hat{c}(u,x) &\geq \hat{c}(u,y) \text{ if } x(u) = y(u) = 1 \label{eql1},\\
    \hat{c}(u,x) &\leq \hat{c}(u,y) \text{ if } x(u) = y(u) = -1.\label{eql2}
\end{align}
These conditions are necessary and sufficient to construct a Markovian coupling of two processes whose generators are $\generatorglauber$ that preserves order. More precisely, suppose that $(Y^1_t)_{t\geq 0}$ and $(Y^2_t)_{t\geq 0}$ are two Markov processes with generator $\generatorglauber$ starting at $y_1, y_2$, respectively. Then there is a Markovian coupling of $(Y^1_t)_{t\geq 0}$ and $(Y^2_t)_{t\geq 0}$ such that, almost surely, 
$$y_1 \leq y_2 \Rightarrow \forall t \geq 0,\;Y^1_t \leq Y^2_t.$$ See Theorem 4.11, p.143 in \cite{Liggett2010} for more details. When such an order-preserving coupling exists, we say that the process $Y$ is attractive.

Although we will work with a general function $c(\cdot)$, we usually use the following example considered by De Masi \etal\, in \cite{DeMasi1986} to illustrate our method.

\begin{example}[Example of De Masi \etal]\label{ex:demasi}
    The model considered is in dimension one, \ie $d = 1$, with the local flip-rate function $c$ given by
    \begin{equation*}
        \forall x \in \{-1,1\}^{\{1,0,-1\}},\; c(x) := 1 - \gamma x(0)(x(1) + x(-1)) + \gamma^2 x(1) x(-1),
    \end{equation*}
    for some $\gamma \in [0,1]$.
    This corresponds to the global flip-rate function $\hat{c}$ given by
    \begin{equation*}
        \forall x \in \Xcal,u \in \lattice, \; c(u, x) := 1 - \gamma x(u)(x(u+1) + x(u-1)) + \gamma^2 x(u+1) x(u-1).
    \end{equation*}
\end{example}

The \emph{Glauber-Exclusion process} is a continuous time Markov process $\textbf{X} = (X_t)_{t \geq 0}$, taking values in $\Xcal$, whose generator is 
\begin{equation}
    \generator_{GE} := L^2 \generatorexclusion + \generatorglauber.
\end{equation}
This is a superposition of a Glauber dynamics and a SSEP accelerated by a diffusive factor $L^2$.

For $\rho \in [-1,1]$, we denote by $\rade{\rho}$ the Rademacher probability measure on $\{-1,1\}$ with average $\rho$. More precisely, 
\begin{align*}
    \rade{\rho}(1) &= \dfrac{1+\rho}{2},\\
    \rade{\rho}(-1) &= \dfrac{1-\rho}{2}.
\end{align*}
\paragraph{Reaction function.}
We denote by $\nu_\rho$ the product measure of $\rade{\rho}$ on $\{-1,1\}^{B(0,m)}$. We define the polynomial $R(\cdot): [-1,1] \to \Rbb$ by 
\begin{align*}
    \forall \rho \in [-1,1],\; R(\rho) := \esperancewithstartingpoint{\nu_\rho}{-2\xi_0 c(\xi)},
\end{align*}
where $\xi = (\xi_u)_{u\in B(0,m)} \sim \nu_\rho$. We call $R(\cdot)$ the \emph{reaction function}. The reason behind this name will be explained in the next paragraph.

\paragraph{The hydrodynamic limit.} Originally, De Masi, Ferrari, and Lebowitz introduced the model in \cite{DeMasi1985, DeMasi1986} for the infinite lattice $\Zbb^d$ and proved that the hydrodynamic limit of the model is a reaction-diffusion equation. We recall here the analogous result obtained by Kipnis, Olla, and Varadhan \cite{Kipnis1989} for our case in a finite box. Let the density field on $(\Rbb/\Zbb)^d$ at time $t$ be defined by
\begin{equation*}
    \mu_t^L (\drm u) = \dfrac{1}{L^d} \sum_{u \in \latticedimensiond} X_t(u) \delta_{u/L}(\drm u) \;\; \text{on } (\Rbb/\Zbb)^d.
\end{equation*}
Then, when $L$ grows to infinity, under appropriate convergence conditions on the initial data, the density field converges weakly to the solution $\rho:\Rbb_+ \times (\Rbb/\Zbb)^d \to [-1,1]$ of the following reaction-diffusion equation: 
\begin{equation}\label{eq:hydro}
    \dfrac{\partial \rho}{\partial t} = \Delta \rho + R(\rho),
\end{equation}
where $\Delta$ is the Laplacian in $(\Rbb/\Zbb)^d$, and $R(\cdot)$ is the function defined above. Equation \eqref{eq:hydro} explains the name reaction function of $R(\cdot)$. The intuition behind this result is as follows. Applying the generator $\generator_{GE}$ to the function $\varphi: x \mapsto x(u)$, we see that 
\begin{equation*}
    \generator \varphi(x) = \left(L^2 \sum_{u' \sim u} (x(u') - x(u))\right) -2x(u) \hat{c}(u,x).
\end{equation*}
Therefore,
\begin{equation*}
    \ddt \esperance{X_t(u)} = \esperance{\generator \varphi(X_t)} = \left(L^2 \sum_{u'\sim u} \esperance{X_t(u') - X_t(u)}\right) -2\esperance{X_t(u) \hat{c}(u, X_t)}.
\end{equation*}
Suppose that $\esperance{X_t(u)}$ is well approximated by $\rho(t, u/L)$ for some smooth function $\rho$, for any site $u$ and any fixed time $t$. Then
\begin{align*}
    L^2 \sum_{u'\sim u} \esperance{X_t(u') - X_t(u)} \approx L^2 \sum_{u'\sim u} \left(\rho(t, u'/L) - \rho(t, u/L)\right) \approx \Delta \rho(t, u/L).
\end{align*}
Moreover, the local law around a site $u$ at any fixed time $t$ is believed to be close to that of a product of $\rade{\esperance{X_t(u)}}$, a phenomenon known as \emph{local equilibrium}, see Subchapter $1.2$ and Chapter 3 in \cite{Kipnis1999} for more details. Then
\begin{align*}
    \esperance{X_t(u) \hat{c}(u, X_t)} \approx \esperancewithstartingpoint{\nu_{\esperance{X_t(u)}}}{\xi_0 c(\xi)} = R\left(\esperance{X_t(u)}\right) \approx  R(\rho(t, u/L)).
\end{align*}
These intuitive arguments indicate that $\rho(\cdot, \cdot)$ satisfies equation \eqref{eq:hydro}. Showing the convergence amounts to verifying the above intuition rigorously. 

Based on this result, we define the \emph{temperature regimes} as follows.
\paragraph{Temperature regimes.}
\begin{itemize}
    \item \emph{High-temperature regime}: The system is at high temperature if the reaction function $R(\cdot)$ has a unique root $\rho_*$ in $[-1,1]$, and $R'(\rho_*) < 0$.
    \item \emph{Critical-temperature regime}: The system is at critical temperature if the reaction function $R(\cdot)$ has a unique root $\rho_*$ in $[-1,1]$, and $R'(\rho_*) = 0$.
    \item \emph{Low-temperature regime}: The system is at low temperature if the reaction function $R(\cdot)$ has at least two roots in $[-1,1]$. 
\end{itemize}

\paragraph{Explanation for the temperature regimes.} Let us explain why we define the temperature regimes for the model as above. Consider the equation \eqref{eq:hydro} with the initial condition $\rho(0, \cdot)$ constant in space, say $\rho(0, \cdot) = \rho_0$ for some $\rho_0 \in [-1,1]$. We then see that $\rho(t,\cdot)$ must also be constant in space, so we can safely denote by $\rho(t)$ the value of the function $\rho$ at time $t$. The equation $\eqref{eq:hydro}$ then simply becomes an ODE:

\begin{equation}\label{eq:ODE}
    \begin{cases}
        \rho' =  R(\rho),\\
        \rho(0) = \rho_0.
    \end{cases}
\end{equation}
Though simple, the solution of $\eqref{eq:ODE}$ can exhibit very different behaviors, depending on $R(\cdot)$. Consider the flip-rate function in Example \ref{ex:demasi}. Then $R(\rho) = -2(1-2\gamma) \rho - 2\gamma^2 \rho^3$. One can verify that the system is at high temperature if $0\leq \gamma < 1/2$, at critical temperature if $\gamma = 1/2$, and at low temperature if $1/2 < \gamma \leq 1$.
We present in Figure \ref{fig:phase-diagram} the phase diagrams of the ODE \eqref{eq:ODE} for two different values of $\gamma$.
\begin{figure}
    \centering
    \begin{subfigure}{0.45\textwidth}
        \centering
        \includegraphics[width=\textwidth]{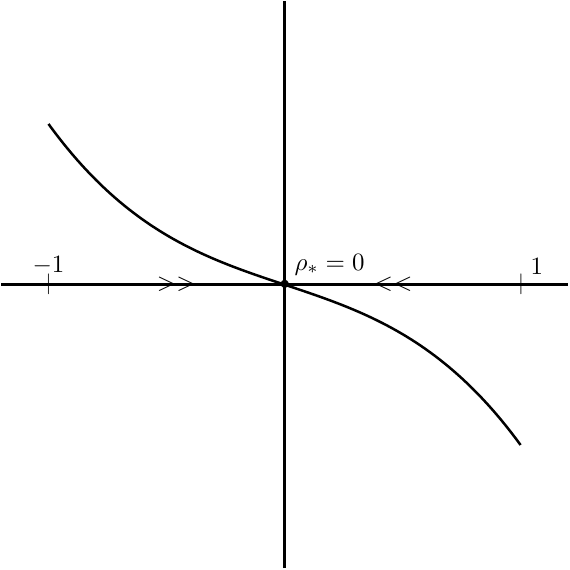}
        \caption{High temperature}
    \end{subfigure}
    \hfill
    \begin{subfigure}{0.45\textwidth}
        \centering
        \includegraphics[width=\textwidth]{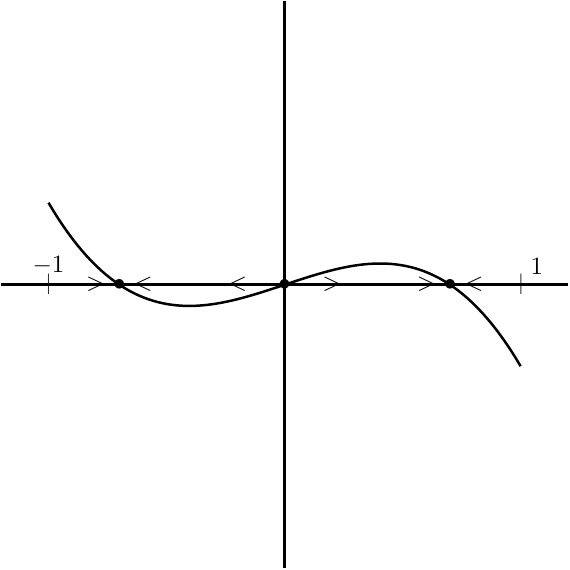}
        \caption{Low temperature}
    \end{subfigure}
    \caption{Phase diagram for $\gamma = 5/12$ and $\gamma = 7/12$.}
    \label{fig:phase-diagram}
\end{figure}

Let $\rhoplus(\cdot)$ and $
\rhominus(\cdot)$ denote the solutions of the ODE \eqref{eq:ODE} with initial conditions $\rhoplus(0) := 1$ and $\rhominus(0) := -1$. Viewing the diagram, we see that, in the first case, for any initial condition $\rho_0$, the solution of ODE \eqref{eq:ODE} converges to the unique root $\rho_*$ exponentially fast: 
$\absolutevalue{\rho(t) - \rho_*} \lesssim e^{R'(\rho_*)t}$. In the second case, $R(\cdot)$ has 3 roots $\rho^{(1)} < \rho^{(2)} < \rho^{(3)}$, and $\rhominus(t) \xrightarrow{t \to \infty} \rho^{(1)}$, and $\rhoplus(t)\xrightarrow{t \to \infty} \rho^{(3)}$. Intuitively, this implies that the system mixes rapidly at high temperature, while at low temperature, it exhibits the phenomenon known as \emph{metastability}, see \cite{Bovier2015} for more details. 

In this paper, we study the model in the high-temperature regime. 

\begin{hypothesis}[High temperature]\label{hyp:high_temperature}\hfill
    \begin{itemize}
        \item $R(\rho)$ admits a unique root $\rhostar$ in $[-1,1]$,
        \item $R'(\rhostar) < 0$.
    \end{itemize}    
\end{hypothesis}

The primary purpose of this paper is to study the mixing times of the system in the high-temperature regime. We begin by explaining some terminologies. 
\paragraph{The mixing times and the spectral gap.} We use the notation $\probawithstartingpoint{x}{\cdot}$ for the law of the process $X$ starting from $x$. The strict positivity of $c(\cdot)$ ensures that the generator $\Lcal^{GE}$ is irreducible. This implies that there exists a unique invariant probability distribution $\pi$, and the distribution of the process converges to $\pi$ when the time tends to infinity:
\begin{equation}
    \forall x, y \in \Xcal, \; \probawithstartingpoint{x}{X_t = y} \xrightarrow[t \to \infty]{} \pi(y).
\end{equation}
This convergence is often measured by the total variation distance, $\dtv{\cdot}{\cdot}$, defined by
\begin{equation}
    \dtv{\mu}{\nu} = \max_{A \subset \Xcal}|\mu(A) - \nu(A)|,
\end{equation}
for any two probability distributions $\mu, \nu$ on $\Xcal$. The speed of convergence is quantified by the \emph{mixing-times}:
\begin{equation}
    \tmix := \inf\left\{t\geq 0: \max_{x \in \Xcal} \dtv{\probawithstartingpoint{x}{X_t \in \cdot}}{\pi} \leq \epsilon\right\},\;\; \epsilon \in (0,1).
\end{equation}
The smallest real part of the non-zero eigenvalues of $-\generator_{GE}$, also called the \emph{spectral gap} of $\Lcal^{GE}$, denoted by $\gap$, governs the asymptotic speed of convergence to equilibrium:
\begin{equation}
    \gap = \lim_{t\to \infty} -\dfrac{1}{t} \log \max_{x \in \Xcal} \dtv{\probawithstartingpoint{x}{X_t \in \cdot}}{\pi}.
\end{equation}

\paragraph{The cutoff phenomenon.} Note that the mixing times and the spectral gap depend on the side length $L$ of the system, though we have kept this dependency implicit to lighten the notation. We say that the system exhibits \emph{cutoff} if in the limit where $L$ tends to infinity, the asymptotic behavior of $\textrm{t}_{\textsc{mix}}(\epsilon)$ does not depend on the precision $\epsilon$ anymore: 
\begin{equation}
    \forall \epsilon \in (0,1)\; \text{fixed}, \dfrac{\textrm{t}_{\textsc{mix}}(\epsilon)}{\textrm{t}_{\textsc{mix}}(1/4)} \xrightarrow[L \to \infty]{} 1.
\end{equation}
We refer the readers to the works \cite{Aldous1986, Diaconis1996, Diaconis1981} for the discovery of the cutoff phenomenon and to the book \cite{Levin2017} for an introduction to the subject. 

\paragraph{Conjectured universal behaviors.}
A stochastic spin system, with an appropriate notion of temperature, under an appropriate time scale, is conjectured to exhibit the following behaviors. 
\begin{itemize}
    \item At high temperatures, the system mixes fast. The inverse-spectral-gap $\gap^{-1}$ is $\Ocal{1}$, the mixing times are logarithmic in the size of the lattice, and cutoff occurs.
    \item At critical temperatures, the inverse gap is (sub)-polynomial in the size of the lattice, and so are the mixing times.
    \item At low temperatures, the inverse gap is exponential in the size of the lattice, and so is the mixing time, and there is no cutoff.
\end{itemize}
See, for example, \cite{Hohenberg1977, Lauritsen1993a, Wang1997} for the conjectures.
\paragraph{Literature.} The Glauber dynamics is arguably the most studied interacting particle system in the literature. Tons of work have been devoted to verifying the conjectured behaviors above for it; see Subsection 1.1 in \cite{Lubetzky2012} for a brief review of the development of the subject. The readers can look at \cite{Aizenman1987, Dobrushin1987,Holley1991,Holley1987,Holley1989,Lu1993,Lubetzky2016,Martinelli1994,Martinelli1994a,Martinelli1994b,Stroock1992,Stroock1992a,Stroock1992b,Zegarlinski1990,Zegarlinski1992} for results in the high-temperature regime, at \cite{Ding2009, Ding2010, Lubetzky2012,Onsager1944} for results in the critical-temperature regime, and at \cite{Chayes1987, Ioffe1995, Martinelli1994, Schonmann1987, Thomas1989} for results in the low-temperature regime. We especially mention here the work \cite{Lubetzky2016}, which proves cutoff in the full high-temperature regime in any dimension. The proof in \cite{Lubetzky2016} involves a new framework called \emph{information percolation}, see also \cite{Ganguly2020, Lubetzky2015, Lubetzky2017, Lubetzky2021}. For background on Glauber dynamics, see the lecture notes \cite{Martinelli1999}, or Chapter 15 in \cite{Levin2017} for a quick introduction. In particular, Theorem $15.4$ in \cite{Levin2017} presents a very simple example that illustrates the difference in mixing behaviors at low and high temperatures.  

Besides Glauber dynamics, the Exclusion Process is also an emblematic interacting particle system. It describes the relaxation of a gas of interacting particles to equilibrium. The convergence is usually studied from two points of view: the macroscopic evolution of the density of particles, which is the study of hydrodynamic limits, or the microscopic evolution of the law of the particles and its total variation distance to equilibrium, which is the study of mixing times. The readers can see \cite{Kipnis1999} for an introduction to the study of hydrodynamic limits and \cite{Levin2017} for an introduction to the study of mixing times. For the results on the hydrodynamic limit of SSEP, see, for example, \cite{Bertini2003, Goncalves2020, Goncalves2021,Lee1998, Kipnis1989}. For the results on the mixing times, see \cite{Diaconis1987,Lacoin2016a,Lacoin2016,Lacoin2017, Lacoin2011,Morris2006exclusion, Oliveira2013, Wilson2004}.      

We mention here that cutoff for SSEP has been proved in relatively few cases: for the complete graph in \cite{Lacoin2011}, for dimension one in \cite{Lacoin2016a, Lacoin2016, Lacoin2017}, see also \cite{Goncalves2021, tran2022cutoff} for nonconservative variants. The only result available in general geometries for a nonconservative variant is obtained in \cite{Salez2023}. However, it is restricted to the case where the system is reversible \wrt a product of \iid Bernoulli. 

As mentioned before, the Glauber-Exclusion process was introduced in \cite{DeMasi1986} and also considered in \cite{Kipnis1989}. To the best of our knowledge, the only works on its mixing times, also the most related works to our paper, are \cite{Tanaka2022} and \cite{Tsunoda2022}. In \cite{Tsunoda2022}, Tsunoda studies the model in dimension one and proves exponentially slow mixing when the primitive of the reaction function (the potential function) has two or more local minima, which implies that the system is at low temperature. In \cite{Tanaka2022}, Tanaka and Tsunoda prove the upper bound  $\Ocal{\log L}$ on the mixing times for any dimension under the condition that $\max_{\rho \in [-1,1]} R'(\rho) < 0$, which subsequently implies that the system is at high temperature. These two works together show a phase transition of the mixing times for the model in Example \ref{ex:demasi} between low and high temperatures, consistent with the conjectured behaviors. In particular, we expect cutoff to occur in the high-temperature regime. However, to establish cutoff, not only do we have to prove that the upper and lower bounds are of the same magnitude $\log L$, but we also need to show that the pre-factors in front of $\log L$ in the two bounds are the same. Surprisingly, not even a lower bound of matching magnitude $\log L$ is available, even though it is often ``the easy part" in proving cutoff. It was not yet known whether the upper bound in \cite{Tanaka2022} is sharp either.

\paragraph{Our contribution.} We sharpen the upper bound in \cite{Tanaka2022} and prove a matching lower bound, therefore showing cutoff for the full high-temperature regime for \emph{any} attractive rate function $c$, in dimensions $1$ and $2$. Furthermore, for any dimension, in the full high-temperature regime, we prove that the inverse gap is of order $1$ and provide lower and upper bounds of magnitude $\log L$ on the mixing times. Our proof involves a new interpretation of the attractiveness, the information percolation framework introduced by Lubetzky and Sly in \cite{Lubetzky2016}, anti-concentration of simple random walk on the lattice, and what we call \emph{the excursion coupling}. We hope that this approach can find new applications in the future.  
\section{Results}
The first main result of the paper is the following.
\begin{theorem}[Cutoff in dimensions $1$ and $2$]\label{thm:cutoff}
    Let $\epsilon \in (0,1)$ be fixed. For $d \in \{1,2\}$, there exists a constant $\kappa$ such that
    \begin{equation}
        \dfrac{\log |\latticedimensiond|}{2|R'(\rho_*)|} -\kappa \log\log L \leq \tmix \leq \dfrac{\log |\latticedimensiond|}{2|R'(\rho_*)|} + \kappa \log\log L.
    \end{equation}
\end{theorem}
\begin{remark}
    In our setup, the dynamics is not reversible, and no explicit formula for the invariant measure $\pi$ is available, except when $c$ is constant. In fact, in \cite{Gabrielli1997}, Gabrielli \etal\, showed that the system is reversible if and only if the function $c(\cdot)$ is of the following form
    \begin{equation}\label{eq:reversibility}
        c(x) = (a_1 + a_2x(0)) h(x),
    \end{equation}
    where $h(x)$ is independent of $x(0)$, and in this case $\pi$ is a product of \iid Rademacher. To the best of our knowledge, it is also the only case that an explicit formula for $\pi$ is available. However, one can show that if a function $c(\cdot)$ satisfies the reversibility condition \eqref{eq:reversibility}, then it is attractive if and only if it is a constant function.  
\end{remark}
Our second main result is about the spectral gap of the system.
\begin{theorem}[Spectral gap]\label{thm:gap}
    For any dimension $d$, 
    \begin{equation}
        \liminf_{L \to \infty} \gap \geq \dfrac{1}{|R(\rho_*)|}.
    \end{equation}
\end{theorem}
Theorem \ref{thm:cutoff} and Theorem \ref{thm:gap} together confirm the conjectured behaviors of the system in the full high-temperature regime in dimensions $1$ and $2$. Furthermore, they connect the mixing behaviors of the Glauber-Exclusion process with the behaviors of the solutions of the ODE \eqref{eq:ODE}.

Our third result investigates the mixing times in higher dimensions.
\begin{theorem}[Pre-cutoff in higher dimensions]\label{thm:precutoff}
    Let $\epsilon \in (0,1)$ be fixed. For any $d \geq 3$, there exists a constant $\kappa$ such that
    \begin{equation}
        \dfrac{1}{d}\dfrac{\log |\latticedimensiond|}{|R'(\rho_*)|} - \kappa \leq \tmix \leq \dfrac{\log |\latticedimensiond|}{|R'(\rho_*)|} + \kappa \sqrt{\log L}.
    \end{equation}
\end{theorem}
Theorem \ref{thm:precutoff} says that the magnitude of the mixing times in all dimensions is $\log L$. Furthermore, the pre-factor in front of $\log L$ is bounded between two constants independent of the precision $\epsilon$. This phenomenon is known as \emph{pre-cutoff}; see Chapter 18 in \cite{Levin2017} for more details. Establishing cutoff for the process in dimensions $d \geq 3$ remains an open problem. We believe that neither the lower nor the upper bounds in Theorem \ref{thm:precutoff} is optimal. Furthermore, our proof of the upper bound in dimensions $1$ and $2$ relies on the fact that the simple random walks in dimensions $1$ or $2$ is recurrent, which is no longer valid for higher dimensions. We do not know if this is just a coincidence.  
The Interchange Process, a Markov process closely related to the Exclusion process, is conjectured to behave somehow differently in dimensions $d \geq 3$ compared to $d \in \{1,2\}$ in \cite{Toth1993}. This fact is verified for all dimensions $d \geq 5$ in \cite{elboim2023infinite}. 
We do not know whether this phase transition in infinite volume is connected to a change of mixing behavior in finite volume.

Nevertheless, the upper bound in Theorem \ref{thm:precutoff} is already better than that in \cite{Tanaka2022}. In \cite{Tanaka2022}, the authors impose the assumption that  the derivative $R'$ of $R$ is uniformly negative, \ie $\max_{\rho \in [-1,1]} R'(\rho) < 0,$ 
and prove the upper bound $\dfrac{\log |\latticedimensiond|}{\left|\max_{\rho \in [-1,1]} R'(\rho)\right|}$ on the mixing time. Hence, our assumption is weaker, and the pre-factor in front of $\log L$ is also sharper.
\begin{example}
    Let $\theta$ be a strictly positive number. In dimension one, we consider the local flip-rate function $c(\cdot): \{-1,1\}^{\{-1,0,1\}} \to \Rbb_+$ defined by 
    \begin{equation}
        \forall x \in \{-1,1\}^{\{-1,0,1\}},\; c(x):= \theta + 2 \times \indicator{x(0) = 1,\; x(1) = -1}.
    \end{equation}
    The corresponding global flip-rate function $\hat{c}$ is given by
    \begin{equation}
        \hat{c}(u,x):= \theta + 2 \times \indicator{x(u) = 1,\; x(u+1) = -1}.
    \end{equation}
    It is not hard to see that the function $c(\cdot)$ satisfies Hypothesis \ref{hyp:attractive}. The corresponding reaction function is 
    \begin{equation}
        R(\rho) = \rho^2 -2\theta \rho - 1.
    \end{equation}
    By direct computation, one can show that $R(\rho)$ has a unique root in the interval $[-1,1]$: 
    \begin{equation*}
        \rho_* = \theta - \sqrt{\theta^2 +1},
    \end{equation*}
    and moreover, 
    \begin{equation*}
        R'(\rho_*) < 0.
    \end{equation*}
    However, if $0 < \theta < 1$, then $\forall \rho \in (\theta,1],\, R'(\rho) > 0$. This means that $R(\rho)$ satisfies our conditions but does not satisfy the condition in \cite{Tanaka2022}.
\end{example}
\paragraph{Structure of the proof.} In Section \ref{sec:monotonicity}, we interpret the attractiveness in a more intuitive way. We also discuss some properties of monotone boolean functions, which will be used later in our proof. In Section \ref{sec:coupling}, we construct the dual coupling, which will subsequently be used in Section \ref{sec:application_coupling} to prove Theorem \ref{thm:gap}, Theorem \ref{thm:precutoff}, and the lower bound in Theorem \ref{thm:cutoff}. Finally, in Section \ref{sec:ub}, we prove the upper bound in Theorem \ref{thm:cutoff}. 
For simplicity, we will assume that $d = 1$. Nevertheless, many ingredients are valid for other dimensions, and we will comment on how to adapt the proof to higher dimensions when necessary.
\paragraph{Acknowledgement.} The author deeply thanks Justin Salez for suggesting the problem, for numerous useful discussions, and for thoroughly reading the drafts.  

\section{Monotone boolean functions and attractiveness}\label{sec:monotonicity}
This section aims to discuss some properties of monotone boolean functions and to prove Proposition \ref{prop:glauber_interpretation}, which allows us to interpret the attractiveness more intuitively.
\subsection{Monotone boolean functions}
\begin{definition}[Boolean function]\label{def:boolean} 
A \emph{boolean function} is a function from some hypercube $\{-1,1\}^n,\, n\in \Zbb_+$, into $\{-1,1\}$. A boolean function $f$ is increasing if $x \leq y \Rightarrow f(x) \leq f(y)$, for any $x, y \in \{-1,1\}^n$.    
\end{definition}
\begin{definition}[Pivotal set] \label{def:pivotal}
    For an increasing boolean function $f$ on $\{-1,1\}^n$, its \emph{pivotal set} is given by
    \begin{equation}
        \piv{f} := \{j \in [n]: \exists x = (x_1, \dots, x_n) \in \{-1,1\}^{n} \text{ such that } f(x^{j, 1}) = 1; \; f(x^{j, -1}) = -1 \}.
    \end{equation}
\end{definition}
In particular, to evaluate the value of an increasing boolean function $f$ on a spin configuration $x$, we do not need to know every coordinate of $x$ but only the coordinates in $\piv{f}$. The following lemma is elementary but very useful in our proof.  
\begin{lemma}[Pivotal sets of monotone boolean functions]\label{lm:pivotal-boolean}
    For any increasing boolean function $f$ on $\{-1,1\}^n$,
            \begin{equation*}
                \piv{f} \neq \O \iff f(1,\dots, 1) = 1 \text{ and } f(-1, \dots, -1) = -1.
            \end{equation*}
    In particular, for any increasing boolean function $f$ on $\{-1,1\}$,
    \begin{equation*}
        |\piv{f}| = \dfrac{f(1) - f(-1)}{2}.
    \end{equation*}
\end{lemma}

\subsection{Interpretation of attractiveness}
We canonically identify the functions on $\Xcal$ which depend only on the coordinates in $B(0,m)$ with the functions on $\{-1,1\}^{B(0,m)}$. We now give a more intuitive interpretation of the attractiveness.  
\begin{proposition}[Interpretation of attractiveness]\label{prop:glauber_interpretation}
    There exist $q \in \Zbb_+$ and increasing boolean functions $f_1, \dots, f_q$ on $\{-1,1\}^{B(0,m)}$ and positive numbers $\lambda_1, \dots, \lambda_q$, such that
    \begin{equation}\label{eql3}
        \forall x \in \{-1,1\}^{B(0,m)},\; c(x) = \sum_{i=1}^q \lambda_i \indicator{f_i(x) = -x(0)}.
    \end{equation}
    In particular, for any function $\varphi$ on $\Xcal$,
    \begin{align}\label{eql4}
        \forall x \in \Xcal,\; \generatorglauber \varphi(x) = \sum_{u \in \lattice} \sum_{i = 1}^q \lambda_i \left(\varphi(x^{u, f_i(x_{u+ \cdot})}) - \varphi(x)\right).
    \end{align}
    Moreover, we can choose $f_1 \equiv 1,\; f_2 \equiv -1$ and $\lambda_1, \lambda_2$ strictly positive.
\end{proposition}
\paragraph{Interpretation.} Proposition \ref{prop:glauber_interpretation} is of independent interest. It implies that the technical conditions \eqref{eql1}, \eqref{eql2} for the attractiveness of the Glauber dynamics can be replaced by a more intuitive interpretation: for each site $u$, $x(u)$ is replaced by $f_i(x_{u + \cdot})$ at rate $\lambda_i$, $1 \leq i\leq q$. The monotonicity of $(f_i)_{1\leq i \leq q}$ clearly implies that the system is attractive. In fact, for two processes $(Y^1_t)_{t\geq 0}$ and $(Y^2_t)_{t\geq 0}$ whose generator is $\Lcal_G$, one can construct a coupling of them by updating the same sites simultaneously using the same functions. This coupling preserves the order thanks to the monotonicity of $(f_i)_{1\leq i \leq q}$. Besides, the updates given by deterministic functions $f_1 \equiv 1$ or $f_2 \equiv -1$ can be carried out without looking at the spins of any sites. We say that these updates are \emph{oblivious}. We will see that oblivious updates play a crucial role in our proof.    
\begin{demasi}
    For the flip-rate function considered in Example \ref{ex:demasi}, one can verify that the following functions and rates satisfy Proposition \ref{prop:glauber_interpretation}. For any $x = (x(-1), x(0),x(1)) \in \{-1,1\}^{\{-1,0,1\}}$, 
    \begin{align*}
        f_1(x) &= 1,\; &\lambda_1 &= (1-\gamma)^2,\\
        f_2(x) &= -1,\; &\lambda_2 &= (1-\gamma)^2,\\
        f_3(x) &= \text{sgn}(x(-1) + x(0) + x(1)),\; &\lambda_3 &= 4\gamma^2,\\
        f_4(x) &= x(-1),\; &\lambda_4 &= 2(\gamma - \gamma^2),\\
        f_5(x) &= x(1),\; &\lambda_5 &= 2(\gamma - \gamma^2).
    \end{align*}
    Here $\text{sgn}$ denotes the sign function: $\text{sgn}(s) = \begin{cases}
        1 &\text{if $s >0$},\\
        0 &\text{if $s =0$},\\
        -1 &\text{if $s <0$}.
    \end{cases}$
\end{demasi}

We first need the following lemma to prove Proposition \ref{prop:glauber_interpretation}.
\begin{lemma}[Decomposition of monotone functions]\label{lm:decomposition} 
    Let $\Omega$ be a finite set with a partial order $\prec$. Let $f: \Omega \to \Rbb_+$ be an increasing function, meaning that $x \prec y \Rightarrow f(x) \leq f(y)$. Then there exist $q \in \Zbb_+$ and increasing (not boolean) functions $f_1, \dots, f_q: \Omega \to \{0,1\}$ and positive numbers $\lambda_1, \dots, \lambda_q$ such that
    \begin{align}
        f(\cdot) = \sum_{i=1}^q \lambda_i f_i(\cdot).
    \end{align}
    Moreover, we can choose $f_1(\cdot) = \indicator{f(\cdot) > 0}$ and $\lambda_1 = \indicator{f \not \equiv 0}\min\limits_{\{f(\cdot) > 0\}} f $.
\end{lemma}
Lemma \ref{lm:decomposition} can be proved by induction on the size of the support of the function $f$. We omit the proof here. Now we prove Proposition \ref{prop:glauber_interpretation}.

\begin{proof}[Proof of Proposition \ref{prop:glauber_interpretation}]
    First we prove that \eqref{eql3} implies \eqref{eql4}. Indeed, if \eqref{eql3} is true, then by definition of $\hat{c}(\cdot)$, for any $u$ in $\lattice$ and $x \in \Xcal$,
    \begin{equation}\label{eq18} 
        \hat{c}(u,x) = \sum_{i=1}^q \lambda_i\indicator{f_i(x_{u+\cdot})= -x(u)}.
    \end{equation}
    Therefore, for any function $\varphi$ on $\Xcal$,
    \begin{align*}
        \generatorglauber \varphi(x) &= \sum_{u \in \lattice} \hat{c}(u,x)(\varphi(x^{u, -x(u)}) - \varphi(x))\\
        &= \sum_{u\in \lattice} \sum_{i=1}^q \lambda_i \indicator{f_i(x_{u + \cdot}) = -x(u)}(\varphi(x^{u, -x(u)}) - \varphi(x))\\
        &= \sum_{u\in \lattice} \sum_{i=1}^q \lambda_i \indicator{f_i(x_{u + \cdot}) = -x(u)}(\varphi(x^{u,f_i(x_{u + \cdot})}) - \varphi(x))\\
        &= \sum_{u\in \lattice} \sum_{i=1}^q \lambda_i (\varphi(x^{u,f_i(x_{u + \cdot})}) - \varphi(x)),
    \end{align*}
    where the second equality is due to \eqref{eq18}, the third equality is because $x^{u,f_i(x_{u + \cdot})} = x^{u,-x(u)}$ whenever $f_i(x_{u +\cdot}) = -x(u)$, and the last equality is because $\varphi(x^{u,f_i(x_{u + \cdot})}) - \varphi(x) = 0$ whenever $f_i(x_{u + \cdot}) = x(u)$. It remains to show \eqref{eql3}.
    
    Let $n:= \absolutevalue{B(0,m)} -1.$ We identify the set ${B(0,m)}$ with $\{0, 1, \dots, n\}$ such that the site $0$ is identified with number $0$.
    Then \eqref{eql3} is reformulated as follows: there exist increasing boolean functions $\ftilde_1, \dots, \ftilde_q$ on $\{-1,1\}^{n+1}$ and positive numbers $\lambda_1, \dots, \lambda_q$ such that, for any $\xi_0, \dots, \xi_n \in \{-1,1\}$,
    \begin{equation}\label{eq:attractive_reformulation}
        c(\xi_0, \dots, \xi_n) = \sum_{i = 1}^q \lambda_i \indicator{\ftilde_i(\xi_0, \dots, \xi_n) = -\xi_0}.
    \end{equation}
    Note that Hypothesis \ref{hyp:attractive} implies that $c$ is a positive function on $\{-1,1\}^{n+1}$ such that $c(-1, \xi_1, \dots, \xi_n)$ is increasing and $c(1, \xi_1, \dots, \xi_n)$ is decreasing as functions of $(\xi_1, \dots, \xi_n)$. 
    
    We can now apply Lemma \ref{lm:decomposition} to the function $c(-1, \cdot)$ to conclude that there exist increasing functions $f_1, \dots, f_{q_1}$ from $\{-1,1\}^n$ to $\{0,1\}$ and positive numbers $\lambda_1, \dots, \lambda_{q_1}$, for some $q_1 \in \Zbb_+$, such that, for any $\xi_1, \dots, \xi_n \in \{-1,1\}$, 
    \begin{equation}
        c(-1, \xi_1, \dots, \xi_n) = \sum_{i=1}^{q_1} \lambda_i f_i(\xi_1, \dots, \xi_n).
    \end{equation}
    Moreover, we can choose 
    \begin{align*}
        f_1(\xi_1,\dots, \xi_n) &= \indicator{c(-1, \xi_1, \dots, \xi_n) > 0} \equiv 1,\\
        \lambda_1 &= \min_{\xi_1, \dots, \xi_n \in \{-1,1\}} c(-1, \xi_1, \dots, \xi_n) > 0.             
    \end{align*}
    Here, we have used the fact that $c(\cdot)$ only takes strictly positive values.
    We then define the boolean functions $(\tilde{f}_i)_{1 \leq i \leq q_1}$ on $\{-1,1\}^{n+1}$ by 
    \begin{equation*}
        \tilde{f}_i (\xi_0, \dots \xi_n) = \begin{cases}
            1 &\text{ if } \xi_0 = 1,\\
            2f_i(\xi_1, \dots, \xi_n) - 1 &\text{ if } \xi_0 = -1.
        \end{cases}
    \end{equation*}
    It is easy to see that $\tilde{f}_i$ is increasing for any $i$. 
    Note also that $\tilde{f}_1 \equiv 1$.
    
    Let $\varphi_1:\{-1,1\}^{n + 1} \to \Rbb_+$ be defined by, for any $\xi_0, \dots, \xi_n \in \{-1,1\}$,
    \begin{equation*}
        \varphi_1(\xi_0, \dots, \xi_n) = \sum_{i=1}^{q_1} \lambda_i \indicator{\tilde{f}_i(\xi_0, \dots, \xi_n) = -\xi_0}.
    \end{equation*}
    Noting that $\indicator{\tilde{f}_i(1, \xi_1 \dots, \xi_n) = -1} = 0$ and $\indicator{\tilde{f}_i(-1,\xi_1, \dots, \xi_n) = 1} = \indicator{f_i(\xi_1, \dots, \xi_n) = 1} = f_i(\xi_1, \dots, \xi_n)$, we conclude that 
    \begin{align*}
        \varphi_1(1, \xi_1, \dots, \xi_n) &= 0,\\
        \varphi_1(-1, \xi_1, \dots, \xi_n) &= c(-1, \xi_1, \dots, \xi_n).
    \end{align*}
    Similarly $c(1, \xi_1, \dots, \xi_n)$ is a positive decreasing function of $(\xi_1, \dots, \xi_n)$. We can still apply Lemma \ref{lm:decomposition}, by replacing the relation $\leq$ by $\geq$, to conclude that there exist decreasing functions (increasing functions \wrt relation $\geq$) $g_1, \dots, g_{q_2}$ and positive numbers $\lambda'_1, \dots, \lambda'_{q_2}$ such that 
    \begin{equation}
        c(1, \xi_1, \dots, \xi_n) = \sum_{i=1}^{q_2} \lambda'_i g_i(\xi_1, \dots, \xi_n).
    \end{equation}
    By a similar argument, we claim that the boolean functions $(\tilde{f}_{q_1+i})_{1 \leq i \leq q_2}$ defined by 
    \begin{equation*}
        \tilde{f}_{q_1+i} (\xi_0, \dots \xi_n) = \begin{cases}
            1 - 2g_i(\xi_1, \dots, \xi_n) &\text{ if } \xi_0 = 1,\\
            -1 &\text{ if } \xi_0 = -1,
        \end{cases}
    \end{equation*}
    are increasing. Similarly, we can choose $\lambda'_1 > 0$ and $g_1 \equiv 1$. Hence $\tilde{f}_{q_1 + 1} \equiv -1$. 
    Similarly, we define $\varphi_2: \{-1,1\}^{n+1} \to \Rbb_+$ by, for all $\xi_0, \dots, \xi_n \in \{-1,1\}$,
    \begin{equation*}
        \varphi_2(\xi_0, \dots, \xi_n) = \sum_{i=1}^{q_2} \lambda'_i \indicator{\tilde{f}_{q_1+i}(\xi_0, \dots, \xi_n) = -\xi_0}.
    \end{equation*}
    $\varphi_2$ satisfies
        \begin{align*}
        \varphi_2(1, \xi_1, \dots, \xi_n) &= c(1, \xi_1, \dots, \xi_n),\\
        \varphi_2(-1, \xi_1, \dots, \xi_n) &= 0 .
    \end{align*}
    Hence \begin{equation*}
        c = \varphi_1 + \varphi_2.
    \end{equation*}
    We can set $\lambda_{q +i} = \lambda'_i,\; 1\leq i \leq q_2$ to write 
    \begin{equation}
        c(\xi_0, \dots, \xi_n) = \sum_{i=1}^{q_1 + q_2} \lambda_i \indicator{\tilde{f}_i(\xi_0, \dots, \xi_n) = -\xi_0}.
    \end{equation}
    We can renumber these functions and their coefficients such that $\tilde{f}_1 \equiv 1,\, \tilde{f}_2 \equiv -1$. So we have established \eqref{eq:attractive_reformulation}. This finishes our proof.
\end{proof}
\begin{remark}
    We can adapt this proposition to \emph{any} flip rate function satisfying \eqref{eql1}, \eqref{eql2}, on any graph. In this case, each site $u$ has its own update functions $f_{u,1}, \dots f_{u, q_u}$, which are increasing and applied at rate $\lambda_{u, 1}, \dots, \lambda_{u, q_u}$ respectively. In our case, those functions are local and invariant by translation because the function $\hat{c}$ is.
\end{remark}
From now on, we fix a choice of $(f_i)_{1\leq i \leq q}$ and $(\lambda_i)_{1\leq i \leq q}$ that satisfies Proposition \ref{prop:glauber_interpretation}.
For our purpose, we will rewrite the reaction function $R$ and its derivative $R'$ in terms of $(f_i)_{1\leq i \leq q}$ and $(\lambda_i)_{1\leq i \leq q}$. To do this, let $\lambda := \sum_{i = 1}^q \lambda_i$. For any $1 \leq i \leq q$, $j \in B(0,m)$, $\xi_{-m}, \dots, \xi_m \in \{-1,1\}$, let us define 
\begin{equation*}
    \nabla_j f_i(\xi_{-m}, \dots, \xi_m) := f_i(\xi_{-m}, \dots, \xi_{j-1}, 1, \xi_{j + 1}, \dots, \xi_m) - f_i(\xi_{-m}, \dots, \xi_{j-1}, -1, \xi_{j + 1}, \dots, \xi_m).
\end{equation*}

\begin{lemma}[$R$ and $R'$ in terms of $(f_i)_{1\leq i \leq q}$ and $(\lambda_i)_{1\leq i \leq q}$]\label{lm:reaction-term-formula}
    For any $\rho \in [-1,1]$,
    \begin{align}
        R(\rho) &= \left(\sum_{i = 1}^q \lambda_i \esperancewithstartingpoint{\nu_\rho}{f_i}\right) - \lambda \rho, \label{eql5}\\
        R'(\rho) &= \left(\sum_{i = 1}^q  \sum_{j\in B(0,m)}\lambda_i\esperancewithstartingpoint{\nu_\rho}{\dfrac{1}{2} \nabla_jf_i}\right) - \lambda.\label{eql6}
    \end{align}
\end{lemma}
\begin{proof}
    Recall that $R(\rho) = \esperancewithstartingpoint{\nu_\rho}{-2\xi_0 c(\xi)}$.
    Hence, by \eqref{eql3}.
    \begin{align*}
        R(\rho) &= \esperancewithstartingpoint{\nu_\rho}{-2\xi_0 \sum_{i = 1}^q \lambda_i \indicator{f_i(\xi) = -\xi_0}}\\
        &= \sum_{i = 1}^q \lambda_i \esperancewithstartingpoint{\nu_\rho}{-2\xi_0  \indicator{f_i(\xi) = -\xi_0}}\\
        &= \sum_{i = 1}^q \lambda_i  \esperancewithstartingpoint{\nu_\rho}{(f_i(\xi) - \xi_0)  \indicator{f_i(\xi) = -\xi_0}}\\
        &= \sum_{i = 1}^q  \lambda_i  \esperancewithstartingpoint{\nu_\rho}{f_i(\xi) - \xi_0 }\\
        &= \left(\sum_{i = 1}^q  \lambda_i  \esperancewithstartingpoint{\nu_\rho}{f_i(\xi)}\right) - \lambda \rho.
    \end{align*}
    So we have established \eqref{eql5}. We also know that, for any $\rho \in [-1,1]$, 
    \begin{equation*}
        \rade{\rho} = \dfrac{1+\rho}{2} \delta_1 + \dfrac{1-\rho}{2} \delta_{-1},
    \end{equation*}
    and hence 
    \begin{equation*}
        \dfrac{d}{d\rho}\rade{\rho} = \dfrac{1}{2} \delta_1 - \dfrac{1}{2} \delta_{-1}.
    \end{equation*}
    Therefore, for $\nu_\rho = \rade{\rho}^{\otimes B(0,m)}$, 
    \begin{equation*}
        \dfrac{d}{d\rho} \nu_\rho = \sum_{j \in B(0,m)} \rade{\rho}^{\otimes \{-m, \dots, j-1\}} \otimes \left(\dfrac{1}{2} \delta_1 - \dfrac{1}{2} \delta_{-1}\right) \otimes \rade{\rho}^{\otimes \{j+1, \dots, m\}}.
    \end{equation*}
    This equality and \eqref{eql5} lead to \eqref{eql6}.
\end{proof}

\section{Dual coupling}\label{sec:coupling}
\subsection{Graphical Construction}\label{sec:description}
We call $\lattice \times \Rbb_+$ the \emph{space-time slab}, where each slice $\lattice \times \{t\}$ is equipped with the graph structure of $\lattice$. We identify $\lattice$ with the subset $\{0, \dots, L-1\}$ of $\Zbb_+$.

We introduce a notation essential to our proof.

\paragraph{A collection of marks.} A collection of marks $\Ccal$ is a pair $(\Ccalexclusion, \Ccalglauber)$, where $\Ccalexclusion$ and $\Ccalglauber$ are subsets of $\lattice \times \Rbb_+$ and $[q] \times \lattice \times \Rbb_+$, respectively, that satisfy the following conditions.
\begin{itemize}
    \item There is at most one mark at a time:
    \begin{align*}
        \forall t > 0,\, \left|\Ccalexclusion \cap \lattice \times \{t\}\right| +  \left|\Ccalglauber \cap [q] \times \lattice \times \{t\}\right| \leq 1. 
    \end{align*}
    \item $\Ccal$ is locally finite in time: for any $0 < t_1< t_2$,
    \begin{align*}
        \left|\Ccalexclusion \cap \lattice \times [t_1, t_2]\right| +  \left|\Ccalglauber \cap [q] \times \lattice \times [t_1, t_2]\right| < \infty.
    \end{align*}
\end{itemize}
In other words, $\Ccal$ can be identified with a counting measure on $(\lattice \times \Rbb_+) \cup ([q] \times \lattice \times \Rbb_+)$ whose projection onto $\Rbb_+$ has no multiple point and is locally finite. We call an element $(u,t) \in \Ccalexclusion$ the exclusion mark on the edge $(u, u+1)$ at time $t$ and an element $(i, u, t) \in \Ccalglauber$ the Glauber mark of type $i$ on site $u$ at time $t$.

\paragraph{Effect of a collection of marks.} Given a collection of marks $\Ccal$ and a configuration $x_0 \in \Xcal$. The process $(X^{x_0}_t)_{t\geq 0}$ is defined recursively according to the following rules.
\begin{itemize}
    \item $X^{x_0}_0 = x_0$ ($x_0$ is the initial configuration).
    \item The process $(X^{x_0}_t)_{t\geq 0}$ is piecewise constant and can only jump when a mark appears.
    \item When we see an exclusion mark on the edge $(u, u+1)$, we make the transition $x \mapsto x^{u \leftrightarrow u+1}$ (exchange the spins at sites $u$ and $u+1$).
    \item When we see a Glauber mark of type $i$ at site $u$, we make the transition $x \mapsto x^{u, f_i(x_{u + \cdot})}$ (update site $u$ using the function $f_i$).
\end{itemize}
Roughly speaking, the marks tell us how the process $(X^{x_0}_t)_{t\geq 0}$ is updated. We have kept the dependency on $\Ccal$ implicit to lighten the notation.
The Glauber-Exclusion process associated with the generator $\generator_{GE}$ can be constructed as follows.
\paragraph{Graphical Construction 1.} Let the background process
\begin{equation}
    \Xi = \left((\xiexclusion_u)_{0\leq u \leq \longeur-1}, (\xiglauber_{i,u})_{1 \leq i\leq q,\, 0 \leq u \leq \longeur-1}\right)
\end{equation}
be defined as follows. $(\xiexclusion_u)_{0\leq u \leq \longeur-1}, (\xiglauber_{i,u})_{1 \leq i\leq q,\, 0 \leq u \leq \longeur-1}$ are independent homogeneous Poisson processes, and 
\begin{itemize}
    \item $\xiexclusion_u$ is of intensity $\longeur^2$,\, $0 \leq u\leq \longeur-1$,
    \item $\xiglauber_{i,u}$ is of intensity $\lambda_i$,\, $1\leq i\leq q,\, 0\leq u\leq \longeur-1$.
\end{itemize}
Almost surely, a realization of $\Xi$ naturally defines a collection of marks $\Ccal$ as follows. For any $(i, u, t) \in [q] \times \lattice \times \Rbb_+$,
\begin{itemize}
    \item $(u,t) \in \Ccalexclusion$ iff the process $\xiexclusion_u$ jumps at time $t$,
    \item $(i, u, t)\in \Ccalglauber$ iff the process $\xiglauber_{i,u}$ jumps at time $t$.
\end{itemize} 
Then the process $(X^{x_0}_t)_{t\geq 0}$ constructed by the rules above is a Markov process with generator $\Lcal^{GE}$ starting from configuration $x_0$. An illustration is given in Figure \ref{fig:history}, where each exclusion mark is naturally drawn as a left-right arrow on the corresponding edge, and each Glauber mark is drawn as $\times$ on the corresponding site, with the corresponding update function next to it. 
\begin{remark}
    In the future, we may construct the (random) collection of marks $\Ccal$ differently, but as long as it has the same distribution as the one given by $\Xi$, the process $(X^{x_0}_t)_{t\geq 0}$ constructed by the above rules is still a Markov process associated with generator $\Lcal^{GE}$ starting at $x_0$. 
\end{remark}

\paragraph{The alphabet.}Let $\Acal := \bigcup\limits_{i =0}^\infty (\Zbb_+ \times B(0,m)^i)$ be the set of all words beginning with an element in $\Zbb_+$ followed by elements in $B(0,m)$. We define a lexicographical order on $\Acal$ using the usual orders on $\Zbb_+$ and $B(0,m)$. For any word $\wsf \in \Acal$, we call the elements of $\{\wsf\} \times B(0, m)$ the children of $\wsf$. The alphabet $\Acal$ will be used to label the branching processes that naturally arise when we study the \emph{update history} of our process.

\paragraph{Update history.} Consider a collection of marks $\Ccal$ and the resulting process $(X^x_t)_{t\geq 0}$ for some $x \in \Xcal$. Let $t \geq 0$, and let $E \subset \lattice$ be a set whose spins we want to know at time $t$. A natural way to reveal the spins of $E$ at time $t$ is to go backward in time, tracing the history of the information involved. At each site of $E$ at time $t$, we put a particle labeled by the corresponding element in $\Acal$ (we have $E \subset \lattice \subset \Zbb_+ \subset\Acal$). The history of $E$, going backward in time, is as follows.
\begin{enumerate}
    \item Every time we see an exclusion mark, the particles at the endpoints of the corresponding edge, if there are any, jump to the other endpoints.
    \item Every time we see a Glauber mark, say at site $u$, if there is a particle $\wsf$ there, we remove $\wsf$, and at each site of $B(u,m)$, we add a new particle. We canonically label the newborn particles by the children of $\wsf$. If a particle is born on a site already occupied by some other particles, all of them will move together as one particle.
\end{enumerate}
We follow all the particles backward until time $0$. The \emph{update history} of the set $E$ at time $t$ is the sub-collection of marks $\Ccal_{his}(E)$ consisting of all the marks (exclusion and Glauber) we meet on the trajectories of all the particles described above. For our purpose, it is more convenient to consider $\Ccal^{bw}_{his}(E)$, the version of $\Ccal_{his}(E)$ when we reverse the time interval $[0,t]$: $\Ccal^{bw}_{his}(E)$ is the set of marks obtained from $\Ccal_{his}(E)$ by replacing the time $s$ in each mark by $t-s$.

We define a spin configuration $\sfrak$ on the set of all particles appearing as follows.
\begin{itemize}
    \item For any particle $\wsf$ reaching time $0$, say at site $u$, 
    \begin{align*}
        \sfrak_\wsf := x(u).
    \end{align*}
    \item For any particle $\wsf$ not reaching time $0$, let $i$ be the type of the Glauber mark removing $\wsf$. Then 
    \begin{align*}
        \sfrak_\wsf := f_i((\sfrak_{\wsf'})_{\wsf' \in \{\wsf\}\times B(0,m)} ).
    \end{align*}
\end{itemize}
These rules allow us to define recursively the spins of all particles that appear. In particular, by construction, 
\begin{align*}
    \forall u \in E,\, \sfrak_u = X^x_t(u),
\end{align*}
where $\sfrak_u$ is the spin of the particle labeled $u$.
\begin{figure}
    \centering
    \includegraphics[width = \textwidth, height = 9cm]{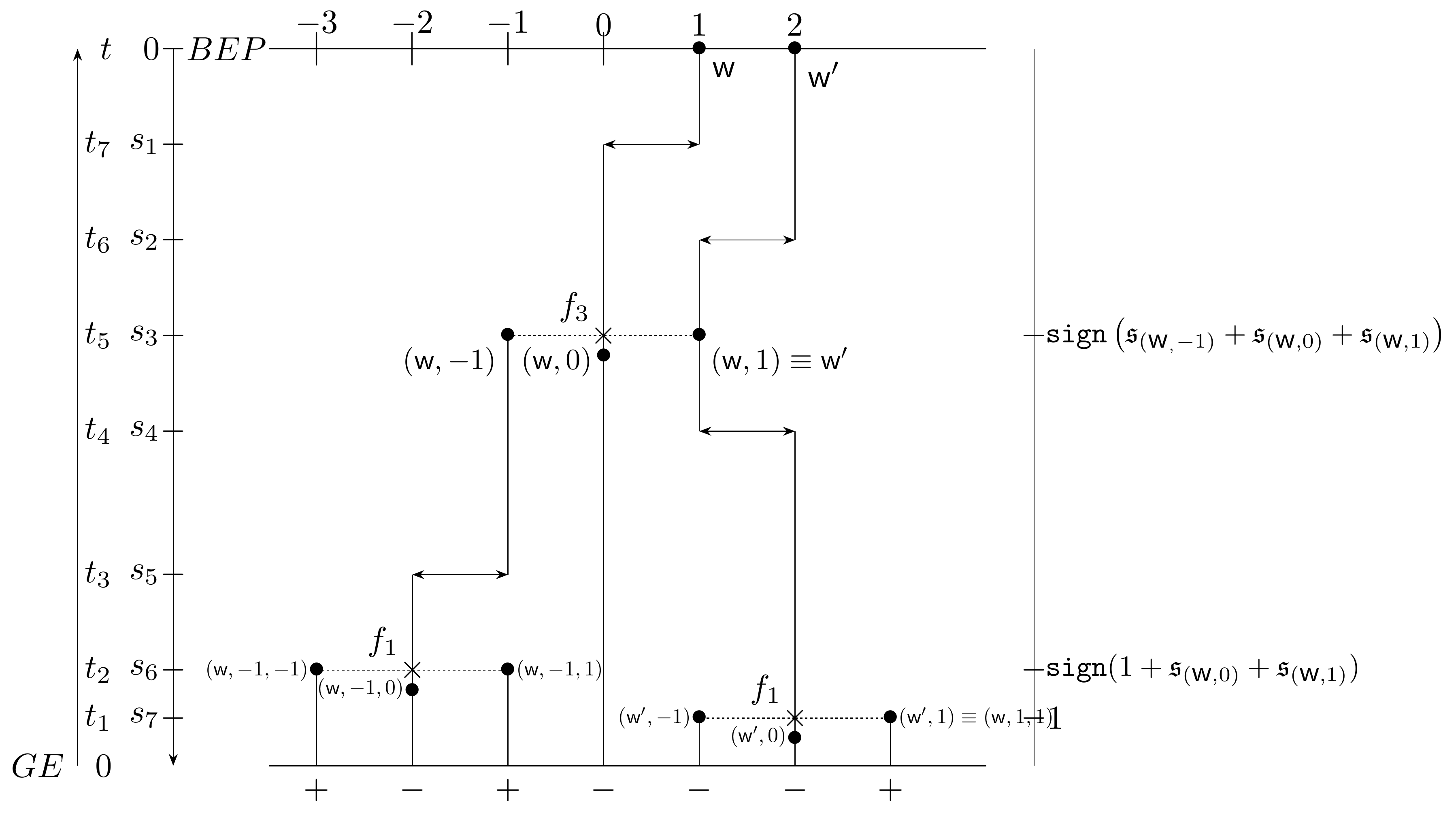}
    \caption{History of Glauber-Exclusion dynamics.}
    \label{fig:history}
\end{figure}


\subsection{The branching exclusion process (BEP)}
We fix a subset $E$ of $\lattice$ in this subsection. 
We claim that the history itself, viewed backward in time, is a Markov process, called the \emph{branching exclusion process} (BEP), defined as follows.

\paragraph{Definition of the BEP.} We consider the following way of putting marks and particles in the space-time slab. We go forward in time.

\begin{enumerate}
    \item At time $0$, at each site of $E$, we put a particle labeled by the corresponding element of $\Acal$. Each particle is associated with $q$ Poisson clocks, called the Glauber clocks of type $1, \dots, q$, of intensity $\lambda_1, \dots, \lambda_q$, respectively. All the Glauber clocks are independent.
    \item Each edge of the lattice is associated with a Poisson clock of intensity $L^2$, called the exclusion clock. The exclusion clocks are independent and independent of the Glauber clocks. Each time an exclusion clock rings, we put an exclusion mark on the corresponding edge.   
    \item The particles follow the exclusion marks until one Glauber clock rings, say the clock of type $i$ of particle $\wsf$ at time $s$. Suppose that $\wsf$ is at site $u$ at time $s-$. Then we put a Glauber mark of type $i$ on site $u$ at time $s$. After that, we remove particle $\wsf$ and add new particles to all sites in $B(u,m)$, canonically labeled by the children of $\wsf$. We think of this as the particle $\wsf$ splits into its neighbor sites. Each particle born on a non-occupied site is given $q$ independent Glauber clocks, independent of all exclusion clocks and all Glauber clocks of other particles. Each particle born on an occupied site ``moves together" with the particles already there. The particles moving together are called a group of particles, and they are associated with the same $q$ Glauber clocks. All the particles then continue to evolve with similar rules. 
\end{enumerate}
The BEP starting from $E$ consists of two parts $(\Psi_{BEP}(E), \Ccal_{BEP}(E))$, where
\begin{itemize}
    \item $\Psi_{BEP}(E)$ is the collection of trajectories of all the particles, \ie the map associating each point $(u,t)$ of the space-time slab with the set of particles occupying the site $u$ at time $t$.
    \item $\Ccal_{BEP}(E)$ is the set consisting of all the marks on the trajectories of all the particles constructed above.
\end{itemize}
\begin{remark}
    Knowing $E$, the pair $(\Psi_{BEP}(E), \Ccal_{BEP}(E))$ is determined by $\Ccal_{BEP}(E)$.
\end{remark}

We say that a particle $\wsf$ is \emph{alive} at time $t$ if it was born before time $t$ and has not been removed by time $t$. Let $W_t$ be the set consisting of the particles with the smallest labels in each group of particles alive at time $t$, for any $t \in \Rbb_+$. A spin configuration on $W_t$ is identified with a spin configuration on the set of all particles alive at time $t$, where the particles in a group are given the same spin. 
\paragraph{Spins of particles in BEP.}Each spin configuration $\sfrak$ on $W_t$ is extended to $\bigcup_{0\leq s \leq t} W_s$ by the following rules. For any $\wsf \notin W_t$, $\wsf$ must be removed at some point in $[0,t]$ by a Glauber mark, say of type $i$. Then 
\begin{align*}
    \sfrak_\wsf := f_i\left((\sfrak_{\wsf'})_{\wsf' \in \{\wsf\}\times B(0,m)} \right).
\end{align*}
This allows us to extend the spin configuration to all the particles alive at some point in the time interval $[0,t]$.

By the homogeneity of the Poisson processes, we can see that $\Ccal_{BEP}(E)$ and $\Ccal_{his}^{bw}(E)$ have the same distribution. In fact, we can show that their first marks have the same distribution, then conditionally on their first marks, their second marks have the same distribution, etc. This leads to the following result.

\begin{proposition}[History is BEP]\label{prop:history=bep}
    Let $t \in \Rbb_+$, $x\in \Xcal$. Consider the BEP starting from $E$ and the history of $E$ at time $t$ given by a collection of mark generated by $\Xi$. Then the point processes $\Ccal_{BEP}(E)$ and $\Ccal_{his}^{bw}(E)$ have the same distribution. In particular, for some $x\in \Xcal$, if we define the spin configuration $\sfrak$ on $W_t$ by
    \begin{align*}
        \forall \wsf \in W_t,\; \sfrak_{\wsf} := x(u),\text{ if $\wsf$ is at site $u$ at time $t$},
    \end{align*}
    then
    \begin{equation}
        X^x_t(E) \overset{d}{=} (\sfrak_u)_{u \in E}. 
    \end{equation}
    
\end{proposition}
\begin{example}
    Figure \ref{fig:history} also illustrates the BEP starting from $E = \{1,2\}$. The time axis upward is the one of the Glauber-Exclusion process, and the one downward is the one of the BEP.
\end{example}

The idea of looking at the history backward in time is not new. It is known as \emph{duality} in the context of interacting particle systems. For our case, the BEP was already introduced by De Masi \etal\, in the original paper \cite{DeMasi1986} (for the proof of Proposition \ref{prop:history=bep}, see Theorem 3.1 in \cite{DeMasi1986}). The primary motivation is that it makes the study of correlation functions easier. To study the $k-$points correlation function, we only need to follow the history of $k$ particles backward instead of realizing the whole process forward. This idea also relates to the famous ``coupling from the past" algorithm in \cite{Propp1996} and is also an essential ingredient of the information percolation framework in \cite{Lubetzky2016}. Proposition \ref{prop:history=bep} says that the history itself can be viewed as a Markov process (the BEP). This will be useful for the construction of our coupling.

We introduce some further necessary notations.
\paragraph{Update functions.}
    By construction, for any particle $\wsf$ alive at some point in the time interval $[0,t]$, there exists an increasing boolean function $F_{\wsf, t}$ on $\{-1,1\}^{W_t}$ such that for any spin configuration $\sfrak$ on $W_t$, its extension to $\bigcup_{0\leq s \leq t} W_s$ satisfies
    \begin{equation}
        \sfrak_\wsf = F_{\wsf,t}((\sfrak_{\wsf'})_{\wsf' \in W_t}).
    \end{equation}
    The function $F_{\wsf,t}$ describes the dependence of the spin of $\wsf$ on the spins of the particles alive at time $t$.
\paragraph{Pivotal sets of the update functions.} For any $t \geq 0$, for any particle $\wsf$ that is alive at some time $s < t$, we denote by 
\begin{equation}
    \piv{\wsf, t} := \piv{F_{\wsf,t}}.
\end{equation}
We write $\piv{E, t}$ for $\bigcup_{\wsf\in E} \piv{\wsf,t}$.
\begin{observation}\label{obs:multiplicative}
For any $s < t$,
        \begin{equation}
            \piv{\wsf, t} \subset \bigcup_{\wsf' \in \piv{\wsf, s}} \piv{\wsf', t}.
        \end{equation}
\end{observation}
\paragraph{Update times of the pivotal set.} Let the sequence $(T_j)_{j\geq 0}$ be defined recursively, as follows.
\begin{align*}
    T_0 &= 0,\\
    T_{j+1} &= \infset{t \geq T_j: \text{a Glauber clock of a particle in $\piv{E, T_j}$ rings at time $t$}}.
\end{align*}
\begin{observation}
    $(\piv{E, t})_{t\geq 0}$ is piece-wise constant and only jumps at the times $(T_j)_{j >0}$. The same conclusion is true for $(F_{\wsf,t})_{t\geq 0}$, for any $\wsf \in E$.
\end{observation}
\paragraph{Information percolation.} We now briefly explain the idea that we borrow from \cite{Lubetzky2016}, which will be made more precise later. First, to determine the spin of $\wsf\in E$, we do not need to follow the whole history, but only the part that is ``essential", say $(\piv{\wsf, s})_{0\leq s \leq t}$. In particular, if $\piv{\wsf,s} = \O$ for some $s < t$, then the update function $F_{\wsf,s}$ degenerates, \ie becomes a constant function. Then, we can determine $\sfrak_\wsf$ without tracking further the history in the past. Second, we regard $(\piv{\wsf, t})_{t\geq 0}$ as a percolation cluster in the space-time slab. Surprisingly, we prove in Lemma \ref{prop:hydro=>subcriticality} below that the information percolation cluster defined this way is ``subcritical" in the full high-temperature regime. This subcriticality means that the update functions quickly degenerate, so the spins at time $t$ quickly become independent of those at time $0$ as $t$ grows to infinity, and the system mixes quickly. 

\subsection{The idealized branching process (IBP)}
We will compare the BEP with its idealized version, the \emph{idealized branching process} (IBP). The IBP starting from a particle $\wsf \in \Acal$ is defined as follows.
\begin{itemize}
    \item Initially, $\wsf$ has $q$ independent Poisson clocks, of intensities $\lambda_1, \dots, \lambda_q$, called the Glauber clocks of type $1, \dots,q$ respectively.
    \item Each time a clock of a particle $\wsf'$ rings, we remove $\wsf'$ and add the children of $\wsf'$ into the set of particles. Each of them is given $q$ independent Glauber clocks, and the clocks of different particles are independent.
\end{itemize}

The IBP starting from $\wsf$ is a pair $(\Psi_{IBP}, \Ccal_{IBP})$, where
\begin{itemize}
    \item $\Psi_{IBP}$ is the process associating each time $t\in \Rbb_+$ with the set of particles alive at time $t$, \ie the set of particles born before time $t$ and not having been removed by time $t$.
    \item $\Ccal_{IBP}$ is the counting measure on $[q] \times \Rbb_+$ that tracks the type of the Glauber clocks that ring.
\end{itemize}

\begin{remark}
    In the IBP, the particles no longer live on the lattice. All the particles are different, and all their clocks are independent. In particular, it is a dimension-free object, \ie it does not depend on the size $L$ of the system.
\end{remark}

For any finite subset $E$ of $\Acal$ such that for all $\wsf_1, \wsf_2 \in E$, $\wsf_1$ is not a descendant of $\wsf_2$, the IBP starting from $E$ consists of $|E|$ independent IBPs starting from particles labeled by the elements of $E$. 

\paragraph{Spins and update functions of the IBP.}
Let $\WW_t$ be the set of particles of the IBP alive at time $t$. Similarly to what we have done for the BEP, for any spin configuration $\sft$ on $\WW_t$, we extend $\sft$ to $\bigcup_{0 \leq s \leq t}\WW_t$ by defining for each $\wsf \notin \WW_t$,
    \begin{equation*}
        \sft_{\wsf} := f_i((\sft_{\wsf'})_{\wsf' \in \{\wsf\} \times B(0,m)}).
    \end{equation*}
We can extend $\sft$ to all particles alive at some time $s < t$ with this procedure. In particular, we can define the spin of $\wsf$. We define the update functions $\Ftilde_{\wsf, t}$ , the pivotal sets $(\pivtilde(\Ftilde_{\wsf, t}))_{t\geq 0}$, and the update times $(\Ttilde_j)_{j \in \Zbb_+}$ analogously as for the BEP. 

\subsection{The coupling}\label{subsec:coupling} 
Let $E$ be a subset of $\lattice$. We will construct a coupling between the BEP and the IBP starting from $E$.
Let there be one exclusion clock at each edge of the lattice and an infinite number of Glauber clocks of types $1, \dots, q$. All of them are independent. We will use these clocks to construct the two processes.
We say that the coupling is \emph{succesful} until time $t$ if $\piv{E,t} = \pivtilde(E,t)$, and for each $\wsf \in \piv{E,t}$, the Glauber clocks associated with it at time $t$ are the same in the BEP and the IBP. We now construct the coupling, which should be thought of as embedding the IBP in the space-time slab $\lattice \times \Rbb_+$.
\begin{enumerate}
    \item Whenever an exclusion clock rings, we put an exclusion mark at the corresponding edge. When a particle of the BEP meets an exclusion mark, it jumps to the other endpoint of the edge.
    \item At time $0$, each particle $\wsf \in E$ is associated with the same $q$ Glauber clocks in the BEP and the IBP.
    \item For each particle, the set of clocks associated with it stays unchanged in the interval $(T_{j}, T_{j+1}]$, where $(T_j)_{j\geq 0}$ are the update times of $\piv{E, \cdot}$.
    \item Suppose that the coupling is still  successful until time $T_j$. Then $T_{j+1}$ corresponds to the ring of a Glauber clock of some particle $\wsf \in \piv{E, T_j}$. Note that $T_{j+1} = \Ttilde_{j+1}$ because the same clock of $\wsf$ rings in the IBP. Suppose that $\wsf$ is at site $u$ at time $T_{j+1}-$. Then one of two cases can happen.  
    \begin{enumerate}
        \item \textbf{Case 1}: At time $T_{j+1}-$, there is no particle in $B(u,m)$ except $\wsf$. This means that the children of $\wsf$ are born on sites not occupied by other particles in the pivotal set. Then we update $\piv{E, T_{j+1}}$ and $\pivtilde(E, \Ttilde_{j+1})$: we let the particles in $\piv{E, T_{j+1}-} \setminus \{\wsf\}$ keep their clocks (in IBP and BEP), and we remove from $\piv{E, T_{j+1}-}$ and $\pivtilde(E, \Ttilde_{j+1}-)$ the particles whose spins are no longer necessary to determine the spins atop, and then we associate the children of $\wsf$ in BEP and IBP with the same Glauber clocks. Note that $\piv{E,T_{j+1}} \subset (\piv{E, T_{j+1}-} \setminus \{\wsf\})\cup (\{\wsf\} \times B(0,m))$. In this case, the coupling is still successful until time $T_{j+1}$. 
        \item \textbf{Case 2}: At time $T_{j+1}-$, there is some particles in $B(u,m)$ other than $\wsf$. Then from time $T_{j+1}$, all the particles in the two pivotal sets are given different Poisson clocks. The coupling becomes unsuccessful since then.
    \end{enumerate}
\end{enumerate}
We say that the coupling is successful until infinity if it is successful until $t$ for any finite $t$.
The interest in the above coupling is the following.

\begin{observation}[Spins atop coincide]\label{obs:spinsatop}
    If the coupling is successful until time $t$, and if $\sfrak_\wsf = \sft_\wsf,\; \forall \wsf \in \piv{E,t}$, then $\sfrak_\wsf = \sft_\wsf, \; \forall \wsf \in E$.
\end{observation}
We will show that the coupling is successful until infinity with high probability.
\begin{proposition}[The coupling is successful \whp]\label{prop:coupling}
Let $E$ be an arbitrary subset of $\lattice$. Consider the coupling of the BEP and the IBP starting from $E$ described above. Then 
\begin{equation}
    \proba{\text{The coupling is not successful until infinity}} \leq  \beta \dfrac{|E|^2}{\longeur},
\end{equation}
for some constant $\beta$.
\end{proposition}

The motivation for the coupling is as follows. We are inspired by the local equilibrium phenomenon, so we try to couple the BEP and the IBP so that when we generate the spins of the particles at time $t$ in the two processes according to the same distribution (assuming that those sets are the same), then the spins atop are the same. If this is the case, as the spins atop of the IBP are independent, the spins atop of the BEP are also independent. The naive coupling is to match every particle in the two processes. However, the total number of particles increases exponentially fast, so we can not expect the naive coupling to work until the mixing times. Nevertheless, it is \emph{not necessary} to match the evolution of every particle, but only the ones that truly influence the spins atop, say $\piv{E,\cdot}$.

The idea to look only at $(\piv{E, t})_{t \geq 0}$ ultimately comes from \cite{Lubetzky2016} (in this paper, the authors call it the update support function, see Subsection $2.2$ in \cite{Lubetzky2016}). In Subsection \ref{subsec:analysisBP}, we analyze the IBP, and in particular, we prove that $\pivtilde(\cdot)$ behaves like a subcritical Galton Watson process. In Subsection \ref{subsec:success_coupling}, we use this to prove Proposition \ref{prop:coupling}.
\subsection{Analysis of the IBP}\label{subsec:analysisBP}
In this subsection, we consider the IBP starting from a particle $\wsf \in \Acal$. For convenience, we identify the set of children of $\wsf$ with $B(0,m)$. Let the functions $\phi, \psi: \; \Rbb_+ \to \Rbb$ be defined by
\begin{align}
    \phi(t) &:= \proba{\pivtilde(\wsf, t) \neq \O},\\
    \psi(t) &:= \esperance{|\pivtilde(\wsf, t)|},
\end{align}
which means that $\phi(t)$ is the probability that the process $\pivtilde(\wsf, \cdot)$ still survives at time $t$ and $\psi(t)$ is the average size of $\pivtilde(\wsf, t)$. The main purpose of this subsection is to analyze several aspects of the IBP and in particular, to prove that $\psi(t) \sim e^{R(\rhostar) t}$.

The average size of $\pivtilde(\wsf, t)$ conditionally on survival is 
\begin{equation}\label{eq:average-survival-size}
    \esperance{|\pivtilde(\wsf, t)| \big|\pivtilde(\wsf, t) \neq \O} = \dfrac{\psi(t)}{\phi(t)}.
\end{equation}

Note that if $\pivtilde(\wsf, t) = \O$, then the spin of $\wsf$ does not depend on the spins at time $t$ anymore, so we can safely write 
\begin{equation}
    \vartheta(t) = \esperance{\sft_\wsf \big| \pivtilde(\wsf, t) = \O}
\end{equation}
for the average spin of $\wsf$ conditionally on the extinction of $\pivtilde(\wsf, t)$. Clearly, the functions $\psi, \phi, \vartheta$ do not depend on our choice of $\wsf \in \Acal$.

First, we have a result linking the average of $\sft_\wsf$ with the ODE \eqref{eq:ODE}.
\begin{lemma}[Spin atop of an IBP]\label{lm:spinatop}
    Let $\rho_0 \in [-1,1]$ and $t \geq 0$. Consider the IBP starting from $\wsf$. Suppose that, conditionally on the IBP, the spins at time $t$ are generated independently by a product of Rademacher: $\sft(\WW_t) \sim \rade{\rho_0}^{\otimes\WW_t}$. Then 
    \begin{equation}
        \esperance{\sft_\wsf} = \rho(t),
    \end{equation}
    where $\rho$ is the solution of equation \eqref{eq:ODE}.
\end{lemma}
\begin{figure}
    \centering
    \begin{subfigure}{0.45\textwidth}
        \centering
        \includegraphics[width=\textwidth]{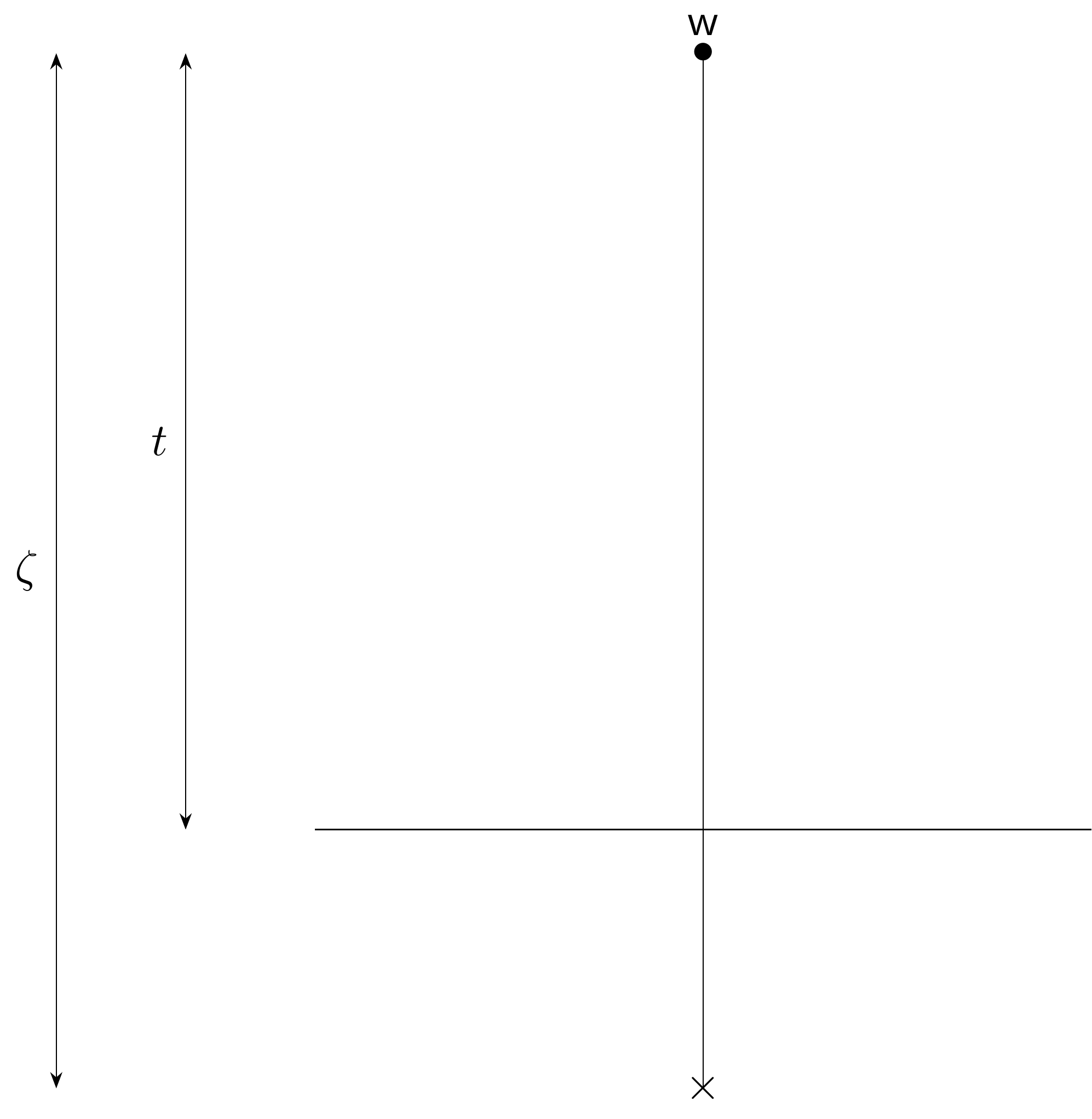}
        \caption{$\zeta > t$}
    \end{subfigure}
    \hfill
    \begin{subfigure}{0.45\textwidth}
        \centering
        \includegraphics[width=\textwidth]{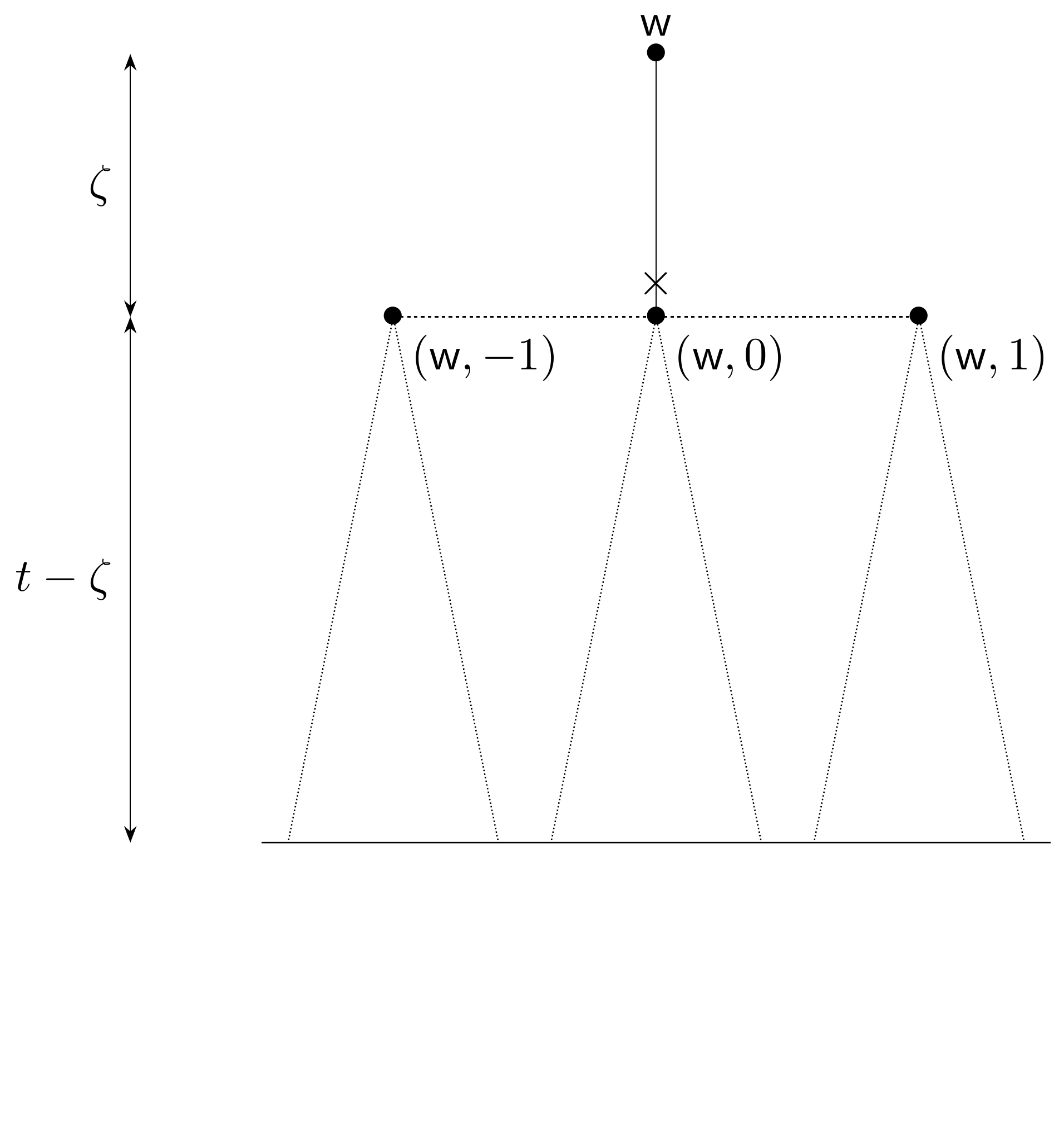}
        \caption{$\zeta < t$}
    \end{subfigure}
    \caption{Conditionally on the first ring}
    \label{fig:first-time-conditional}
\end{figure}
\begin{proof}
    Let $g(t):= \esperance{\sft_\wsf}$. Let $\zeta$ be the first time that a Glauber clock of $\wsf$ rings. Let $\Ecal(i)$ be the event that the Glauber clock ringing at time $\zeta$ is of type $i$. Then $\zeta$ and $\Ecal(i)$ are independent, and $\zeta \sim\exp(\lambda),\; \proba{\Ecal(i)} = \lambda_i/\lambda$. Moreover, the branches starting from the children of $\wsf$ are independent and independent of $\zeta$ and $\Ecal(i)$, each behaving as an IBP. 
    We see that
    \begin{align}
        g(t) = \esperance{\indicator{ \zeta > t}\sft_\wsf} + \esperance{\indicator{\zeta \leq t}\sft_\wsf}. \label{eq17}
    \end{align}
    We estimate the two terms on the right-hand side of \eqref{eq17} separately.
    \begin{itemize}
        \item \textbf{First term}: conditionally on $\zeta$ (see Figure \ref{fig:first-time-conditional}), on the event $\{\zeta > t\}$, $\sft_\wsf \sim \rade{\rho_0}$, and hence
            \begin{align*}
                \esperance{\indicator{ \zeta > t}\sft_\wsf} = \esperance{\indicator{ \zeta > t}\esperance{\sft_\wsf|\zeta}} = \esperance{\indicator{\zeta > t} \times \rho_0} = \rho_0 \proba{\zeta > t} = \rho_0 e^{-\lambda t}.
            \end{align*}
        \item \textbf{Second term:} we see that
        \begin{align*}
            \esperance{\indicator{\zeta \leq t}\sft_\wsf} &= \sum_{i = 1}^q \esperance{\eventindicator{\Ecal(i)}\indicator{\zeta \leq t} f_i((\sft_{\wsf'})_{\wsf' \in \{\wsf\}\times B(0,m)})}\\
            &= \sum_{i = 1}^q \proba{\Ecal(i)} \esperance{\indicator{\zeta \leq t} f_i\left((\sft_{\wsf'})_{\wsf' \in \{\wsf\}\times B(0,m)}\right)},
        \end{align*}
        since $\Ecal(i)$ is independent of $\zeta$ and of the IBPs starting from the children of $\wsf$.
        Conditionally on $\zeta$, on the event $\{\zeta < t\}$, by construction, $(\sft_{\wsf'})_{\wsf'\in \{\wsf\} \times B(0,m)}$ are \iid $\rade{g(t - \zeta)}$. Hence
        \begin{align*}
            \esperance{\indicator{\zeta < t}f_i\left((\sft_{\wsf'})_{\wsf' \in \{\wsf\}\times B(0,m)}\right)} = \esperance{\indicator{\zeta < t} \esperancewithstartingpoint{\nu_{g(t-\zeta)}}{f_i}}.
        \end{align*}
        Hence
        \begin{align*}
             \esperance{\indicator{\zeta \leq t}\sft_\wsf} &= \sum_{i = 1}^q \proba{\Ecal(i)} \esperance{\indicator{\zeta < t} \esperancewithstartingpoint{\nu_{g(t-\zeta)}}{f_i}}\\
             &=\sum_{i = 1}^q \dfrac{\lambda_i}{\lambda}\esperance{\indicator{\zeta < t} \esperancewithstartingpoint{\nu_{g(t-\zeta)}}{f_i}}\\
             &= \esperance{\indicator{\zeta < t} \sum_{i = 1}^q \dfrac{\lambda_i}{\lambda}\esperancewithstartingpoint{\nu_{g(t-\zeta)}}{f_i}}\\
             &= \esperance{\indicator{\zeta < t} \times \dfrac{1}{\lambda} \times \left[R(g(t-\zeta)) + \lambda g(t-\zeta)\right]},
        \end{align*}
            where the last equality comes from Lemma \ref{lm:reaction-term-formula}.
    \end{itemize}
    This implies that the first term of \eqref{eq17} equals $e^{-\lambda t}\rho_0$ and the second term of \eqref{eq17} equals
    \begin{align*}
        \int_0^t \lambda e^{-\lambda s} \dfrac{1}{\lambda} \left(R(g(t-s)) + \lambda g(t-s)\right) \drm s
        &= \int_0^t e^{-\lambda s}\left(R(g(t-s)) + \lambda g(t-s)\right) \drm s\\
        &= \int_0^t e^{-\lambda (t -s)} \left(R(g(s)) + \lambda g(s)\right) \drm s,
    \end{align*}
    where the last inequality follows from the change of variable $s' = t - s$. This implies that 
    \begin{align*}
        g(t) &= e^{-\lambda t} \rho_0 + \int_0^t e^{-\lambda (t -s)} \left(R(g(s)) + \lambda g(s)\right) \drm s\\
        &= e^{-\lambda t}\left(\rho_0 + \int_0^t e^{\lambda s} \left(R(g(s)) + \lambda g(s)\right) \drm s \right).
    \end{align*}
    Now we can differentiate both sides with respect to $t$ to conclude that
    \begin{align*}
        g'(t) &= -\lambda e^{-\lambda t}\left(\rho_0 + \int_0^t e^{\lambda s)} \left(R(g(s)) + \lambda g(s)\right) \drm s \right) + e^{-\lambda t} e^{\lambda t}\left( R(g(t)) + \lambda g(t)\right)\\
        &= -\lambda g(t) + R(g(t)) + \lambda g(t)\\
        &= R(g(t)).
    \end{align*}
    This means that $g$ solves equation \eqref{eq:ODE} with initial condition $g(0) = \rho_0$. This finishes our proof.
\end{proof}

Recall that $\rhoplus$ and $\rhominus$ are solutions of equation \eqref{eq:hydro} with initial conditions $1$ and $-1$. We have the following lemma.
\begin{lemma}\label{lm:h,gamma}
\begin{align}
    \phi(t) &= \dfrac{\rhoplus(t) - \rhominus(t)}{2},\\
    \vartheta(t) &= \dfrac{\rhoplus(t) + \rhominus(t)}{2(1 - \phi(t))}.
\end{align}   
\end{lemma}

\begin{proof}
    Let us generate the IBP up to time $t$. We denote by $\esperancewithstartingpoint{+}{\cdot}$ the probability taken with respect to the IBP when we assign the spins at time $t$ to be all-plus,\,\ie $\sft_{\wsf'} := 1,\, \forall \wsf' \in \WW_t$ and $\esperancewithstartingpoint{-}{\cdot}$ the probability taken when we assign the spins at time $t$ to be all-minus,\,\ie $\sft_{\wsf'} := -1,\, \forall \wsf' \in \WW_t$.
    Note that 
    \begin{align}\label{eq19}
        \esperancewithstartingpoint{+}{\sft_\wsf} &= \esperancewithstartingpoint{+}{\sft_\wsf \indicator{\pivtilde(\wsf,t) \neq \O}} + \esperancewithstartingpoint{+}{\sft_\wsf \indicator{\pivtilde(\wsf,t) = \O}}.
    \end{align}
    Recall that $\pivtilde(\wsf,t) = \piv{\Ftilde_{\wsf,t}}$, where $\sft_\wsf = \Ftilde_{\wsf,t}\left((\sft_{\wsf'})_{\wsf'\in\WW_t}\right)$ for some increasing boolean function $\Ftilde_{\wsf,t}$. Hence 
    \begin{align*}
        \esperancewithstartingpoint{+}{\sft_\wsf \indicator{\pivtilde(\wsf,t) \neq \O}} = \esperancewithstartingpoint{+}{\indicator{\pivtilde(\wsf,t) \neq \O}} = \proba{\pivtilde(\wsf,t) \neq \O},
    \end{align*}
    where the first equality is due to Lemma \ref{lm:pivotal-boolean}. Moreover, when $\pivtilde(\wsf,t) = \O$, $\sft_\wsf$ no longer depends on the spins of particles in $\WW_t$, so the second term in \eqref{eq19} is equal to $(1-\phi(t)) \vartheta(t)$ by definition of the functions $\phi, \vartheta$. This implies that
    \begin{equation*}
        \rhoplus(t) = \phi(t) + (1-\phi(t))\vartheta(t).
    \end{equation*}
    Similarly
    \begin{equation*}
        \rhominus(t) = -\phi(t) + (1-\phi(t))\vartheta(t).
    \end{equation*}
    These two equalities imply what we want.
\end{proof}

In the following proposition, we prove that $\psi$ satisfies a particular differential equation, implying that the pivotal set is subcritical, one of the most important results in this paper.
\begin{proposition}[Differential equation for $\psi$]\label{prop:hydro=>subcriticality} 
There exist a univariate polynomial $Q_1$ and a bivariate polynomial $Q_2$ such that
    \begin{equation}
        \psi' = \psi \left(R'(\rho_*) + (\vartheta - \rho_*) Q_1(\vartheta) + \phi Q_2(\phi, \vartheta)\right).
    \end{equation}
\end{proposition}
\begin{proof}
    The idea is to proceed as in the proof of Lemma \ref{lm:spinatop}. See Figure \ref{fig:first-time-conditional} for intuition. Let us generate the IBP starting from one particle $\wsf$ up to time $t$. Let $\zeta$ be the first time that a Glauber clock of $\wsf$ rings. Then $\zeta \sim \exp(\lambda)$. We see that 
    \begin{align}
        \psi(t) &=  \esperance{|\pivtilde(\wsf, t)|}\nonumber\\
        &=  \esperance{|\pivtilde(\wsf, t)| \indicator{\zeta > t}} +  \esperance{|\pivtilde(\wsf, t)|\indicator{\zeta < t}}. \label{eql10}
    \end{align}
    We estimate the two terms in \eqref{eql10} separately.
    \begin{itemize}
        \item \textbf{First term:} when $\zeta > t$, $\pivtilde(\wsf, t) = \{\wsf\}$ by construction, and hence 
        \begin{align}\label{eq29}
            \esperance{|\pivtilde(\wsf, t)| \indicator{\zeta > t}} = \esperance{|\{\wsf\}| \indicator{\zeta > t}} = \proba{\zeta > t}.
        \end{align}
        \item \textbf{Second term:}\\
        For any subset $A \subset B(0,m)$, any $i \in [q]$, any spin configuration $\eta, \xi$ on $B(0,m)$, let us define the spin configuration $\xi^{A, \eta}$ on $B(0,m)$, the boolean function $f_i^{A, \eta}$ on $\{-1,1\}^{B(0,m)}$, and the events $\Ecal(i), \Ecal(A, \eta)$  by
        \begin{align*}
            \xi^{A, \eta}(j) &:= \eta(j) \times \indicator{j \in A} + \xi(j) \indicator{j \in A^C}, \\
            f_i^{A, \eta}(\xi) &:= f_i(\xi^{A, \eta}), \\
            \Ecal(i) &:= \{\text{The Glauber clock of $\wsf$ ringing at time $\zeta$ is of type $i$}\},\\
            \Ecal(A, \eta) &:= \{\pivtilde(w',t) \neq \O, \forall \wsf' \in A, \text{ and } \pivtilde(w',t) = \O, \forall \wsf' \in A^C  \} \cap \{(\sft_{\wsf'})_{\wsf' \in A^C} = \eta|_{A^C}\}.
        \end{align*}
        Note that, on the event $\Ecal(i) \cap \Ecal(A,\eta)$, 
        \begin{align*}
             \pivtilde(\wsf, t) = \bigcup_{\wsf' \in \piv{f_i^{A^C, \eta}}} \pivtilde(\wsf', t).
        \end{align*}
        Hence 
        \begin{align}
            \esperance{\absolutevalue{\pivtilde(\wsf,t)}\indicator{\zeta < t} \eventindicator{\Ecal(i)} \eventindicator{\Ecal(A, \eta)}} = \sum_{\wsf' \in \piv{f_i^{A^C, \eta}}} \esperance{\absolutevalue{\pivtilde(\wsf',t)}\indicator{\zeta < t} \eventindicator{\Ecal(i)} \eventindicator{\Ecal(A, \eta)}}.
        \end{align}
        By construction, $\zeta, \Ecal(i)$ and the branches of the IBP starting at the children of $\wsf$ are independent. Hence, for any non empty $A \subset B(0,m)$, $\wsf' \in A, \eta \in \{-1,1\}^{B(0,m)}$,   
        \begin{align*}
            &\esperance{\absolutevalue{\pivtilde(\wsf',t)}\indicator{\zeta < t} \eventindicator{\Ecal(i)} \eventindicator{\Ecal(A, \eta)}}\\
            &= \proba{\Ecal(i)} \esperance{\absolutevalue{\pivtilde(\wsf',t)}\indicator{\zeta < t} \eventindicator{\Ecal(A, \eta)}}.\\
            &= \dfrac{\lambda_i}{\lambda} \esperance{\indicator{\zeta < t} \absolutevalue{\pivtilde(\wsf',t)} \prod_{\wsf'' \in A\setminus \{\wsf'\}}\indicator{\pivtilde(\wsf'', t) \neq \O} \prod_{\wsf'' \in A^C}\indicator{\pivtilde(\wsf'', t)=\O, \sft_{\wsf''} = \eta_{\wsf''}}}.
        \end{align*}
        Note that the branches starting at the children of $\wsf$ at time $\zeta$ are independent IBPs independent of $\zeta$, shifted by a time $\zeta$. Then we can integrate out the randomness of these branches to conclude that
        \begin{align*}
            \esperance{\absolutevalue{\pivtilde(\wsf',t-\zeta)}\indicator{\zeta < t} \eventindicator{\Ecal(i)} \eventindicator{\Ecal(A, \eta)}} = \dfrac{\lambda_i}{\lambda} \esperance{\indicator{\zeta < t} \psi \phi^{|A| - 1} \prod_{\wsf'' \in A^C} (1-\phi) \rade{\vartheta}(\eta_{\wsf''})},
        \end{align*}
        where we have used the definition of $\psi, \phi,$ and $\vartheta$. In the formula above, the functions $\psi, \phi, \vartheta$ inside the expectation are evaluated at $t -\zeta$. Hence 
        \begin{align*}
            \esperance{\absolutevalue{\pivtilde(\wsf',t-\zeta)}\indicator{\zeta < t} \eventindicator{\Ecal(i)} \eventindicator{\Ecal(A, \eta)}} = \esperance{\indicator{\zeta < t} \Yfrak(i, A, \eta)},
        \end{align*}
        with 
        \begin{align*}
            \Yfrak(i, A, \eta) := \dfrac{\lambda_i}{\lambda} \psi \phi^{|A| - 1} (1-\phi)^{2m+1 - |A|} \prod_{\wsf'' \in A^C}\rade{\vartheta}(\eta_{\wsf''}),
        \end{align*}
        where the functions $\psi, \phi, \vartheta$ are evaluated at $t -\zeta$. For convention, $\Yfrak(i, A, \eta) := 0$ if $|A| = 0$.
        Hence  
        \begin{align}
            \esperance{\absolutevalue{\pivtilde(\wsf,t)}\indicator{\zeta < t}}
            &= \sum_{i, A, \eta}  \esperance{\eventindicator{\Ecal(i)} \eventindicator{\Ecal(A, \eta)} \indicator{\zeta < t} |\pivtilde(\wsf,t)|}\nonumber\\
            &= \sum_{i, A, \eta} \sum_{\wsf' \in \piv{f_i^{A^C, \eta}}} \esperance{\eventindicator{\Ecal(i)} \eventindicator{\Ecal(A, \eta)} \indicator{\zeta < t} |\pivtilde(\wsf',t)|}\nonumber\\
            &= \esperance{\indicator{\zeta < t} \sum_{i, A, \eta}  \absolutevalue{\piv{f_i^{A^C, \eta}}} \Yfrak(i,A, \eta)},\label{eq25}
        \end{align}
        where the sums are taken on all $i \in [q], A \subset B(0,m)$, and $\eta \in \{-1,1\}^{B(0,m)}$. 
        Note that,
        \begin{align}\label{eq26}
            \sum_{i,A,\eta} \indicator{|A| = 0}\absolutevalue{\piv{f_i^{A, \eta}}} \Yfrak(i,A, \eta) = 0,
        \end{align}
        and
        \begin{align} \label{eq27}
            \sum_{i,A,\eta} \indicator{|A| \geq 2}\absolutevalue{\piv{f_i^{A^C, \eta}}} \Yfrak(i,A, \eta) = \psi \phi g,
        \end{align}
        where $g$ is a polynomial of $\phi$ and $\vartheta$, and the functions $\psi, \phi, \vartheta$ are evaluated at $t - \zeta$. We now estimate the sum on the subsets $A$ such that $|A| = 1$. Note that if $A = \{j\}$, then 
        \begin{align*}
            \absolutevalue{\piv{f_i^{\{j\}^C,\eta}}} = \dfrac{1}{2}\nabla_j f_i(\eta),
        \end{align*}
        due to Lemma \ref{lm:pivotal-boolean}.
        Hence 
        \begin{align}
             \sum_{i,A,\eta} \indicator{|A| = 1}\absolutevalue{\piv{f_i^{A^C, \eta}}} \Yfrak(i,A, \eta)
             &=\sum_{i,j,\eta}  \dfrac{1}{2}\nabla_j f_i(\eta) \Yfrak(i,\{j\}, \eta)\nonumber\\
             &= \sum_{i, j, \eta} \dfrac{1}{2}\nabla_j f_i(\eta)   \times \dfrac{\lambda_i}{\lambda}  \psi (1-\phi)^{2m} \prod_{\wsf'' \in B(0,m) \setminus\{j\}}\rade{\vartheta}(\eta_{\wsf''})\nonumber\\
             &=  \psi (1-\phi)^{2m} \sum_{i = 1}^q \dfrac{\lambda_i}{\lambda} \sum_{j \in B(0,m)} \esperancewithstartingpoint{\nu_\vartheta}{\dfrac{1}{2}\nabla_j f_i}\nonumber\\
             &= \psi(1-\phi)^{2m} \dfrac{1}{\lambda} (R'(\vartheta) + \lambda), \label{eq28}
        \end{align}
    \end{itemize}
where we have used Lemma \ref{lm:reaction-term-formula} in the last equality. The equations \eqref{eq25}, \eqref{eq26}, \eqref{eq27}, \eqref{eq28} together imply that 
   \begin{align}\label{eql9}
       \esperance{\pivtilde(\wsf, t) \indicator{\zeta < t}} &= \esperance{\indicator{\zeta < t} \times \left[\psi(1-\phi)^{2m} \left(\dfrac{R'(\vartheta)}{\lambda} +1\right) + \psi \phi g\right]},
   \end{align}
   where the functions $\psi, \phi, \vartheta$ inside the expectation are evaluated at $t - \zeta$. The equations \eqref{eq29}, \eqref{eql9}, \eqref{eql10} together imply that 
   \begin{align*}
       \psi(t) &= \proba{\zeta > t} + \esperance{\indicator{\zeta < t} \times \left[\psi(1-\phi)^{2m} \left(\dfrac{R'(\vartheta)}{\lambda} +1\right) + \psi \phi g\right]}\\
       &= e^{-\lambda t} + \int_0^t \lambda e^{-\lambda s}\left[\psi(1-\phi)^{2m} \left(\dfrac{R'(\vartheta)}{\lambda} +1\right) + \psi \phi g\right] \drm s
   \end{align*}
   where the functions $\psi, \phi, \vartheta$ inside the integral sign are evaluated at $t - s$. We make a change of variable $s' = t- s$ to conclude that
   \begin{align*}
       \psi(t) &= e^{-\lambda t} + \int_0^t \lambda e^{-\lambda (t-s)}\left[\psi(1-\phi)^{2m} \left(\dfrac{R'(\vartheta)}{\lambda} +1\right) + \psi \phi g\right] \drm s\\
       &= e^{-\lambda t} \left(1 + \int_0^t \lambda e^{\lambda s}\left[\psi(1-\phi)^{2m} \left(\dfrac{R'(\vartheta)}{\lambda} +1\right) + \psi \phi g\right] \drm s\right),
   \end{align*}
   where the functions $\psi, \phi, \vartheta$ are now evaluated at $s$.
   Now, we can differentiate both sides with respect to $t$ to conclude that 
   \begin{align*}
       \psi' = -\lambda \psi + e^{-\lambda t} \lambda e^{\lambda t} \left[\psi(1-\phi)^{2m} \left(\dfrac{R'(\vartheta)}{\lambda} +1\right) + \psi \phi g\right],
   \end{align*}
   where all functions $\psi', \phi, \psi, \vartheta$ are evaluated at $t$. Simplifying that formula, we obtain 
   \begin{align*}
       \psi' = \psi\left(-\lambda + (1-2\phi)^{2m} \lambda + (1-2\phi)^{2m} R'(\vartheta) + \phi g\right).
   \end{align*}
   We can take $Q_1$ and $Q_2$ such that 
   \begin{align*}
       (\vartheta - \rho_*) Q_1(\vartheta) &= R'(\vartheta) - R'(\rho_*),\\
       \phi Q_2(\phi, \vartheta) &= \left((1-2\phi)^{2m} - 1\right) (R'(\vartheta) + \lambda) + \phi g,
   \end{align*}
   to finish the proof.
\end{proof}
\begin{remark}
    In the proof above, the monotonicity of the functions $(f_i)_{1 \leq i \leq q}$ is crucial, and so is Hypothesis \ref{hyp:attractive}.    
\end{remark}

We derive the asymptotic of $\psi$ using Hypothesis \ref{hyp:high_temperature}. 
\begin{lemma}[Asymptotic of some important functions]\label{lm:asymp}
    There exists a constant $\kappa$ such that, for any $t \geq 0$, all the numbers $\log \psi(t), \log \phi(t), \log (\rhoplus(t) - \rho_*), \log (\rho_* - \rhominus(t))$ lie in the interval
    $[R'(\rho_*)t - \kappa, R'(\rho_*)t + \kappa]$.
\end{lemma}

\begin{proof}[Proof of Lemma \ref{lm:asymp}]
    We prove the result for $\psi(t)$. The results for other functions are proved similarly. By Hypothesis \ref{hyp:high_temperature}, we easily see that
    \begin{align*}
        R(\rho) &> 0,\, \forall \rho \in [-1,\rho_*),\\
        R(\rho) &< 0,\, \forall \rho \in (\rho_*, 1].
    \end{align*}
    Hence $\rhoplus(t) \searrow \rho_*$. Moreover, 
    \begin{align*}
        \lim_{t \to \infty} \dfrac{\rhoplus'(t)}{\rhoplus(t) - \rho} = \lim_{t \to \infty} \dfrac{R(\rhoplus(t)) - R(\rho_*)}{\rhoplus(t) - \rho_*} = R'(\rho_*). 
    \end{align*}
    Then for $t$ large enough, 
    \begin{align*}
        \dfrac{\rhoplus'(t)}{\rhoplus(t) - \rho_*} \leq \dfrac{R'(\rho_*)}{2}.
    \end{align*}
    Then, by Gronwall's lemma, there exists a positive constant $\kappa_1$ such that 
    \begin{align*}
        \forall t \geq 0,\; \rhoplus(t) - \rho_* \leq \kappa_1 e^{\frac{1}{2}R'(\rho_*) t}.
    \end{align*}
    One can prove that the inequality above is still true if we replace $\rhoplus(t) - \rho_*$ by the positive functions $\rho_* - \rhominus(t),\, \phi(t),\, |\vartheta(t) - \rho_*|$. Therefore,
    \begin{enumerate}
        \item $\phi(t),\, |\vartheta(t) - \rho_*|$ are absolutely integrable with respect to $t$.
        \item $Q_1(\vartheta)$ and $Q_2(\phi, \vartheta)$ are bounded uniformly in $t$.
    \end{enumerate} 
    Hence $(\vartheta - \rho_*) Q_1(\vartheta) + \phi Q_2(\phi, \vartheta)$ is absolutely integrable as a function of $t$.
    Recall that Proposition \ref{prop:hydro=>subcriticality} implies that 
    \begin{align*}
        \dfrac{\psi'}{\psi} - R'(\rho_*) = (\vartheta - \rho_*) Q_1(\vartheta) + \phi Q_2(\phi, \vartheta).
    \end{align*}
    Hence, taking integration from $0$ to $t$, and using the inequality $\absolutevalue{\int f -\int g} \leq \int\absolutevalue{f-g}$, and noting that $\psi(0) = 1$, we obtain
    \begin{align*}
        \absolutevalue{\log \psi(t) - R'(\rho_*) t} \leq \int_0^t \absolutevalue{(\vartheta - \rho_*) Q_1(\vartheta) + \phi Q_2(\phi, \vartheta)} \leq \kappa,
    \end{align*}
    where 
    \begin{align*}
        \kappa = \int_0^\infty \absolutevalue{(\vartheta - \rho_*) Q_1(\vartheta) + \phi Q_2(\phi, \vartheta)}.
    \end{align*}
    This finishes our proof.
\end{proof}

\begin{remark}
    Hypothesis \ref{hyp:high_temperature} is crucial for the proof above.
\end{remark}
We present some elementary but useful results on $\WW_t$.
\begin{lemma}\label{lm:ub_for_size}
    Consider the IBP starting from a set $E$. Then there exists a constant $\kappa$ such that, for any number $t > 0$, 
    \begin{equation*}
        \esperance{|\WW_t|^2} \leq e^{\kappa t} |E|^2.
    \end{equation*}
\end{lemma}

\begin{proof}[Proof of Lemma \ref{lm:ub_for_size}]
    Note that $(|\WW_t|)_{t\geq 0}$ is a Markov process itself. Let $\Lcal$ be the generator associated with $(|\WW_t|)_{t\geq 0}$. Then, for any function $\varphi: \Zbb_+ \to \Rbb$, for any $n \in \Zbb_+$, 
    \begin{equation*}
        \Lcal \varphi(n) = \lambda n \left[\varphi(n + 2m+1 -1) - \varphi(n)\right] = \lambda n\left[\varphi(n + 2m) - \varphi(n)\right].
    \end{equation*}
    We also know that, see \eg \cite{Ethier1986},
    \begin{align*}
        \ddt \esperance{\varphi(|\WW_t|)} &= \esperance{\Lcal \varphi(|\WW_t|)}.
    \end{align*}
    Applying the formula above for $\varphi: n \mapsto n^2$, we obtain
    \begin{align*}
        \ddt \esperance{|\WW_t|^2} &= \esperance{\lambda |\WW_t|(4m |\WW_t| + 4m^2)}\\
        &\leq \lambda (4m+4m^2) \esperance{|\WW_t|^2}.
    \end{align*}
    Then, we can apply Gronwall's lemma to conclude that 
    \begin{align*}
        \esperance{|\WW_t|^2} \leq e^{(4m + 4m^2)\lambda t} \esperance{|\WW_0|^2},
    \end{align*}
    which finishes the proof.
\end{proof}

\begin{lemma}\label{lm:wtilde_dominate_w}
    Consider the BEP and IBP starting from $E\subset\lattice$. Then, $|\WW_t|$ stochastically dominates $|W_t|$ for any $t \geq 0$.
\end{lemma}
\begin{proof}[Proof of Lemma \ref{lm:wtilde_dominate_w}]
    The reason is that the particles in the two processes are given $q$ Poisson clocks with the same rates. For the BEP, each time a clock of a particle rings, the particle is removed, and \emph{at most} $2m+1$ new particles are added to $W_t$ because some new particles are born on the sites already occupied, while for the process $|\WW_t|$, each time a clock of a particle rings, \emph{exactly} $2m+1$ new particles are added.
\end{proof}

\subsection{Success of the coupling}\label{subsec:success_coupling}
Now we can prove Proposition \ref{prop:coupling}. We will need a result about anti-concentration of the Interchange Process IP(2).
\paragraph{Interchange process IP(2).} The interchange process IP(2) on $\latticedimensiond$, whose edges have conductance $L^2$, is a couple of random walk $(\Ucal_1, \Ucal_2)$ that has the following description. The state space is $(\latticedimensiond \times \latticedimensiond) \setminus\{(u,u)| u\in \latticedimensiond\}$. For any initial condition, the evolution is as follows. Each edge of $\latticedimensiond$ is associated with an independent Poisson clock of intensity $L^2$. Whenever a clock rings, if $\Ucal_1$ (or $\Ucal_2$) is at one endpoint of the corresponding edge, it jumps to the other endpoint.  

From now on, we denote by $\distance{\cdot,\cdot}$ the shortest-path distance on the lattice. We have the following result.
\begin{lemma}[Anticoncentration]\label{lm:anticoncentration}
    For any $d \in \Zbb_{>0}$. Let $(\Ucal_1, \Ucal_2)$ be an IP(2) on $\latticedimensiond$, whose edges have conductance $L^2$. Let $\theta$ be a strictly positive number and $\zeta \sim \exp(\theta)$, independent of $(\Ucal_1,\Ucal_2)$. Then, for any $k \in \Zbb_+$, 
    \begin{equation*}
        \max_{u_1, u_2 \in \latticedimensiond; \, u_1 \neq u_2}\probawithstartingpoint{u_1, u_2}{\distance{\Ucal_1(\zeta), \Ucal_2(\zeta)} \leq k} = \begin{cases}
            \mathcal{O}_{\theta,k} (1/L) &\text{if $d = 1$},\\
            \mathcal{O}_{\theta,k} (\log L/L^2) &\text{if $d = 2$},\\
            \mathcal{O}_{\theta,k} (1/L^2) &\text{if $d \geq 3$}.
        \end{cases} 
    \end{equation*}
\end{lemma}
To lighten the notation, we write $\piv{t}$ for $\piv{E,t}$ and $\pivtilde(t)$ for $\pivtilde(E,t)$.
The proof of Proposition \ref{prop:coupling} is divided into two following lemmas.
\begin{lemma}\label{lm:1}
    With the same notations as in Proposition \ref{prop:coupling}, 
    \begin{equation*}
        \proba{\text{The coupling is not successful until infinity}} = \Ocal{\dfrac{\sum_{j = 0}^\infty\esperance{\pivtilde(\Ttilde_j)}}{L}}.
    \end{equation*}
\end{lemma}
\begin{lemma}\label{lm:2}
    \begin{equation*}
        \sum_{j = 0}^\infty \esperance{\absolutevalue{\pivtilde(\Ttilde_j)}} = \Ocal{|E|^2}.
    \end{equation*}
\end{lemma}
Clearly, those two lemmas imply Proposition \ref{prop:coupling}. Those two lemmas separate the effect of the exclusion process and the IBP. 
\begin{proof}[Proof of Lemma \ref{lm:1}]
    We say that two particles $\wsf_1$ and $\wsf_2$ have an interaction if there is a time $t$ such that:
    \begin{enumerate}
        \item Both are in the pivotal set at time $t-$, \ie $\wsf_1, \wsf_2 \in \piv{t-}$, and they do not occupy the same site.
        \item At time $t$, a clock of one of the two particles rings, and then one of the children is born on the site occupied by the other particle.
    \end{enumerate}
    In short, two particles in $\piv{\cdot}$ have an interaction if one splits into the site occupied by the other. For example, in Figure \ref{fig:history}, $\wsf$ and $\wsf'$ have an interaction at time $s_3$. By definition, the coupling is successful up to time $t$ if there is no interaction up to time $t$. Hence 
    \begin{align*}
        \proba{\text{The coupling is not successful until infinity}} \leq \proba{\text{There is an interaction}}.
    \end{align*}
    We only need to prove 
    \begin{equation}\label{}
        \proba{\text{There is an interaction}} \leq \Ocal{\dfrac{\sum_{j = 0}^\infty\esperance{\pivtilde(\Ttilde_j)}}{L}}.
    \end{equation}
    Note that 
    \begin{align*}
        \proba{\text{There is an interaction}} = \sum_{\wsf_1, \wsf_2} \proba{\text{The first interaction is between $\wsf_1, \wsf_2$}}.
    \end{align*}
    Recall the definition of the update times $(T_j)_{j\in \Zbb_+}$ of the pivotal set. Let $\Ecal(j, \wsf_1, \wsf_2)$ be the event that $T_j$ is the first time such that both $\wsf_1$ and $\wsf_2$ are in the pivotal set.
    Let $\Ecal(j)$ be the event that there is no interaction up to time $T_j$, and let $\Ecal(\wsf_1, \wsf_2)$ be the event that there is an interaction between $\wsf_1, \wsf_2$. 
    Then
    \begin{align*}
        \proba{\text{The first interaction is of $\wsf_1, \wsf_2$}} \leq \sum_{j = 0}^\infty \proba{\Ecal(j, \wsf_1, \wsf_2) \cap \Ecal(j) \cap \Ecal(\wsf_1, \wsf_2)}.
    \end{align*}
    Let $\Fcal_j$ be the \sigmaalgebra generated by the Poisson processes up to time $T_j$. Then 
    \begin{align*}
        \proba{\Ecal(j, \wsf_1, \wsf_2) \cap \Ecal(j) \cap \Ecal(\wsf_1, \wsf_2)} = \esperance{\eventindicator{\Ecal(j, \wsf_1, \wsf_2)} \eventindicator{\Ecal(j)} \proba{\Ecal(\wsf_1, \wsf_2) \big| \Fcal_j}}
    \end{align*}
    Let $\tau$ be the time such that a Glauber clock of $\wsf_1$ or $\wsf_2$ rings counted from $T_j$:
    \begin{align*}
        \tau := \minset{t > 0: \text{a Glauber clock of $\wsf_1$ or $\wsf_2$ rings at time $T_j + t$}}. 
    \end{align*} Conditionally on $\Fcal_j$, on the event $\Ecal(j, \wsf_1, \wsf_2) \cap \Ecal(j)$, from time $T_j$ onwards, $\wsf_1$ and $\wsf_2$ move as an IP(2) on the lattice, until one of their Glauber clocks rings at time $\tau\sim \exp(2\lambda)$. To have an interaction, the distance between $\wsf_1$ and $\wsf_2$ at time $\tau$ must be smaller than $m$. Therefore, by Lemma \ref{lm:anticoncentration}, conditionally on $\Fcal_j$, on the event $\Ecal(j, \wsf_1, \wsf_2) \cap \Ecal(j)$,
    \begin{align*}
        \proba{\Ecal(\wsf_1, \wsf_2) \big| \Fcal_j} \leq \beta /L,
    \end{align*}
    for some constant $\beta$.
    We deduce that, 
    \begin{align*}
        \proba{\text{There is an interaction}} \leq \sum_{\wsf_1, \wsf_2} \sum_{j=0}^{\infty} \dfrac{\beta }{L} \proba{\Ecal(j, \wsf_1, \wsf_2) \cap \Ecal(j)}.
    \end{align*}
    Note that on the event $\Ecal(j, \wsf_1, \wsf_2)$, either $\wsf_1$ or $\wsf_2$ (or both of them) is born at time $T_j$. We denote by $\Ecal(j,\wsf_i)$ the event that $\wsf_i$ is born at time $T_j$, $i \in \{1,2\}$. Hence 
    \begin{align*}
        \sum_{\wsf_1, \wsf_2} \eventindicator{\Ecal(j, \wsf_1, \wsf_2)}\eventindicator{\Ecal(j)} 
        \leq \sum_{i \in \{1,2\}} \sum_{\wsf_1, \wsf_2} \eventindicator{\Ecal(j, \wsf_1, \wsf_2)}\eventindicator{\Ecal(j)} \eventindicator{\Ecal(j, \wsf_i)}.
    \end{align*}
    Moreover, note also that on the event $\Ecal(j, \wsf_1, \wsf_2)$, $\wsf_2 \in \piv{T_j}$. Hence 
    \begin{align*}
        \sum_{\wsf_1, \wsf_2}\eventindicator{\Ecal(j, \wsf_1, \wsf_2)}\eventindicator{\Ecal(j)} \eventindicator{\Ecal(j, \wsf_1)} &\leq \sum_{\wsf_1, \wsf_2} \indicator{\wsf_2 \in \piv{T_j}} \eventindicator{\Ecal(j, \wsf_1)}\eventindicator{\Ecal(j)}\\
        &\leq \sum_{\wsf_1} \absolutevalue{\piv{T_j}}\eventindicator{\Ecal(j, \wsf_1)} \eventindicator{\Ecal(j)}\\
        &\leq (2m+1) \absolutevalue{\piv{T_j}}\eventindicator{\Ecal(j)},
    \end{align*}
    where all these inequalities are true almost surely. The last inequality is because at most $(2m+1)$ particles are born at time $T_j$.
    Note that, on the event $\Ecal(j)$, $\piv{T_j} = \pivtilde(\Ttilde_j)$. This implies that
    \begin{align*}
        &\sum_{\wsf_1, \wsf_2} \eventindicator{\Ecal(j, \wsf_1, \wsf_2)}\eventindicator{\Ecal(j)}\leq 2(2m+1) \absolutevalue{\pivtilde(\Ttilde_j)}.
    \end{align*}
    Taking expectation, we conclude that 
    \begin{align*}
        \proba{\text{There is an interaction}} \leq 2(2m+1) \dfrac{\beta}{L} \sum_{j = 0}^\infty \esperance{\pivtilde(\Ttilde_j)},
    \end{align*}
    which finishes the proof.
\end{proof}
We finish this section by proving Lemma \ref{lm:2}. The left-hand side is the sum of the expected sizes of the process $\absolutevalue{\pivtilde(\cdot)}$ at its update times. The intuition is that the inequality is true if we replace $\pivtilde(\cdot)$ by a subcritical Galton-Watson process. The only problem is that $\absolutevalue{\pivtilde(\cdot)}$ is not Markovian. Nevertheless, it is ``subcritical", and the idea is to enlarge the process $\absolutevalue{\pivtilde(\cdot)}$ to make it Markovian and then adapt the proof of the subcritical Galton-Watson case.
\begin{proof}[Proof of Lemma \ref{lm:2}]
    Let $t_0$ be a number such that $\psi(t_0) < 1$, which exists thanks to Lemma \ref{lm:asymp}. We rewrite the left-hand side as
    \begin{align*}
        \sum_{j=0}^\infty \sum_{i=0}^\infty \esperance{\absolutevalue{\pivtilde(\Ttilde_j)} \indicator{i t_0 \leq \Ttilde_j < (i+1) t_0}}. 
    \end{align*}
    Let the sequence $(B_i)_{i \in \Zbb_+}$ be defined recursively as follows.
    \begin{align*}
        B_0 &:=\pivtilde(0),\\
        B_i &:= \bigcup_{\wsf \in B_{i-1}} \pivtilde(\wsf, it_0),\, \forall i \geq 1.
    \end{align*}
    It is not hard to see that $\pivtilde(it_0) \subset B_i, \forall u \in \Zbb_+$, because given the IBP up to time $it_0$ and the spins of the particles in $B_i$, by definition of $B_i$, we can recursively determine the spins of the particles in $B_{i-1}$, and so on, up to time $0$. It is also not hard to see that the sequence $(|B_i|)_{i \in \Zbb_+}$ is a subcritical Galton-Watson process, by its definition and the definition of $t_0$.

    Let $\kappa$ be the constant in Lemma \ref{lm:ub_for_size}. We claim that
    \begin{align*}
        \esperance{\sum_{j = 0}^\infty \absolutevalue{\pivtilde(\Ttilde_j)} \indicator{it_0 \leq \Ttilde_j < (i+1) t_0}} \leq \dfrac{e^{\kappa t_0}}{2m} \esperance{|B_i|^2}.
    \end{align*}
    To see this, consider the IBP starting from $B_i$ (at time $it_0$) up to time $(i+1) t_0$. We denote this process by $\gw_i$. Let $(\Tcal_j)_{j \in \Zbb_+}$ be the update times of $\gw_i$. Then $\gw_i(it_0) = B_i$, and
    \begin{align*}
        \esperance{\sum_{j = 0}^\infty \absolutevalue{\pivtilde(\Ttilde_j)} \indicator{it_0 \leq \Ttilde_j < (i+1) t_0}} &\leq \esperance{\sum_{j = 0}^\infty \absolutevalue{\gw_i(\Tcal_j) \indicator{it_0 \leq \Tcal_j < (i+1) t_0}}}.
    \end{align*}
    This is because $\gw_i$ can be thought of as an enlargement of $\pivtilde(\cdot)$ in the time interval $[it_0, (i+1)t_0]$, where we first enlarge the set $\pivtilde(i t_0)$ to $B_i$, and then we do not remove any particle of the branching process starting from $B_i$. 
    
    Note that the size of $\gw_i(\cdot)$ increases almost surely, hence
    \begin{align*}
        \absolutevalue{\gw_i(\Tcal_j)\indicator{it_0 \leq \Tcal_j < (i+1) t_0}} \leq \absolutevalue{\gw_i((i+1)t_0)}.
    \end{align*}
    Moreover, at each update time $\Tcal_j$, the size of $\gw_i$ increases by 2m, and hence 
    \begin{align*}
        \sum_{j = 0}^\infty \indicator{it_0 \leq \Tcal_j < (i+1) t_0} \leq \dfrac{\absolutevalue{\gw_i((i+1)t_0)}}{2m}. 
    \end{align*}
    This implies 
    \begin{align*}
        \sum_{j = 0}^\infty \absolutevalue{\gw_i(\Tcal_j) \indicator{it_0 \leq \Tcal_j < (i+1) t_0}} \leq \dfrac{\absolutevalue{\gw_i((i+1)t_0)}^2}{2m}.
    \end{align*}
    By Lemma \ref{lm:ub_for_size}, 
    \begin{align*}
        \esperance{\absolutevalue{\gw_i((i+1)t_0)}^2} \leq e^{\kappa t_0}\esperance{|B_i|^2}.
    \end{align*}
    This proves the claim. Now, we can take the sum over $i$ to conclude that 
    \begin{align*}
        \esperance{\sum_{j = 0}^\infty \absolutevalue{\pivtilde(\Ttilde_j)}} \leq \dfrac{e^{\kappa t_0}}{2m} \sum_{i = 0}^\infty \esperance{|B_i|^2}.
    \end{align*}
    Since $(|B_i|)_{i\in \Zbb_+}$ is a subcritical Galton-Watson process, we have 
    \begin{align*}
        \sum_{i = 0}^\infty \esperance{|B_i|^2} = \Ocal{\esperance{|B_0|^2}} = \Ocal{|E|^2}.
    \end{align*}
    This finishes the proof.
\end{proof}
\section{Application of the dual coupling}\label{sec:application_coupling}
Now we can prove Theorem \ref{thm:gap}, Theorem \ref{thm:precutoff}, and the lower bound in Theorem \ref{thm:cutoff}. 
\begin{proposition}[Replacement lemma]\label{prop:replacement}
    Suppose that $X_0 \sim \rade{\rho_0}^{\otimes \lattice}$, for some $\rho_0 \in (-1,1)$.
    Then, for any subset $E \subset \lattice$, 
    \begin{equation}
        \sup_{t\geq 0} \dtv{\text{Law}\left(X_t(E)\right)}{\rade{\rho(t)}^{\otimes |E|}} \leq \dfrac{ \beta |E|^2}{\longeur},
    \end{equation}
    where $\rho(\cdot)$ is the solution of equation \eqref{eq:ODE}, and $\beta$ is the constant in Proposition \ref{prop:coupling}.
\end{proposition}
\begin{remark}
    The proposition above is of independent interest. It shows that the joint distribution of $o\left(\sqrt{\longeur}\right)$ arbitrary sites is close to that of a product of \iid Rademacher. This provides sharp estimates on the correlation functions, allowing us to derive the correct lower bound on the mixing times using distinguishing statistics.
\end{remark}

Proposition \ref{prop:replacement} is just a direct corollary of Proposition \ref{prop:coupling}.
\begin{proof}
    Consider the coupling of BEP and IBP starting from $E$ in Subsection \ref{subsec:coupling}. We generate the spins at time $t$ by $\sfrak(W_t) \sim \rade{\rho_0}^{\otimes W_t}$ and $\sft(\WW_t)\sim \rade{\rho_0}^{\otimes \WW_t}$. Conditionally on the success of the coupling until time $t$, we couple the spins on $\piv{E,t}$ and $\pivtilde(E,t)$ such that $\sfrak_\wsf = \sft_\wsf,\; \forall \wsf \in \piv{E,t}$. Then Observation \ref{obs:spinsatop} implies that
    \begin{equation*}
        \sfrak_\wsf = \sft_\wsf,\; \forall \wsf \in E.
    \end{equation*}
    Recall that, by Proposition \ref{prop:history=bep}, 
    \begin{align*}
        X_t(E) &\overset{d}{=} (\sfrak_\wsf)_{\wsf \in E},
    \end{align*}
    and by Lemma \ref{lm:spinatop},
    \begin{align*}
        (\sft_\wsf)_{\wsf \in E} &\overset{d}{=} \rade{\rho(t)}^{\otimes E}.
    \end{align*}
    Hence 
    \begin{align*}
        \dtv{\text{Law}\left(X_t(E)\right)}{\rade{\rho(t)}^{\otimes |E|}} &\leq \proba{(\sfrak_\wsf)_{\wsf \in E} \neq (\sft_\wsf)_{\wsf \in E} }\\
        &\leq \proba{\text{The coupling is not successful until time $t$}}\\
        &\leq \dfrac{\beta |E|^2}{\longeur}.
    \end{align*}
    This finishes our proof.
\end{proof}

\begin{corollary}[1-point and 2-point correlation functions]\label{cor:correlation}
    Under the hypothesis of Proposition \ref{prop:replacement}, for any $u, v \in \lattice$, for any $t \geq 0$, 
    \begin{align}
        \absolutevalue{\esperance{X_t(u)} - \rho(t)} &\leq 2\beta/\longeur, \label{1point}\\
        \absolutevalue{\covariance{X_t(u)}{X_t(v)}} &\leq 12\beta/\longeur. \label{2points}
    \end{align}
\end{corollary}
\begin{proof}
    Let $u, v \in \lattice$, and let $\xi_1, \xi_2$ be \iid Rademacher $\rade{\rho(t)}$. Let $\beta$ be the constant in Proposition \ref{prop:coupling}.
    \begin{enumerate}
        \item Applying Proposition \ref{prop:replacement} with $E := \{u\}$, we deduce that there exists a coupling of $X_t(u)$ and $\xi_1$ such that 
        \begin{align*}
            \proba{X_t(u) \neq \xi_1} \leq \dfrac{\beta}{L}.
        \end{align*}
        Hence, 
        \begin{align*}
            \absolutevalue{\esperance{X_t(u)} - \rho(t)} &= \absolutevalue{\esperance{X_t(u) - \xi_1}} \\
            &= \absolutevalue{\esperance{(X_t(u) - \xi_1) \indicator{X_t(u) \neq \xi_1}}} \\
            &\leq \esperance{\absolutevalue{X_t(u) - \xi_1}\indicator{X_t(u) \neq \xi_1}}\\
            &\leq 2\proba{X_t(u) \neq \xi_1}\\
            &\leq 2\beta/L.
        \end{align*}
        \item Similarly, applying Proposition \ref{prop:replacement} with $E := \{u,v\}$, we deduce that there exists a coupling of $(X_t(u), X_t(v))$ and $(\xi_1, \xi_2)$ such that if we let $\Ecal := \{(X_t(u), X_t(v)) \neq (\xi_1, \xi_2)\}$, then $\proba{\Ecal} \leq \dfrac{4\beta}{L}$. Hence 
        \begin{align*}
            \absolutevalue{\covariance{X_t(u)}{X_t(v)}}&= \absolutevalue{\esperance{X_t(u) X_t(v)} - \esperance{X_t(u)}\esperance{X_t(v)}}\\
            &\leq \absolutevalue{\esperance{X_t(u) X_t(v)} - \esperance{\xi_1\xi_2}} + \absolutevalue{\esperance{\xi_1\xi_2} - \esperance{X_t(u)}\esperance{X_t(v)}}.
        \end{align*}
        The first term in the last expression is 
        \begin{align*}
            \absolutevalue{\esperance{(X_t(u)X_t(v) - \xi_1\xi_2) \eventindicator{\Ecal}}} \leq 2\proba{\Ecal} \leq 8\beta/L.
        \end{align*}
        The second term does not exceed
        \begin{align*}
            \absolutevalue{\esperance{\xi_1}} \absolutevalue{\esperance{\xi_2} - \esperance{X_t(v)}} + \absolutevalue{\esperance{X_t(v)}} \absolutevalue{\esperance{\xi_1} - \esperance{X_t(u)}} \leq 2\beta/L + 2\beta/L = 4\beta/L.
        \end{align*}
    \end{enumerate}
    This finishes our proof.
\end{proof}
We are ready to prove the lower bound on the mixing times. We denote by $(X^+_t)_{t \geq 0}$ the Glauber-Exclusion process starting from all-plus and recall that $\rhoplus(\cdot)$ is the solution of \eqref{eq:ODE} with the initial condition $\rhoplus(0) = 1$.
\begin{proof}[Proof of the lower bound]
    The proof uses the classical distinguishing statistic method (see Section $7.3$ in \cite{Levin2017}). Our distinguishing statistic is the total magnetization $\sum\limits_{u \in \lattice}x(u)$.
    By abuse of notation, let $X_\infty \sim \pi$ (so that for any initial configuration, $X_t \xrightarrow[t \to \infty]{d} X_\infty$). Let $\beta$ be defined as in Proposition \ref{prop:coupling}. From \eqref{1point}, we deduce that 
    \begin{align*}
        \absolutevalue{\esperance{\sum_{u \in \lattice} X^+_t(u)} - \longeur\rhoplus(t)} \leq 2\beta,
    \end{align*}
    and by letting $t \to \infty$,
    \begin{equation*}
        \absolutevalue{\esperance{\sum_{u \in \lattice} X_\infty(u)} - \longeur\rho_*} \leq 2\beta.
    \end{equation*}
    From \eqref{2points}, we see that 
    \begin{align*}
        \variance{\sum_{u \in \lattice} X^+_t(u)} &= \sum_{u \in \lattice} \variance{X^+_t(u)} + \sum_{u \neq v} \covariance{X^+_t(u)}{X^+_t(v)}\\
        &\leq L + L(L-1) \times 12\beta/L \\
        &< (12\beta+1) L.
    \end{align*}
    By letting $t$ tend to infinity, we get
    \begin{equation*}
        \variance{\sum_{u \in \lattice} X_\infty(u)} \leq (12\beta+1)L.
    \end{equation*}
    So by Proposition 7.9 in \cite{Levin2017}, 
    \begin{align*}
        \dtv{\proba{X^+_t \in \cdot}}{\pi} \geq 1 - 8 \dfrac{\maxset{\variance{\sum\limits_{u \in \lattice} X^+_t(u)}, \variance{\sum\limits_{u \in \lattice} X_\infty(u)}}}{\left(\esperance{\sum\limits_{u \in \lattice} X^+_t(u)} - \esperance{\sum\limits_{u \in \lattice} X_\infty(u)}\right)^2}.
    \end{align*}
    Let $t := \dfrac{\log \longeur}{-2R(\rho_*)} - \kappa$ for some constant $\kappa$ that we will choose later. Let $\kappa_1$ be a constant satisfying Lemma \ref{lm:asymp}. Then
    \begin{align*}
        \rhoplus(t) - \rho_* &\geq e^{R(\rho_*) t - \kappa_1}\\
        &= \dfrac{e^{-R(\rho_*)\kappa - \kappa_1}}{\sqrt{L}}.
    \end{align*}
    Therefore, when $L$ is large enough,
    \begin{align*}
        \esperance{\sum_{u \in \lattice} X^+_t(u)} - \esperance{\sum_{u \in \lattice} X_\infty(u)} &\geq L(\rhoplus(t) - \rho_*) - 4\beta\\
        &= \sqrt{L} e^{-R(\rho_*)\kappa - \kappa_1} - 4\beta.\\
        &> \dfrac{1}{2} \sqrt{L} e^{-R(\rho_*)\kappa - \kappa_1}.
    \end{align*}
    Therefore,
    \begin{equation*}
        8 \dfrac{\maxset{\variance{\sum_{u \in \lattice} X^+_t(u)}, \variance{\sum_{u \in \lattice} X_\infty(u)}}}{\left(\esperance{\sum_{u \in \lattice} X^+_t(u)} - \esperance{\sum_{u \in \lattice} X_\infty(u)}\right)^2} < \dfrac{32(12\beta + 1)}{e^{-2R(\rho_*)\kappa - 2\kappa_1}}.
    \end{equation*}
    Note that as $R'(\rho_*) < 0$, for $\kappa$ large enough, the last expression is smaller than $\epsilon$. Hence
    \begin{equation*}
        \dtv{\proba{X^+_t \in \cdot}}{\pi} \geq 1 - \epsilon,
    \end{equation*}
    and therefore, 
    \begin{align*}
        \tmix > t = \dfrac{\log \longeur}{-2R(\rho_*)} - \kappa,
    \end{align*}
    which finishes our proof.
\end{proof}
\begin{remark}
    The result in higher dimensions is proved similarly. One can first show a version of Proposition \ref{prop:coupling} for higher dimension, where the bound on the right-hand side is replaced by $\Ocal{\dfrac{|E|^2 \log L}{L^2}}$ for dimension $d = 2$, and by $\Ocal{\dfrac{|E|^2}{\longeur^2}}$ for $d \geq 3$, and then propagates these changes. The only modification needed in the proof is to use the anticoncentration in Lemma \ref{lm:anticoncentration} for the appropriate dimension. 
\end{remark}

Now we prove Theorem \ref{thm:gap} and the upper bound in Theorem \ref{thm:precutoff}. First, we need the following lemma. 
\begin{lemma}\label{lm:subcriticalityBEP}
    Consider a BEP starting with one particle $\wsf$. Then, there exist constants $\kappa$ and $\beta$ such that for any $t \in \Rbb_+$, 
    \begin{equation*}
        \esperance{\absolutevalue{\piv{\wsf, t}}} \leq \psi(t) + e^{\kappa t} \dfrac{\beta}{\sqrt{L}}.
    \end{equation*}
\end{lemma}
\begin{proof}[Proof of Theorem \ref{thm:gap} and upper bound in Theorem \ref{thm:precutoff}]
    Let $x \in \Xcal$ be an arbitrary configuration. We want to compare $\probawithstartingpoint{x}{X_t \in \cdot}$ and $\probawithstartingpoint{\pi}{X_t \in \cdot}$ for some $t > 0$. We use the same BEP\, starting from $E = \lattice$ to generate $X^x_t$ and $X^\pi_t$, the processes starting from $x$ and $\pi$. Let $\wsf \in \lattice$ be arbitrary. We see that 
    \begin{align*}
        \dtv{\probawithstartingpoint{x}{X_t \in \cdot}}{\probawithstartingpoint{\pi}{X_t \in \cdot}} &\leq \proba{X^x_t \neq X^\pi_t}\\
        &\leq \sum_{u \in \lattice} \proba{X^x_t(u) \neq X^\pi_t(u)}\\
        &\leq \sum_{u \in \lattice} \proba{\piv{u, t} \neq \O}\\
        &= \longeur \proba{\piv{\wsf, t} \neq \O}\\
        &\leq \longeur \esperance{\absolutevalue{\piv{\wsf,t}}}.
    \end{align*}
    By Observation \ref{obs:multiplicative}, it is not hard to see that,  
    \begin{equation*}
        \esperance{\absolutevalue{\piv{\wsf,mt}}}  \leq \esperance{\absolutevalue{\piv{\wsf,t}}}^m.
    \end{equation*}
    Hence, with the constant $\kappa$ and $\beta$ as in Lemma \ref{lm:subcriticalityBEP}, 
    \begin{equation}\label{eq:exponentialdecay}
        \dtv{\probawithstartingpoint{x}{X_{mt} \in \cdot}}{\probawithstartingpoint{\pi}{X_{mt} \in \cdot}} \leq \longeur \esperance{\absolutevalue{\piv{\wsf,mt}}} \leq \longeur \left(\psi(t) + e^{\kappa t} \dfrac{\beta}{\sqrt{L}}\right)^m.
    \end{equation}
    Provided that $\left(\psi(t) + e^{\kappa t} \dfrac{\beta}{\sqrt{L}}\right) < 1$, 
    we can take $m = \ceil{\dfrac{\log \longeur - \log \epsilon}{- \log \left(\psi(t) + e^{\kappa t} \dfrac{\beta}{\sqrt{L}}\right)}}$ 
    to make the right-hand side smaller than $\epsilon$. This implies that 
    \begin{equation*}
        \tmixxepsilon \leq mt \leq \left(\dfrac{\log \longeur - \log \epsilon}{- \log \left(\psi(t) + e^{\kappa t} \dfrac{\beta}{\sqrt{L}}\right)} + 1\right) t,
    \end{equation*}
    for any $t$ such that $\left(\psi(t) + e^{\kappa t} \dfrac{\beta}{\sqrt{L}}\right) < 1$. It remains to choose a suitable $t$ that leads to our results.
    We will choose $t$ such that $e^{\kappa t} \dfrac{\beta}{\sqrt{L}} \ll \psi(t)L^{-1/3}$. This will imply 
    \begin{align*}
        \psi(t) + e^{\kappa t} \dfrac{\beta}{\sqrt{L}} \leq \psi(t)(1 + L^{-1/3}). 
    \end{align*}
    Hence, with $\kappa_1$ the constant in Lemma \ref{lm:asymp}, we have
    \begin{align*}
        \log \left(\psi(t) + e^{\kappa t} \dfrac{\beta}{\sqrt{L}} \right) \leq \log \psi(t) + \log(1 + L^{-1/3} )\leq R(\rho_*)t + \kappa_1 + L^{-1/3}.
    \end{align*}
    Therefore, if $R(\rho_*)t + \kappa_1 + L^{-1/3} < 0$, 
    \begin{align*}
        \tmix &\leq \dfrac{\log \longeur - \log \epsilon}{-(R(\rho_*)t + \kappa_1 + L^{-1/3})} t + t\\
        &= \dfrac{\log \longeur - \log \epsilon}{-(R(\rho_*) + \kappa_1/t + 1/(tL^{1/3})} + t\\
        &= \dfrac{\log L}{-R(\rho_*)} + \Ocal{\dfrac{\log L}{t}} + t.
    \end{align*}
    We can take $t = \sqrt{\log L}$, which satisfies all the conditions we impose on $t$, to finish the proof of the upper bound in Theorem \ref{thm:precutoff} for $d = 1$. The proof for higher dimensions is similar.
    
    It remains to prove Theorem \ref{thm:gap}. Recall that 
    \begin{align*}
        \gap &= \lim_{t \to \infty} -\dfrac{1}{t}\log \max_{x\in \Omega} \dtv{\probawithstartingpoint{x}{X_{t} \in \cdot}}{\pi}.\\
        &= \lim_{m \to \infty} -\dfrac{1}{mt}\log \max_{x\in \Omega} \dtv{\probawithstartingpoint{x}{X_{mt} \in \cdot}}{\pi}.
    \end{align*}
    On the other hand, equation \eqref{eq:exponentialdecay} implies that, for any $t \geq 0$, $m\in \Zbb_+$, 
    \begin{equation*}
        \log \max_{x\in \Omega} \dtv{\probawithstartingpoint{x}{X_{mt} \in \cdot}}{\pi} \leq \log \longeur + m \log \left(\psi(t) + e^{\kappa t} \dfrac{\beta}{\sqrt{L}} \right).
    \end{equation*}
    Therefore, 
    \begin{equation*}
        \gap \geq \lim_{m \to \infty} -\dfrac{1}{mt} \left(\log \longeur + m \log \left(\psi(t) + e^{\kappa t} \dfrac{\beta}{\sqrt{L}} \right) \right) = -\dfrac{1}{t} \log \left(\psi(t) + e^{\kappa t} \dfrac{\beta}{\sqrt{L}}\right).
    \end{equation*}
    This implies 
    \begin{align*}
        \liminf_{L \to \infty} \gap \geq -\dfrac{1}{t} \log \psi(t) \geq -\dfrac{1}{t}(R(\rho_*) t + \kappa_1) = -R(\rho_*) - \kappa_1/t.
    \end{align*}
    Note that this inequality is true for any positive $t$. Letting $t$ tend to infinity, we obtain the conclusion of Theorem \ref{thm:gap}.
\end{proof}
Now we prove Lemma \ref{lm:subcriticalityBEP}.
\begin{proof}[Proof of Lemma \ref{lm:subcriticalityBEP}]
    We use the coupling of BEP and IBP starting from one particle $\wsf$. We write $\piv{t}$ for $\piv{\wsf, t}$ and $\pivtilde(t)$ for $\pivtilde(\wsf,t)$. Let $\Ecal$ denote the event that the coupling is successful until time $t$. Then, on $\Ecal$, we have 
    \begin{equation*}
        \absolutevalue{\piv{t}} = \absolutevalue{\pivtilde(t)}.
    \end{equation*}
    Let $\beta_1$ be as in Proposition \ref{prop:coupling}. Then $\proba{\Ecal^C} \leq \beta_1/L$, where $\Ecal^C$ denotes the complement of $E$. Hence 
    \begin{align*}
        \esperance{|\piv{t}|} &= \esperance{|\piv{t}| (\oneds_{\Ecal} + \oneds_{\Ecal^C})}\\
        &= \esperance{|\piv{t}| \oneds_{\Ecal}} + \esperance{|\piv{t}| \oneds_{\Ecal^C}}.
    \end{align*}
    Note that 
    \begin{equation*}
        \esperance{|\piv{t}| \oneds_{\Ecal}} = \esperance{|\pivtilde(t)| \oneds_{\Ecal}} \leq \esperance{|\pivtilde(t)|} = \psi(t).
    \end{equation*}
    Let $\kappa_1$ be the constant in Lemma \ref{lm:ub_for_size}. Note also that
    \begin{equation*}
        \esperance{|\piv{t}| \oneds_{\Ecal^C}}^2\leq \esperance{|W_t| \oneds_{\Ecal^C}}^2 \leq \esperance{|W_t|^2} \proba{\Ecal^C} \leq \esperance{|\WW_t|^2} \proba{\Ecal^C}  \leq e^{\kappa_1 t} \times \dfrac{\beta_1}{L}, 
    \end{equation*}
    where we have used $|\piv{t}| \leq |W_t|$ for the first inequality, the Cauchy-Schwarz inequality for the second inequality, Lemma \ref{lm:wtilde_dominate_w} for the third inequality, and Lemma \ref{lm:ub_for_size} for the last inequality. 
    Therefore, 
    \begin{align*}
        \esperance{|\piv{t}|} \leq \psi(t) + e^{\kappa_1 t/2} \times \dfrac{\sqrt{\beta_1}}{\sqrt{L}}.
    \end{align*}
    We can take $\kappa = \kappa_1/2$ and $\beta = \sqrt{\beta_1}$ to finish the proof.
\end{proof}
\section{Upper bound}\label{sec:ub}
In this section, we prove the upper bound in Theorem \ref{thm:cutoff}. For simplicity, we only give the proof for $d = 1$, but we will comment on how to adapt the proof for $d = 2$.

Recall that in the decomposition of the local flip-rate function $c(\cdot)$ in Proposition \ref{prop:glauber_interpretation}, $f_1 \equiv 1,\, f_2 \equiv -1$, so the updates given by $f_1$, $f_2$ are called \emph{oblivious} because they can be carried out without looking at the spins of any sites. These updates will play a special role in our proof. It is more convenient to merge two oblivious deterministic updates into a single oblivious random update as follows. We allow the collection of marks to contain another type of mark: the refresh marks, which can appear on the sites in the space-time slab. We adapt the definition of the collection of marks accordingly.
From now on, let 
\begin{equation}
    \rhobar := \dfrac{\lambda_1 - \lambda_2}{\lambda_1 + \lambda_2}.
\end{equation}
We will use the following graphical construction.
\paragraph{Graphical Construction 2.} Let the background process $\Xi$ be defined as follows.
\begin{equation}
    \Xi = \left((\xiexclusion_u)_{0\leq u \leq \longeur-1}, (\xirefresh_u)_{0 \leq u \leq L-1} , (\xiglauber_{i,u})_{3 \leq i\leq q,\, 0 \leq u \leq \longeur-1}\right),
\end{equation}
where $(\xiexclusion_u)_{0\leq u \leq \longeur-1}, (\xirefresh_u)_{0 \leq u \leq L-1} , (\xiglauber_{i,u})_{3 \leq i\leq q,\, 0 \leq u \leq \longeur-1}$ are independent homogeneous Poisson processes, and 

\begin{itemize}
    \item $\xiexclusion_u$ is of intensity $\longeur^2$,\, $0 \leq u\leq \longeur-1$,
    \item $\xirefresh_u$ is of intensity $\lambda_1 + \lambda_2$, $0 \leq u \leq L-1$,
    \item $\xiglauber_{i,u}$ is of intensity $\lambda_i$,\, $3\leq i\leq q,\, 0\leq u\leq \longeur-1$.
\end{itemize}

The process $\Xi$ naturally defines a collection of marks $\Ccal$ consisting of exclusion, refresh, and Glauber marks as follows.

\begin{itemize}
    \item  Whenever $\xiexclusion_u$ jumps, we put an exclusion mark on the edge $(u, u+1)$.
    \item Whenever $\xirefresh_u$ jumps, we put a refresh mark at site $u$.
    \item Whenever $\xiglauber_{i,u}$ jumps, we put a Glauber mark of type $i$ at site $u$. 
\end{itemize}

Given the collection of marks $\Ccal$, an infinite sequence of independent Rademacher variables $\rade{\rhobar}$, and an initial condition $x_0 \in \Xcal$, we construct the process, which we always denote by $(X^{x_0}_t)_{t\geq 0}$, as follows.
\begin{itemize}
    \item The process $(X^{x_0}_t)_{t\geq 0}$ is a piecewise constant process starting from $x_0$  and can only jump when a mark appears.
    \item When we see an exclusion mark, say on the edge $(u, u+1)$, we make the transition $x \mapsto x^{u \leftrightarrow u+1}$ (exchange the spins of sites $u$ and $u+1$).
    \item When we see a refresh mark, say at site $u$, we replace the $u-th$ coordinate of $X$ by an independent $\rade{\rhobar}$ (randomize the spin at site $u$). 
    \item When we see a Glauber mark, say of type $i$ at site $u$, we make the transition $x \mapsto x^{u, f_i(x_{u + \cdot})}$ (update site $u$ using the function $f_i$).
\end{itemize}

Then, when the collection of marks $\Ccal$ is generated by $\Xi$ and the Rademacher variables $\rade{\rhobar}$ are independent of $\Xi$, the process $(X^{x_0}_t)_{t\geq 0}$ is a Markov process with the generator $\Lcal_{GE}$ starting from $x_0$. 
\paragraph{The grand coupling.} We can use the same Poisson process $\Xi$ and the same independent Rademacher variables $\rade{\rhobar}$ to construct the Glauber-Exclusion process from any configuration $x \in \Xcal$. It is not hard to see that this coupling preserves order, \ie, $x \leq x' \Rightarrow \forall t \geq 0, \, X^{x}_t \leq X^{x'}_t$.

\paragraph{Correlation marks associated with Glauber marks.} It is convenient to associate each Glauber mark with a set of points as follows. Whenever we see a Glauber mark at site $u$, we put on each site in $B(u,m)$ a mark, and we call those marks the correlation marks associated to the aforementioned Glauber mark. These correlation marks indicate that the spins of those sites are (possibly) used in the corresponding Glauber updates. 

We introduce another Markov process $Z$, which we call the \emph{independent-site process}, constructed as follows.
\paragraph{Independent-site process.} From the collection of marks generated by $\Xi$, let $(Z_t)_{t \geq 0}$ be the process taking values in $\Zcal \setvalue \{0,1\}^{\lattice}$ constructed as follows. For any initial configuration $z_0 \in \Zcal$,
\begin{itemize}
    \item The process $(Z^{z_0}_t)_{t \geq 0}$ is a piecewise constant process starting from $z_0$ and can only jump when a mark appears.
    \item When we see an exclusion mark, say on the edge $(u, u+1)$, we make the transition $z \mapsto z^{u \leftrightarrow u+1}$.
    \item When we see a refresh mark, say at site $u$, we make the transition $z \mapsto z^{u,1}$.
    \item When we see a Glauber mark, say at site $u$, we remove all the sites with the associated correlation marks from $Z$, \ie we make the transition $z \mapsto z \setminus B(u,m)$.
\end{itemize}
Here we have identified the elements of $\{0,1\}^{\lattice}$ with the subsets of $\lattice$. As its name suggests, the process $Z$ tracks the positions of independent sites in $\textbf{X}$ conditionally on $\Xi$. Roughly speaking, when a site is resampled according to an independent Rademacher $\rade{\rhobar}$ due to a refresh mark, it becomes independent of all other sites and hence is added to $Z$. On the other hand, when a site $u$ is updated due to a Glauber mark, it creates a correlation with its neighbors, and all these neighbors are removed from $Z$. By abuse of notation, for any $z \subset \lattice$, we denote by $\probawithstartingpoint{z}{\cdot}$ the distribution of the process $Z$ starting from $z$.
\paragraph{The grand coupling for $Z$.} We can use the same background process $\Xi$ to construct the process $Z$ from any configuration $z \subset \lattice$. 
\begin{remark}[Attractiveness of the process $Z$]\label{monotonicityofZ}
    For any $z_1, z_2 \subset \lattice$ such that $z_1 \leq z_2$, if $Z^{z_1}$ and $Z^{z_2}$ are constructed using the grand coupling, then almost surely,
    \begin{align*}
        \forall t \geq 0, \; Z^{z_1}_t \leq Z^{z_2}_t.
    \end{align*}
\end{remark}

\paragraph{Red, blue, and green regions.}
Let the red, blue, and green regions be defined as follows:
    \begin{align}
        \blueset(t) &= Z^{\O}_{t}, \nonumber\\
        \redset(t) &= \{u \in \lattice | X^{+}_{t}(u) \neq X^{-}_{t}(u)  \}, \nonumber\\
        \greenset(t) &= \lattice \setminus \left(\redset(t) \cup \blueset(t)\right).\nonumber
    \end{align}
    Here, $(X^{+}_{t})_{t\geq 0}$ and $(X^{-}_{t})_{t\geq 0}$ denote the process $X$ starting from configurations all-plus and all-minus, respectively.
    To lighten the notation, we write $X^x_t(\blueset)$ for $X^x_t(\blueset(t))$, and similarly for $\redset$, and $\greenset$. Let $\Fcal_t$ be the \sigmaalgebra generated by $\Xi([0,t])$ and $\left(X_s^x(\lattice \setminus\blueset(s)\right)_{x \in \Xcal, 0 \leq s \leq t}$.  
    The interest in introducing red, blue, and green regions is the following. 
    \begin{observation}
        By construction, almost surely, conditionally on $\Fcal_t$, for any $x \in \Xcal$, $X^x_t(\blueset) \sim \rade{\rhobar}^{\otimes {\blueset(t)}}$ and is independent of $X^x_t(\redset \cup \greenset)$.   
    \end{observation}
    From now on, for any subset $E \subset\lattice$, we will use the shorthand notation $\nu_E := \rade{\rhobar}^{\otimes E}$.
    The meanings of the red, blue, and green regions above are as follows: the red region $\redset$ is the set of disagreements of $X^{+}$ and $X^{-}$. 
    Conditionally on $\Fcal_t$, the spins in the blue region $\blueset$ have the product measure $\nu_{\blueset(t)}$, independent of the spins on $\redset\cup\greenset$, and the green region $\greenset$ is the set of sites whose spins may present correlation but independent of the initial configuration. More precisely, notice that, as $X^{+}_{t}(\greenset) = X^{-}_{t}(\greenset)$, by monotonicity, in the grand coupling, the spins on $\greenset$ do not depend on the initial configuration anymore:
    \[X^{+}_t(\greenset)= X^{-}_t(\greenset) = X^{x}_t(\greenset),\; \forall x \in \Xcal.\]
    Therefore, we can safely denote it by $X_t(\greenset)$. Note also that $X_t(\greenset)$ and $X^x_t(\redset),\, x\in \Xcal$, are $\Fcal_t$ measurable.\\ 

    The following lemma summarizes the discussion above.
    \begin{lemma}[The conditional law]\label{lm:conditional_law_onxi}
        Almost surely, for any $t \geq 0$, conditionally on $\Fcal_t$, 
    \begin{equation}
        X^x_{t} \sim \nu_{\blueset(t)} \otimes \delta_{X^x_{t}(\redset)} \otimes \delta_{X_{t}(\greenset)},
    \end{equation}
    where $\delta_{X_{t}(\greenset)}$ (resp. $\delta_{X^x_{t}(\redset)}$) is the Dirac measure on $\{-1,1\}^{\greenset(t)}$ (resp. $\{-1,1\}^{\redset(t)}$) which gives all mass to $X_{t}(\greenset)$ (resp. $X^x_{t}(\redset)$).
    \end{lemma}

\subsection{Proof of the upper bound}
First, note that, by the definition of $\redset$, 
    \begin{align*}
        \absolutevalue{\redset(t)} = \dfrac{1}{2}\esperance{\sum_{u \in \lattice} \left(X^{+}_{t}(u) - X^{-}_{t}(u)\right) \Big| \Fcal_t}.
    \end{align*}
This and Corollary \ref{cor:correlation} lead to the following lemma.
\begin{lemma}[Decay of the red region]\label{lm:decay_of_redset}
    \begin{align}
        \esperance{\absolutevalue{\redset(t)}} = \dfrac{L}{2} (\rho^+(t) - \rho^-(t)) + \Ocal{1}.
    \end{align}
\end{lemma}

For any $\beta \in \Rbb_{>0}$, let 
\[\badset_\beta = \left\{z \in \Zcal: \exists\, 0 \leq l \leq \ceil{\dfrac{L}{\ceil{\beta\log L}}},\, z \cap \integerinterval{l \ceil{\beta \log L}}{(l+1) \ceil{\beta \log L} - 1} = \O \right\},\]
where $\integerinterval{i}{j} := [i,j] \cap \Zbb.$ 

We will need the following results.

\begin{lemma}[BAD is rarely visited] \label{lm:badset}
There exist constants $\beta_2, \beta_3 \in \Rbb_{>0}$ such that, for any $t > \beta_3$, 
    \begin{equation}
        \max_{z \in \Zcal} \probawithstartingpoint{z}{\exists s \in [t,t + 1]: Z_s \in \badset_{\beta_2}} = \Ocal{1/L^6}.
    \end{equation}    
\end{lemma}

\begin{proposition}[``Adjacent" measures quickly become close]\label{prop:adjacent_distribution}
    Let $z \subset \lattice$, let $u \in \lattice \setminus z$, and let $y$ be an arbitrary spin configuration on $\lattice \setminus (z\cup \{u\})$. Let $\pi_1, \pi_2$ be the two probability distributions on $\Xcal$ given by 
    \begin{equation}
        \pi_1 = \delta^u_1 \otimes \nu_z \otimes \delta_{y}; \; \pi_2 = \nu_{z\cup \{u\}} \otimes \delta_{y},
    \end{equation}
    where $\delta^u_1$ denotes the Dirac measure on $\{-1,1\}^{\{u\}}$ which gives all mass to $1$. Let $\beta_2$ be as in Lemma \ref{lm:badset}. Then there exists a constant $\kappa$ such that 
    \begin{equation}
        \normtv{\probawithstartingpoint{\pi_1}{X_1 \in \cdot} - \probawithstartingpoint{\pi_2}{X_1 \in \cdot}} \leq  \kappa \left( \dfrac{\log L}{\sqrt{L}} + \sqrt{\probawithstartingpoint{z \cup \{u\}}{\exists s \in [0,1]: Z_s \in \badset_{\beta_2}}}\right).
    \end{equation}
    The same result holds if we replace $\pi_1$ by $  \delta^u_{-1} \otimes \nu_z \otimes \delta_{y}$.
\end{proposition}
Proposition \ref{prop:adjacent_distribution} says that very quickly, one cannot distinguish two processes starting from the measure of the form $\pi_2$ and its perturbation $\pi_1$: the total variation distance drop from $1$ to $\dfrac{\log L }{\sqrt{L}}$ in a time $\Ocal{1}$. This is the result whose proof involves the information percolation framework and the idea from excursion theory.

With Lemma \ref{lm:badset} and Proposition \ref{prop:adjacent_distribution}, we can now prove the upper bound in Theorem \ref{thm:cutoff}.
\begin{proof}[Proof of the upper bound]
    Let $t = t_1 + 1$, for some number $t_1$ that we will choose later. Let $\beta_2, \beta_3$ be as in Lemma \ref{lm:badset}. To lighten the notation, we write $\badset$ for $\badset_{\beta_2}$. We will estimate directly $\dtv{\probawithstartingpoint{x_1}{X_t \in \cdot}}{\probawithstartingpoint{x_2}{X_t \in \cdot}}$ for $x_1, x_2$ arbitrary in $\Xcal$, as we do not have an explicit formula for the invariant measure $\pi$. We use the grand coupling to construct the Glauber-Exclusion process from any initial configuration. To lighten the notation, here we write $\redset,\, \blueset$, and $\greenset$ for $\redset(t_1), \blueset(t_1)$, and $\greenset(t_1)$. Let $\absolutevalue{\redset} = \cardinalred$ and $\redset = \{u_1,\dots,u_{\cardinalred}\}$, which are all measurable \wrt $\Fcal_{t_1}$. Let the (random) distribution $\Lfrak(x, j)$, for $x\in \Xcal$, $0 \leq j \leq \cardinalred$, be defined by 
    \[\Lfrak(x, j) = \nu_{\blueset \cup \{u_1,\dots, u_j\}} \otimes \delta_{X^x_{t_1}(\redset \setminus\{u_1,\dots, u_j\})} \otimes \delta_{X_{t_1}(\greenset)},\]
    which means that $\Lfrak(x, j)$ is the law of the random configuration obtained by resampling the spins on the sites in $\{u_1, \dots, u_j\}$ of $X^x_{t_1}$ according to \iid $\rade{\rhobar}$ independent of $X^x_{t_1}$. By Lemma \ref{lm:conditional_law_onxi}, almost surely, $\Lfrak(x, 0)$ is exactly the conditional distribution of $X^x_{t_1}$ given $\Fcal_{t_1}$, and 
    \[\Lfrak(x, \cardinalred) = \nu_{\blueset \cup \redset}\otimes \delta_{X_{t_1}(\greenset)},\]
    which does not depend on $x$. So we can write $\Lfrak(\cardinalred)$ for $\Lfrak(x, \cardinalred)$. The point here is that the sequence $(\Lfrak(x, j))_{j = 0, \dots, \cardinalred}$ forms a ``path" from $\Lfrak(x, 0)$ to $\Lfrak(\cardinalred)$. 
    We see that, for any $x_1, x_2 \in \Xcal$, 
    \begin{align}\label{eq1}
        \dtv{\probawithstartingpoint{x_1}{X_t \in \cdot}}{\probawithstartingpoint{x_2}{X_t \in \cdot}} &\leq \esperance{\normtv{\probawithstartingpoint{x_1}{X_t \in \cdot | \Fcal_{t_1}}-\probawithstartingpoint{x_2}{X_t \in \cdot | \Fcal_{t_1}}}}\nonumber\\
        &= \esperance{\normtv{\probawithstartingpoint{\Lfrak(x_1,0)}{X_1 \in \cdot}-\probawithstartingpoint{\Lfrak(x_2,0)}{X_1 \in \cdot}}}\nonumber\\
        &\leq 2\max_{x \in \Xcal} \esperance{\normtv{\probawithstartingpoint{\Lfrak(x,0)}{X_1 \in \cdot}-\probawithstartingpoint{\Lfrak(\cardinalred)}{X_1 \in \cdot}}}.
    \end{align}
    By the triangle inequality, for any $x \in \Xcal$,
    \begin{align}
        \normtv{\probawithstartingpoint{\Lfrak(x,0)}{X_1 \in \cdot}-\probawithstartingpoint{\Lfrak(\cardinalred)}{X_1 \in \cdot}} \leq \sum_{j = 0}^{\cardinalred - 1} \normtv{\probawithstartingpoint{\Lfrak(x,j)}{X_1 \in \cdot}-\probawithstartingpoint{\Lfrak(x, j+1)}{X_1 \in \cdot}}.\label{eq13}
    \end{align}
    By Proposition \ref{prop:adjacent_distribution}, there exists a constant $\kappa$ such that 
    \begin{align}\label{eq17}
        &\normtv{\probawithstartingpoint{\Lfrak(x,j)}{X_1 \in \cdot}-\probawithstartingpoint{\Lfrak(x, j+1)}{X_1 \in \cdot}}\nonumber\\
        &\leq \kappa  \left(\dfrac{\log L}{\sqrt{L}} + \sqrt{\probawithstartingpoint{\blueset \cup\{u_1, \dots, u_{j+1}\}}{\exists s \in [0,1]: Z_s \in \badset}}\right)\nonumber\\
        &\leq \kappa  \left(\dfrac{\log L}{\sqrt{L}} + \sqrt{\probawithstartingpoint{\blueset}{\exists s \in [0,1]: Z_s \in \badset}}\right).
    \end{align}
    We have used the attractiveness of $Z$ and the fact that $\badset$ is a decreasing subset of $\lattice$ in the last inequality. The equations \eqref{eq1}, \eqref{eq13}, \eqref{eq17} together imply
    \begin{align}
        \dtv{\probawithstartingpoint{x_1}{X_t \in \cdot}}{\probawithstartingpoint{x_2}{X_t \in \cdot}} &\leq 2 \esperance{\cardinalred \, \kappa  \left(\dfrac{\log L}{\sqrt{L}} + \sqrt{\probawithstartingpoint{\blueset}{\exists s \in [0,1]: Z_s \in \badset}}\right)}\nonumber\\
        &\leq 2\dfrac{\kappa\log L}{\sqrt{L}} \esperance{\cardinalred} + 2\kappa\esperance{\cardinalred \sqrt{\probawithstartingpoint{\blueset}{\exists s \in [0,1]: Z_s \in \badset}}}.\label{eq16}
    \end{align}    
The latter term in the last expression above can be bounded by the Cauchy-Schwarz inequality:  
\begin{align*}
    \esperance{\cardinalred \sqrt{\probawithstartingpoint{\blueset}{\exists s \in [0,1]: Z_s \in \badset}}} &\leq \sqrt{\esperance{\cardinalred^2} \esperance{\probawithstartingpoint{\blueset}{\exists s \in [0,1]: Z_s \in \badset}}}\\
    &\leq \sqrt{L^2\probawithstartingpoint{\O}{\exists s \in [t_1, t_1 + 1]: Z_s \in \badset}}.
\end{align*}
Here, we have used the facts that $\cardinalred \leq L$, $\blueset = Z^{\O}_{t_1}$ (by definition) and the Markov property of the process $Z^{\O}$ at time $t_1$. Then, by Lemma \ref{lm:badset}, provided that $t_1 \geq \beta_3$,
\begin{align*}
    \esperance{\cardinalred \sqrt{\probawithstartingpoint{\blueset}{\exists s \in [0,1]: Z_s \in \badset}}}  &\leq \sqrt{L^2 \Ocal{1/L^6}}\\
    &= \Ocal{1/L^2}.
\end{align*}
It remains to estimate the first term in \eqref{eq16}. Thanks to Lemma \ref{lm:decay_of_redset} and Lemma \ref{lm:asymp}, we can take $t_1 = \dfrac{\log L}{2|R'(\rhostar)|} + \dfrac{2\log\log L}{|R'(\rhostar)|}$ to make $\dfrac{\log L}{\sqrt{L}} \esperance{\cardinalred} = \Ocal{1/\log L}$. Clearly, $t_1 \geq \beta_3$ when $L$ is large enough. This implies that 
\begin{align*}
    \dtv{\probawithstartingpoint{x_1}{X_t \in \cdot}}{\probawithstartingpoint{x_2}{X_t \in \cdot}} = \Ocal{1/\log L} + \Ocal{1/L^2} = \Ocal{1/\log L},
\end{align*}
for any $x_1, x_2 \in \Xcal$. This finishes our proof.

\end{proof}
The rest of the paper is devoted to proving Lemma \ref{lm:badset} and Proposition \ref{prop:adjacent_distribution}. In Subsection \ref{subsec_badset}, we prove Lemma \ref{lm:badset}, and finally, in Subsection \ref{subsec:localtime} and Subsection \ref{subsec:excursion}, we prove Proposition \ref{prop:adjacent_distribution}.

\subsection{Proof of Lemma \ref{lm:badset}: the bad set is rarely visited}\label{subsec_badset}
First, we need the following lemmas.
\begin{lemma}[the bad set has small mass]\label{lm:ponctuel_badset}
    There exist positive constants $\beta_2$ and $\beta_3$ such that for any $t \geq \beta_3$, and for any deterministic subset $E$ of $\lattice$ of cardinality at least $\ceil{\beta_2 \log L}$,
    \begin{equation}
        \max_{z \in \Zcal} \probawithstartingpoint{z}{Z_t \cap E = \O} =\Ocal{1/L^{10}}. 
    \end{equation}
\end{lemma}
\begin{lemma}[Inertia in a short time]\label{inertia}
    Let $\tau_{out}$ be the time at which the process $Z$ jumps out of the initial state $Z_0$: 
    \begin{equation}
        \tau_{out} = \infset{ t \geq 0: Z_t \neq Z_0}.
    \end{equation}
    Then there is a constant $\beta  > 0$ such that,  
    \begin{equation}
        \min_{z \in \Zcal}\probawithstartingpoint{z}{\tau_{out} > \dfrac{1}{L^3}} \geq \beta . 
    \end{equation}
\end{lemma}
Now we can prove Lemma \ref{lm:badset}.
\begin{proof}[Proof of Lemma \ref{lm:badset}]
    Let $z$ be an element of $\Zcal$. Let $\beta_2, \beta_3$ be as in Lemma \ref{lm:ponctuel_badset}. We write $\badset$ for $\badset_{\beta_2}$ to lighten the notation. We apply the union bound in Lemma \ref{lm:ponctuel_badset} for the subsets $E_j = \integerinterval{j \ceil{\beta_2 \log L}}{(j+1) \ceil{\beta_2 \log L} - 1},\, j \in \integerinterval{0}{\ceil{\dfrac{L}{\ceil{\beta_2 \log L}}}}$, to obtain: 
    \begin{equation*}
        \forall t \geq \beta_3,\;   \probawithstartingpoint{z}{Z_t \in \badset} = \Ocal{1/L^9}.
    \end{equation*}
    This implies that for any $t \geq \beta_3$,
    \begin{equation}\label{eq5}
        \esperancewithstartingpoint{z}{\int_{t}^{t+1+1/L^3} \indicator{Z_s \in \badset} \drm s} = \int_t^{t+1+1/L^3} \probawithstartingpoint{z}{Z_s \in \badset} \drm s = \Ocal{1/L^9}.
    \end{equation}
    Let 
    \begin{align*}
        \tau_{\badset} := \infset{t \geq 0: Z_t \in \badset}.
    \end{align*}
    By the Markov property, 
    \begin{align}\label{eq6}
        &\esperancewithstartingpoint{z}{\int_{t}^{t+1+1/L^3} \indicator{Z_s \in \badset} \drm s} \nonumber\\
        &= \esperancewithstartingpoint{z}{\esperancewithstartingpoint{Z_t}{\int_{0}^{1+1/L^3} \indicator{Z_s \in \badset} \drm s}}\nonumber\\
        &\geq \esperancewithstartingpoint{z}{\esperancewithstartingpoint{Z_t}{\indicator{\tau_{\badset} \leq 1}\int_{0}^{1+1/L^3} \indicator{Z_s \in \badset} \drm s}} \nonumber\\
        &\geq \esperancewithstartingpoint{z}{\esperancewithstartingpoint{Z_t}{\indicator{\tau_{\badset} \leq 1}\esperancewithstartingpoint{Z_{\tau_{\badset}}}{\int_0^{1 + 1/L^3 - \tau_{\badset}} \indicator{Z_s \in \badset} \drm s }}}.
    \end{align}
    Note that if $\tau_{\badset} \leq 1$, then $\int_0^{1 + 1/L^3 - \tau_{\badset}} \indicator{Z_s \in \badset} \drm s \geq \int_0^{1/L^3} \indicator{Z_s \in \badset} \drm s$. On the other hand, with $\beta$ as in Lemma \ref{inertia}, we have 
    \[\esperancewithstartingpoint{Z_{\tau_{\badset}}}{\int_0^{1/L^3} \indicator{Z_s \in \badset} \drm s } \geq \esperancewithstartingpoint{Z_{\tau_{\badset}}}{\dfrac{1}{L^3} \indicator{\tau_{out} > \dfrac{1}{L^3}}} \geq \min_{z \in \Zcal} \dfrac{1}{L^3} \probawithstartingpoint{z}{\tau_{out} > \dfrac{1}{L^3}} \geq \dfrac{\beta }{L^3}.\]
    This and \eqref{eq6} together imply
    \begin{align}
        \esperancewithstartingpoint{z}{\int_{t}^{t+1+1/L^3} \indicator{Z_s \in \badset} \drm s} \geq \dfrac{\beta }{L^3}\esperancewithstartingpoint{z}{\esperancewithstartingpoint{Z_t}{\indicator{\tau_{\badset} \leq 1}}}.
    \end{align}
    This and \eqref{eq5} together imply
    \begin{equation}
        \esperancewithstartingpoint{z}{\probawithstartingpoint{Z_t}{\tau_{\badset} \leq 1}} = \Ocal{1/L^6}.
    \end{equation}
    Note that the left-hand side of the equation above is exactly $\probawithstartingpoint{z}{\exists s \in [t,t + 1]: Z_s \in \badset}$. So this leads to what we want.
\end{proof}

We now prove Lemma \ref{inertia} and Lemma \ref{lm:ponctuel_badset}.
\begin{proof}[Proof of Lemma \ref{inertia}]
    $\tau_{out}$ is at least the time that the first point of $\Xi$ appears, which has distribution $\exp(L^2 \times L + \lambda L)$, where $\lambda = \sum_{i = 1}^q \lambda_i$. Hence 
    \begin{align*}
        \probawithstartingpoint{z}{\tau_{out} > \dfrac{1}{L^3}} &\geq \proba{\exp(L^3 + \lambda L) > \dfrac{1}{L^3}} \\
        &= \exp\left(- \dfrac{L^3 + \lambda L}{L^3}\right) \\
        &= \exp(-1 - \dfrac{\lambda}{L^2}).
    \end{align*}
    We can choose, for example, $\beta = e^{-1 -\lambda}$.
\end{proof}
To finish this subsection, we prove Lemma \ref{lm:ponctuel_badset}.
\begin{proof}[Proof of Lemma \ref{lm:ponctuel_badset}]
   Recall that the process $Z$ is attractive. Note also that for any $E \subset 
   \lattice$, the set $\{z \subset \lattice| z\cap E = \O\}$ is decreasing. So, we only need to prove the statement for the initial condition $z := \O$. To know whether $E \cap Z^{\O}_t = \O$, a natural way, again, is to trace the history of the set $E$ backward in time. We do it by placing at each site in $E$ at time $t$ a particle labeled by the corresponding element in $\Acal$ and trace them back to time $0$, given the collection of marks generated by $\Xi$.
   \begin{itemize}
       \item When a particle meets an exclusion mark, it jumps to the other endpoint of the corresponding edge.
       \item When a particle, say $\wsf$, meet a correlation mark, it is removed, and $Z^{\O}_t(\wsf) := 0$.
       \item When a particle, say $\wsf$, meet a refresh mark, it is removed, and $Z^{\O}_t(\wsf) := 1$.
       \item If a particle $\wsf$ reaches time $0$, then $Z^{\O}_t(\wsf) := 0$.
   \end{itemize}
   With those rules, we can construct $Z^{\O}_t(E)$. 
   
   It is more convenient to generate the history in the forward direction. Let $\Xi^* = (\xistarexclusion, \xistarglauber, \xistarrefresh)$ be an independent copy of $\Xi$. In another copy of the space-time slab, we construct the process $(E_s)_{0\leq s \leq t}$ takings values in $\Zcal$ and the collection of marks $\Ccal^*$ as follows.
    \begin{itemize}
        \item $E_0:= E$.
        \item Whenever $\xistarexclusion_u$ jumps, say at time $s$, we perform the following operations:
        \begin{enumerate}
            \item Put an exclusion mark on the edge $(u, u+1)$,
            \item $E_s := E_{s-}^{u \leftrightarrow u+1}$.
        \end{enumerate}
        \item Whenever $\xistarrefresh_u$ jumps, $0\leq u \leq L-1$, say at time $s$,
        \begin{itemize}
            \item If $u \leq \absolutevalue{E_{s-}}$, put a refresh mark on the $u-$th site in $E_{s-}$ , say site $u'$. Then $E_s := E_{s-} \setminus\{u'\}$.
            \item If $u > \absolutevalue{E_{s-}}$, put a refresh mark on the $(u - \absolutevalue{E_{s-}})$-th site in $\lattice \setminus E_{s-}$. Then $E_s := E_{s-}$.
        \end{itemize}
        \item Whenever $\xistarglauber_{i,u}$ jumps, say at time $s$,
        \begin{itemize}
            \item if $u \leq \absolutevalue{B(E_{s-}, m)}$, put a Glauber mark on the $u-$th site in $B(E_{s-}, m)$, say site $u'$, and put correlation marks on all the sites in $B(u', m)$. Then $E_s := E_{s-} \setminus B(u',m)$.
            \item if $u > \absolutevalue{B(E_{s-}, m)}$, put a Glauber mark on the $(u-\absolutevalue{B(E_{s-}, m)})-$th site in $\lattice \setminus B(E_{s-}, m)$, say site $u'$, and put correlation marks on all the sites in $B(u', m)$. Then $E_s := E_{s-}$.
        \end{itemize}
    \end{itemize}
    By the homogeneity of the Poisson processes $\Xi$ and $\Xi^*$, the collection of marks $\Ccal^*|_{[0,t]}$ has the same distribution as $\Ccal|_{[0,t]}$ viewed backward in time. In particular, the process $(E_s)_{0\leq s \leq t}$ has the same distribution as the set of particles alive when we track the history of the set $E$ backward from time $t$. In particular, $\proba{Z^{\O}_t \cap E = \O}$ is the same as the probability that no site is removed from $(E_s)_{0\leq s \leq t}$ due to the appearance of a refresh mark.  
    
    Here, we have rearranged the Poisson clocks so that the first clocks in $\xistarrefresh$ indicate the refresh of the sites in $(E_s)_{0\leq s \leq t}$, and the first clocks in $\xistarglauber$ indicate the Glauber updates of the sites in the neighbors of the process $(E_s)_{0 \leq s \leq t}$. This makes our computation easier. More precisely, let $\Gcal_s$ be the \sigmaalgebra generated by $\Xi^*$ up to time $s$. Let $\Tcal_0 = 0$ and $(\tau_i, \Tcal_i)_{i\geq 1}$ be defined recursively as follows.
    \begin{align*}
        \Tcal_{i+1} &= \infset{t > \Tcal_i:\, \text{a clock among $(\xistarrefresh_u)_{1 \leq u \leq \absolutevalue{E_{\Tcal_i}}}, (\xistarglauber_{u,i})_{3 \leq i \leq q, 1 \leq u \leq (2m+1)\absolutevalue{E_{\Tcal_i}} }$ rings at time $t$}}\\
        \tau_{i+1} &= \Tcal_{i+1} - \Tcal_i.
    \end{align*}
        
    Then we have the following:
    \begin{enumerate}
        \item Conditionally on $\Gcal_{\Tcal_i}$, $\tau_{i+1} \sim \exp\left((\lambda_1 + \lambda_2) \absolutevalue{E_{\Tcal_i}} + (\lambda - \lambda_1 - \lambda_2)(2m+1)  \absolutevalue{E_{\Tcal_i}}\right)$. \label{claim1}
        \item \[\absolutevalue{E_s} = \absolutevalue{E_{\Tcal_i}}, \, \forall s\in [\Tcal_i, \Tcal_{i+1}); \; \absolutevalue{E_{\Tcal_{i+1}}} \geq \absolutevalue{E_{\Tcal_i}} - (2m+1),\] \label{claim2}
        and hence 
        \[\absolutevalue{E_{\Tcal_i}} \geq |E| - i(2m+1).\]
        This is because at most $(2m+1)$ sites are removed each time a Glauber mark appears.
        \item 
        \begin{align}\label{claim3}
            &\proba{\text{the ring at $\Tcal_{i+1}$ is among $\xistarrefresh_1, \dots, \xistarrefresh_{\absolutevalue{E_{\Tcal_i}}}$} \big| \Gcal_{\Tcal_{i}}}\nonumber\\
            &= \dfrac{\lambda_1 + \lambda_2}{(\lambda_1+\lambda_2) + (\lambda - \lambda_1 - \lambda_2)(2m+1)}\nonumber\\
            &=: 1-\alpha.
        \end{align}
    \end{enumerate}
    Based on the arguments above, 
    \begin{align*}
        \proba{Z^{\O}_t \cap E = \O} = \proba{\text{For any $\Tcal_i \leq t$, the ring at $\Tcal_i$ is among $\xistarglauber$}}.
    \end{align*}
    Note that, for any $l \in \Zbb_+$,
    \begin{align}
        &\proba{\text{For any $\Tcal_i \leq t$, the ring at $\Tcal_i$ is among $\xistarglauber$}}\nonumber\\ 
        &\leq \proba{\Tcal_l > t} + \proba{\Tcal_l \leq t, \text{the ring at $\Tcal_1, \dots, \Tcal_l$ are among $\xistarglauber$}}.\label{eq7}
    \end{align}
    By equation \eqref{claim3},  
    \begin{equation}\label{eq8}
        \proba{\Tcal_l \leq t, \text{the ring at $\Tcal_1, \dots, \Tcal_l$ are among $\xistarglauber$}} \leq \alpha^l.
    \end{equation}
    We can choose $l = 10\ceil{\dfrac{\log L}{-\log \alpha}}$ to make this term smaller than $\dfrac{1}{L^{10}}$. The bound for the first term on the right side of \eqref{eq7} is a consequence of the concentration of the sum $\tau_1 + \dots + \tau_l$. 
    Let $\beta   := (\lambda_1 + \lambda_2) + (\lambda - \lambda_1 - \lambda_2)(2m+1)$. By Claims \ref{claim1} and \ref{claim2}, conditionally on $\Gcal_{\Tcal_i}$, $\tau_{i+1}$ is stochastically dominated by $\exp\left(\beta (\absolutevalue{E} - i(2m+1)\right)$. 
    Set $\theta := (2m+1)l$, then provided that $|E| > \theta + \dfrac{2\theta}{\beta }$,
    \[\dfrac{\beta (|E| - (i - 1)(2m+1))}{\beta (|E| - (i - 1)(2m+1)) - \theta} \leq 2, \, \forall 1 \leq i \leq l.\] 
    Therefore,
    \begin{align*}
        \esperance{e^{\theta \Tcal_l}} &= \esperance{\esperance{e^{\theta(\Tcal_{l-1} + \tau_l)} |\Gcal_{\Tcal_{l-1}}}}\\
        &= \esperance{e^{\theta \Tcal_{l-1}}\esperance{e^{\tau_l} |\Gcal_{\Tcal_{l-1}}}}\\
        &\leq \esperance{e^{\theta \Tcal_{l-1}} \times \dfrac{\beta (|E| - (l - 1)(2m+1))}{\beta (|E| - (l - 1)(2m+1)) - \theta} }\\
        &\leq 2\esperance{e^{\theta \Tcal_{l-1}}}.
    \end{align*}
    By induction, we deduce that 
    \begin{align*}
         \esperance{e^{\theta \Tcal_l}} \leq 2^l.
    \end{align*}
    Hence, by Chernoff's bound,
    \begin{equation*}
        \proba{\Tcal_l > t} \leq e^{-\theta t} 2^l = (2e^{-(2m+1)t})^l.     
    \end{equation*}
    With $l$ chosen as above, there exists $\beta_3$ such that when $t \geq \beta_3$, the last expression is smaller than $1/L^{10}$. We finish the proof by choosing $\beta_2$ such that $\beta_2 \log L \geq \theta + \theta/\beta $, which clearly exists when $\theta = (2m+1)l,\; l = 10\ceil{\dfrac{\log L}{-\log q}}$.
\end{proof}
\subsection{First step of the proof of Proposition \ref{prop:adjacent_distribution}: reformulation}\label{subsec:localtime}
The rest of the paper is devoted to proving Proposition \ref{prop:adjacent_distribution}. This is done in two steps. In this subsection, we do the first step: reduce our task to estimating the local time that a certain Markov process spends in a certain subset of its state space. In the next subsection, we use techniques from excursion theory to estimate this local time.

Without loss of generality, we can suppose $u = 0$. Throughout this subsection, let $z, \pi_1, \pi_2$ be as in the statement of Proposition \ref{prop:adjacent_distribution}. We first recall a result about perturbations of a product measure. 

\begin{lemma}[Perturbation of a product measure]\label{lm:pertubation_product}
    Let $\Omega = \{-1,1\}^n$. For each subset $S \subset [n]$, let $\varphi_S$ be a distribution on $\{-1,1\}^S$. Let $\rho \in (-1,1)$, and let $\nu$ be the product measure $\rade{\rho}^{\otimes n}$ on $\Omega$. Let $\mu$ be the measure on $\Omega$ obtained by first sampling a subset $S \subset [n]$ via some measure $\mutilde$, and then, conditionally on $S$, generating independently the values on $S$ via $\varphi_S$ and the values on $[n] \setminus S$ via $\rade{p}^{\otimes [n] \setminus S}$. Then 
    \begin{equation*}
        4 \dtv{\mu}{\nu}^2 \leq \norm{\dfrac{\mu}{\nu}-1}^2_{L^2(\nu)} \leq \esperance{\theta^{\absolutevalue{S \cap S'}}} -1,
    \end{equation*}
    where $S, S'$ are \iid with law $\mutilde$, and $\theta = \maxset{\dfrac{2}{1+\rho}, \dfrac{2}{1-\rho}}$.
\end{lemma}
Originally introduced by Miller and Peres in \cite{Miller2012}, this result has been successfully used to prove cutoff for the Glauber dynamics by Lubetzky and Sly in \cite{Lubetzky2015, Lubetzky2016, Lubetzky2017}. It has recently been used in conjunction with negative dependence to prove cutoff for SSEP with reservoirs in \cite{Salez2023, tran2022cutoff}. For the proof of Lemma \ref{lm:pertubation_product}, see Lemma $1$ in \cite{Salez2023}. 

We will use the following graphical construction. We will construct a coupling $(X,Z)$ where $X$ starts from $\pi_1$ or $\pi_2$, and $Z$ starts from $z \cup \{0\}$. Recall that a site in $Z$ is viewed as colored blue. From now on, we will color all sites not in $Z$ black. The content of a site $u$ at time $t$ is its spin $X_t(u) \in \{-1,1\}$ and its color $Z_t(u) \in \{0,1\}$.
From now on, we allow a collection of marks to contain new types of marks: blue and black marks, which can appear on the edge of $\lattice$, and whose effects are explained later. 
\paragraph{Graphical Construction 3.} Let $\xirefresh$ and $\xiglauber$ be defined
as in Graphical Construction $2$. They generate the refresh and Glauber marks as before. However, we replace the effect of the exclusion marks with that of the blue and black marks, as follows.
Let $(\xiblue_u)_{0 \leq u \leq L-1}$ and $(\xiblack_u)_{0 \leq u \leq L-1}$ be independent of $\xirefresh, \xiglauber$ and defined as follows.
    \begin{itemize}
        \item $(\xiblack_u)_{ 0 \leq u \leq L-1}$ are independent Poisson clocks of intensity $L^2$ on the edges. When $\xiblack_u$ jumps, we put a black mark on the edge $(u, u+1)$. 
        \item $(\xiblue_u)_{ 0 \leq u \leq L-1}$ are independent Poisson clocks of intensity $2L^2$ on the edges. When $\xiblue_u$ jumps, we put a blue mark on the edge $(u, u+1)$. 
    \end{itemize}
From now on, let $\Xi := (\xirefresh, \xiglauber, \xiblack, \xiblue)$. Each realization of $\Xi$ defines a collection of marks consisting of refresh, Glauber, black, and blue marks. We construct the process $X$ from this collection of marks as follows.
\begin{itemize}
    \item The effect of the refresh marks and Glauber marks are as before.
    \item Each time we see a black mark, say on the edge $(u, u+1)$, if at least one of the endpoints of the edge $(u, u+1)$ is black, we exchange the contents of those endpoints: $x \mapsto x^{u \leftrightarrow u+1}, z \mapsto z^{u \leftrightarrow u+1}$.
    \item Each time we see a blue mark, say on the edge $(u, u+1)$, if both endpoints are blue, we exchange their spins with probability $1/2$ (we decide by sampling an independent Bernoulli variable with parameter $1/2$).
\end{itemize}
With this construction, $X$ is still a Glauber-Exclusion process, and $Z$ still tracks the independent sites of $X$ if $X$ starts from $\pi_2$.
The idea is to reveal $\Xi$ but not the Rademacher variables $\rade{\rhobar}$ used to refresh the sites at the refresh marks and the Bernoulli variables used to make decisions at the blue marks.

\paragraph{The conditional distribution of $X$.} Conditionally on $\Xi([0,1])$, let $(u_1, \tau_1), (u_2, \tau_2), \dots, (u_k, \tau_k)$ be the positions of the correlation marks that appear on the blue sites, \ie $\forall i,\; (u_i, \tau_i -) \in Z_{\tau_i-}$. In other words, they are the correlation marks that turn some blue sites black. We claim that, conditionally on $\Xi$, almost surely, there is a function $\Phi$ independent of the initial configuration of $X$ such that
   \begin{equation}\label{representation}
       X_1 = \Phi(y, X_{\tau_1-}(u_1), \dots, X_{\tau_k-}(u_k), X_1(Z_1)).
   \end{equation}
   The explanation for the formula above is as follows. At time $1$, almost surely, each site is either black or blue. The spins of the blue sites are given by $X_1(Z_1)$. The spins of the black sites can be constructed from
   \begin{enumerate}
       \item The spins of the black sites at the beginning, \ie\, $y$,
       \item The trajectories of the black sites given by $\xiblack$ and the update functions given by $\xiglauber$. All that information is encoded in the function $\Phi$,
       \item The spins of the blue sites used in the update functions at the Glauber marks. Those spins are exactly $X_{\tau_1-}(u_1), \dots, X_{\tau_k-}(u_k)$.
   \end{enumerate}
Let $\Lfrak_1$ (resp. $\Lfrak_2$) be the distribution of $(X_{\tau_1-}(u_1), \dots, X_{\tau_k-}(u_k), X_1(Z_1))$ when the Glauber-Exclusion process starts from $\pi_1$ (resp. $\pi_2$). We remark here that if we start with the distribution $\pi_2$, then $\Lfrak_2 = \rade{\rhobar}^{\otimes k}\otimes \rade{\rhobar}^{\otimes Z_1}$, and $\Lfrak_1$ is a perturbation of $\Lfrak_2$.    
Now, we can take advantage of the information-percolation framework in \cite{Lubetzky2016}. Thanks to partially revealing the randomness of the graphical construction, we transform the problem into comparing a product law with its perturbation, which can be done using Lemma \ref{lm:pertubation_product}. More formally, we have the following lemma.
\begin{lemma}[Projection]\label{lm:info-perco}
    \begin{align}\label{eq:info-perco}
        \normtv{\probawithstartingpoint{\pi_1}{X_1 \in \cdot} - \probawithstartingpoint{\pi_2}{X_1 \in \cdot}} \leq \esperance{\normtv{\Lfrak_1 - \Lfrak_2}}.
    \end{align}
\end{lemma}
\begin{proof}
    We have
    \begin{align*}
        &\normtv{\probawithstartingpoint{\pi_1}{X_1 \in \cdot} - \probawithstartingpoint{\pi_2}{X_1 \in \cdot}}\\
        &\leq \esperance{\normtv{\probawithstartingpoint{\pi_1}{X_1 \in \cdot \big| \Xi([0,1])} - \probawithstartingpoint{\pi_2}{X_1 \in \cdot \big| \Xi([0,1])}}}\\
        &\leq \esperance{\normtv{\Lfrak_1 - \Lfrak_2}},
    \end{align*}
    where in the last inequality, we have use \eqref{representation} and the definitions of $\Lfrak_1, \Lfrak_2$.
\end{proof}
It remains to estimate the right-hand side in equation \eqref{eq:info-perco}. First, we introduce a random walk used to track the position of the perturbation.
\paragraph{The perturbation position.} 
Let the random walk $(U_1(t))_{t\geq 0}$ taking values in $\lattice \cup \{\dag\}$ be constructed as follows. 
\begin{itemize}
    \item $U_1(0) = 0$ (the walk $U_1$ starts at site $0$).
    \item Whenever a refresh mark or a correlation mark appears on the site occupied by $U_1$, the walk $U_1$ is killed (it becomes the cemetery state $\dag$).
    \item Whenever a black mark appears, if $U_1$ is at one of the two endpoints of the corresponding edge and the other endpoint is black, then $U_1$ takes the edge ($U_1$ jumps to the other endpoint).
    \item Whenever a blue mark appears, if $U_1$ is at one of the endpoints of the mark, with probability $1/2$ (decided by an independent Bernoulli variable with parameter $1/2$), it attempts to take the edge, and it succeeds if the other endpoint is blue and fails if the other endpoint is black. 
\end{itemize}
Let $U_2$ be an independent copy of $U_1$ conditionally on $\Xi$, \ie the decisions of the two walks each time a blue mark appears are independent.
The law of $U_1$ (or $U_2$) conditionally on $\Xi$ is the averaging process, introduced by Aldous in \cite{Aldous2012}, defined as follows.
\paragraph{The averaging process.} Let $H = (H_t)_{t\geq 0}$ be the piece-wise constant process taking values in the space of function from $\lattice$ to $[0,1]$, constructed as follows:
   \begin{itemize}
       \item $H_0(0) = 1; \; H_0(u) = 0, \, \forall u \in \lattice\setminus\{0\}$.
       \item When we see a refresh mark, say at site $u$, we make the transition $h \mapsto h^{u,0}$, where $h^{u,0}$ is obtained from $h$ by replacing the values at $u$ by $0$, \ie  \[\forall u' \in \lattice,\; h^{u,0}(u') = h(u') \times \indicator{u'\neq u}.\]
       \item When we see a Glauber mark, say at site $u$, we make the transition $h \mapsto h^{B(u,m),0}$, where $h^{B(u,m),0}$ is obtained from $h$ by replacing the values in $B(u,m)$ by $0$, \ie  \[\forall u' \in \lattice,\; h^{B(u,m),0}(u') = h(u') \times \indicator{u' \notin B(0,m)}.\]
       \item When a black mark appears, say on the edge $(u, u+1)$, if the process $Z$ makes the transition $z \mapsto z^{u \leftrightarrow u+1}$, we make the transition $h \mapsto h^{u \leftrightarrow u+1}$, where $ h^{u \leftrightarrow u+1}$ is obtained from $h$ by exchanging the values at sites $u$ and $u+1$, \ie \[h^{u \leftrightarrow u+1}(u') = h(u') \times \indicator{u' \notin \{u, u+1\}} + h(u) \times \indicator{u' = u+1} + h(u+1) \times \indicator{u' = u}.\]
       \item When a blue mark appears, say at site $u$, if both sites $u$ and $u+1$ are blue, we replace $h(u)$ and $h(u+1)$ by their average $\dfrac{h(u) + h(u+1)}{2}$, \ie  $h \mapsto \dfrac{h + h^{u\leftrightarrow u+1}}{2}$.
   \end{itemize}
The averaging process is an interesting Markov process by itself. We refer the readers to \cite{Quattropani2023, Caputo2023} for recent development. By construction, $(H_s)_{s\in [0,t]}$ is measurable \wrt $\Xi([0,t])$. Moreover, 
\begin{equation*}
    H_t(u) = \proba{U_1(t) = u| \Xi},\; \forall i \in \lattice, \forall t \geq 0,
\end{equation*}
and in a particular 
\begin{equation}
    H_t(u)^2 = \proba{U_1(t) = U_2(t) = u|\Xi}.    
\end{equation}
Recall that $\distance{\cdot, \cdot}$ denotes the shortest-path distance on $\lattice$. By convention, we define
\begin{align*}
    \distance{u, \dag} := \distance{\dag, u} := \infty, \; \forall u\in \lattice \cup\{\dag\}.
\end{align*}
The purpose of this subsection is to prove the following result.

\begin{lemma}[Comparision with local time]\label{lm:localtimecomparision} There exists a constant $\beta$ such that
    \begin{align}\label{eq31}
        \normtv{\probawithstartingpoint{\pi_1}{X_1 \in \cdot} - \probawithstartingpoint{\pi_2}{X_1 \in \cdot}}^2 \leq \beta \esperance{\int_0^1 \indicator{\distance{U_1(s), U_2(s)} = 0} \drm s}.
    \end{align}
\end{lemma}
Thanks to Lemma \ref{lm:info-perco}, we only need to show the following results to prove Lemma \ref{lm:localtimecomparision}.
\begin{lemma}[Conditional upper bound by the Averaging process]\label{averaging_process}
   \begin{equation}
        \normtv{\Lfrak_1 -\Lfrak_2}^2 \leq \left(\maxset{\dfrac{2}{1+\rhobar}, \dfrac{2}{1 - \rhobar}} - 1\right)\left( \sum_{j = 1}^k H_{\tau_j-}^2(i_j)  + \norm{H_1}_2^2 \right).
    \end{equation}
\end{lemma}

\begin{lemma}\label{lm:Kacformule}
There exist a constant $\beta$ such that 
    \begin{equation}\label{eq:localtime}
        \esperance{\sum_{j = 1}^k H_{\tau_j-}^2(i_j) + \norm{H_1}_2^2} \leq \beta \esperance{\int_0^1 \indicator{\distance{U_1(s), U_2(s)} = 0} \drm s}.    
    \end{equation}
\end{lemma}
It remains to show Lemma \ref{averaging_process} and Lemma \ref{lm:Kacformule}. We first show Lemma \ref{averaging_process}.
\begin{proof}[Proof of Lemma \ref{averaging_process}]  
   By construction, $\Lfrak_2$ is a product of $\rade{\rhobar}$, and $\Lfrak_1$ is a perturbed version of $\Lfrak_2$, and the walk $U_1$ tracks the perturbation position. Let
   \begin{align*}
       S_1 &:= \{(u_j,\tau_j-) | j\in [k],\, U_1(\tau_j-) = u_j\} \cup \{(u,1)|  u\in Z_1,\, U_1(1) = u\}.
   \end{align*}
   In words, $S_1$ contains places perturbed in the vector $(X_{\tau_1-}(u_1), \dots, X_{\tau_k-}(u_k), X_1(Z_1))$. Note that the walk $U_1$ is killed when it meets a correlation mark or a refresh mark. Hence $|S_1| \leq 1$. Let $S_2$ be defined as $S_1$ when we replace $U_1$ by $U_2$. Then, by Lemma \ref{lm:pertubation_product}, 
    \begin{align}
        \normtv{\Lfrak_1 - \Lfrak_2}^2 \leq \esperance{\theta^{\absolutevalue{S_1 \cap S_2}} - 1 | \Xi}, 
    \end{align}
    where $\theta = \maxset{\dfrac{2}{1+\rhobar}, \dfrac{2}{1 - \rhobar}}$. Note that, as $|S_1 \cap S_2| \leq 1$, 
    \begin{align}
        &\esperance{\theta^{\absolutevalue{S_1 \cap S_2}} - 1 \big| \Xi} \\
        &= \esperance{(\theta-1)\absolutevalue{S_1 \cap S_2} \bigg| \Xi}\\  
        &= (\theta-1) \left( \left(\sum_{j = 1}^k \proba{U_1(\tau_j-) = U_2(\tau_j-) = u_j \big| \Xi}\right) + \left(\sum_{u\in \lattice}\proba{U_1(1)= U_2(1) = u|\Xi} \right) \right)\\
        &= (\theta-1)\left(\sum_{j = 1}^k H_{\tau_j-}^2(u_j) \times + \norm{H_1}_2^2 \right),
    \end{align}
    which is what we want.
\end{proof}

Now we show Lemma \ref{lm:Kacformule}.
\begin{proof}[Proof of Lemma \ref{lm:Kacformule}]
Recall that $(u_1, \tau_1), \dots, (u_k,\tau_k)$ are the correlation marks that turn some blue sites green up to time $1$. Let $(u_i, \tau_i)_{i\geq k+1}$ be the correlation marks that turn some blue sites green from time $1$ onwards. Let us define
\begin{equation}
    F_t := \sum_{j = 1}^\infty H^2_{\tau_j-}(u_j) \indicator{\tau_j \leq t}.
\end{equation}
Then $(Z_t, H_t, F_t)_{t\geq 0}$ is a Markov process. In particular, let $\Lcal$ be the generator corresponding to this Markov process. Then for any function $\varphi$ on the state space of this process, 
\begin{align}
    \generator \varphi(z, h, f) = &\sum_{u = 0}^{L-1} L^2\left(\varphi(z^{u\leftrightarrow u+1}, h^{u\leftrightarrow u+1}, f) - \varphi(z,h,f) \right) \left(\indicator{z(u) = 0} \vee \indicator{z(u+1) = 0}\right) \nonumber\\ 
    &+ \sum_{u = 0}^{L-1} 2L^2\left(\varphi\left(z, \dfrac{h + h^{u \leftrightarrow u+1}}{2}, f\right) - \varphi(z,h,f) \right) \indicator{z(u) = z(u+1) = 1} \nonumber\\
    &+ \sum_{u = 0}^{L-1} (\lambda_1 + \lambda_2) \left(\varphi(z^{u,1}, h^{i, 0}, f) - \varphi(z,h,f) \right) \nonumber\\
    &+ \sum_{i = 3}^q\sum_{u = 0}^{L-1} \lambda_i \left(\varphi\left(z\setminus B(u,m), h^{B(u,m),0}, f+ \sum_{u' \in B(0,m)} h^2(u') \indicator{z(u')=1}\right) - \varphi(z,h,f) \right). \nonumber 
\end{align}
Let $\varphi$ be the projection onto the third coordinate. Then
\begin{align*}
    \generator \varphi(z, h, f) &= \sum_{i = 3}^q\sum_{u = 0}^{L-1} \lambda_i \sum_{u' \in B(u,m)}h^2 (u') \indicator{z(u')=1}\\ 
    &= (\lambda -\lambda_1 -\lambda_2) \sum_{u = 0}^{L-1} \sum_{u' = 0}^{L-1} h^2 (u') \indicator{z(u')=1} \indicator{u' \in B(u,m)}\\
    &= (\lambda -\lambda_1 -\lambda_2) \sum_{u' = 0}^{L-1} \sum_{u = 0}^{L-1} h^2 (u') \indicator{z(u')=1} \indicator{u \in B(u',m)}\\
    &\leq (\lambda -\lambda_1 -\lambda_2)(2m+1) \sum_{u' = 0}^{L-1} h^2 (u') \indicator{z(u')=1}\\
    &= (\lambda -\lambda_1 -\lambda_2)(2m+1) \norm{h}_2^2.
\end{align*}

A classical result (see \eg \cite{Ethier1986}) says that
\[F_t - F_0 - \int_0^t \Lcal \varphi (Z_s, H_s, F_s) ds\]
is a zero-mean martingale, and hence 
\begin{equation}\label{eq9}
    \esperance{F_t} = \esperance{F_0} + \esperance{\int_0^t \Lcal \varphi (Z_s, H_s, F_s) ds} \leq (\lambda -\lambda_1 -\lambda_2)(2m+1) \esperance{\int_0^t \norm{H_s}_2^2 ds}.
\end{equation}
Moreover, $\norm{H_t}_2^2$ almost surely decreases, hence 
\begin{equation}\label{eq10}
    \norm{H_1}_2^2 \leq \int_0^1 \norm{H_s}_2^2 \drm s.
\end{equation}
Let 
\begin{align*}
    \beta := (\lambda -\lambda_1 -\lambda_2)(2m+1) + 1.
\end{align*}
Two equations \eqref{eq9}, \eqref{eq10}, and Fubini's theorem together imply that 
\begin{align} \label{eq20}
    \esperance{\sum_{j = 1}^k H_{\tau_j-}^2(i_j) + \norm{H_1}_2^2} \leq \beta\int_0^1  \esperance{\norm{H_s}_2^2} \drm s.
\end{align}
Note that for any $s \in [0,1]$,  
\begin{align*}
    \esperance{\norm{H_s}_2^2} = \esperance{\sum_{u \in \lattice}\proba{U_1(s) = U_2(s)=u|\Xi}} = \esperance{\indicator{\distance{U_1(s), U_2(s)} = 0}}.
\end{align*}
Therefore 
\begin{align} \label{eq21}
    \int_0^1  \esperance{\norm{H_s}_2^2} \drm s = \int_0^1 \esperance{\indicator{\distance{U_1(s), U_2(s)} = 0}} \drm s =  \esperance{\int_0^1 \indicator{\distance{U_1(s), U_2(s)} = 0}\drm s } 
\end{align}
Two equations \eqref{eq20}, \eqref{eq21} lead to what we want.
\end{proof}

\subsection{Second step of the proof of Proposition \ref{prop:adjacent_distribution}: the excursion coupling}\label{subsec:excursion}
This subsection is devoted to estimating the right-hand side of \eqref{eq:localtime}. It should be more convenient to consider a modified version $(\Zbar, \Ubar_1, \Ubar_2)$ of $(Z, U_1, U_2)$, where we do not kill the walks when we see refresh marks or correlation marks, and we always color the sites occupied by $\Ubar_1, \Ubar_2$ blue. Note that before $U_1$ or $U_2$ are killed, $(\Zbar, \Ubar_1, \Ubar_2) = (Z, U_1, U_2)$. Therefore, almost surely, 
\begin{align}\label{eq22}
    \int_0^1 \indicator{\distance{U_1(s), U_2(s)} = 0} \drm s \leq \int_0^1 \indicator{\distance{\Ubar_1(s), \Ubar_2(s)} = 0}\drm s.
\end{align}
 Besides, not killing the walks only increases the blue sites, \ie, almost surely,
\begin{align}\label{eq23}
    \Zbar_s \geq Z_s,\, \forall s \in \Rbb_+.
\end{align} 
An important fact here is the observation of Aldous in \cite{Aldous2012}: $(\Zbar, \Ubar_1, \Ubar_2)$ itself is a Markov process (without conditional probability), constructed as follows. We still use the collection of marks given in Graphical Construction 3.
Suppose that we are at state $(z, u_1, u_2)$.
\begin{itemize}
    \item When a refresh mark appears, say at site $u$,  make the transition $(z, u_1, u_2) \mapsto (z^{u,1}, u_1, u_2)$.
    \item When a Glauber mark appears, say at site $u$,  make the transition $(z, u_1, u_2) \mapsto ((z\setminus b(u,m))\cup\{u_1, u_2\}, u_1, u_2)$.
    \item When a black mark appears, say on the edge $(u, u+1)$, if $z(u) = 0$ or $z(u+1) = 0$,  make the transition $(z, u_1, u_2) \mapsto (z^{u \leftrightarrow u+1}, u_1^{u \leftrightarrow u+1}, u_2^{u \leftrightarrow u+1})$, 
    where $u_1^{u \leftrightarrow u+1}$ is given by
    \begin{equation*}
        u_1^{u \leftrightarrow u+1} = \begin{cases}
        u + 1 &\text{if $u_1 = u$},\\
        u &\text{if $u_1 = u + 1$},\\
        u_1 &\text{if $u_1 \notin\{u, u + 1\}$},
    \end{cases}
    \end{equation*}
    and $u_2^{u \leftrightarrow u+1}$ is defined similarly.
    \item When a blue mark appears, say on edge $(u,u+1)$, if $z(u) = z(u+1) = 1$,
    \begin{align*}
        (z, u_1, u_2) &\mapsto (z, u_1, u_2) &\text{with probability $1/4$},\\
        (z, u_1, u_2) &\mapsto (z, u_1^{u\leftrightarrow u+1}, u_2) &\text{with probability $1/4$},\\
        (z, u_1, u_2) &\mapsto (z, u_1, u_2^{u\leftrightarrow u+1}) &\text{with probability $1/4$},\\
        (z, u_1, u_2) &\mapsto (z, u_1^{u\leftrightarrow u+1}, u_2^{u\leftrightarrow u+1}) &\text{with probability $1/4$}.
    \end{align*}
\end{itemize}

By abuse of notation, we write $\probawithstartingpoint{z, u_1, u_2}{\cdot}, \esperancewithstartingpoint{z,u_1,u_2}{\cdot}$ for the probability and expectation taken with respect to the process $(\Zbar, \Ubar_1, \Ubar_2)$ starting from $(z, u_1, u_2)$. We will compare $(\Ubar_1, \Ubar_2)$ with its idealized version $(\Utilde_1, \Utilde_2)$, defined§ as follows.
\paragraph{Idealized version of $(\Ubar_1, \Ubar_2)$.} Let two random walks $(\Utilde_1, \Utilde_2)$ on $\lattice$ be constructed as follows. Each edge of the lattice is associated with a Poisson clock of intensity $2L^2$. Whenever a clock rings, if $\Utilde_1$ (or $\Utilde_2$) is at one endpoint of the edge, it decides to take the edge with probability $1/2$. If both can take the edge, their decisions are made independently.
Informally, $(\Utilde_1, \Utilde_2)$ is the version of $(\Ubar_1, \Ubar_2)$ where we assume that all the sites are always blue.\\
\begin{figure}
    \centering
    \includegraphics[width = \textwidth]{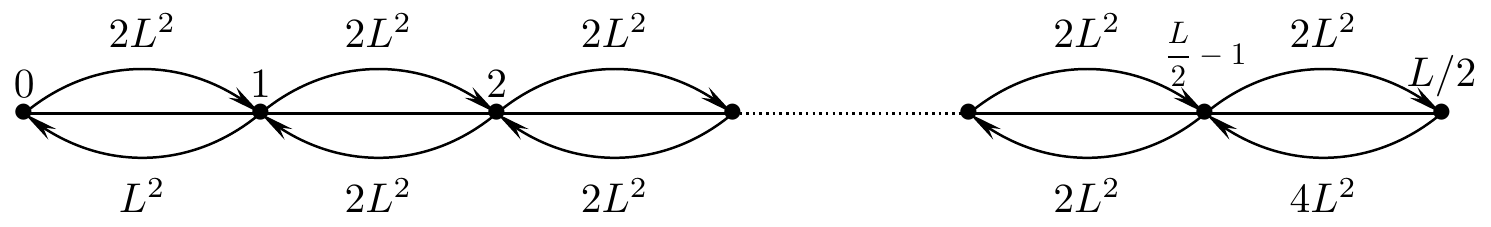}
    \caption{Transition rates of $\distance{\Utilde_1, \Utilde_2}$ when $L$ is even.}
    \label{fig:transition-rate-1}
\end{figure}
\paragraph{The excursion.}
Let $\tauzerotoone_i$ be the time for $i$-th passage from $0$ to $1$, and let $\tauonetozero_i$ be the time for $i$-th passage from $1$ to $0$ of $\left(\distance{\Ubar_1(t), \Ubar_2(t)}\right)_{t\geq0}$. More precisely, let $\tauzerotoone_i, \tauonetozero_i$, and $\Tcal_i$ be defined recursively as follows.
\begin{align}
    \Tcal_j &= \sum_{i = 1}^{j-1} (\tauzerotoone_i + \tauonetozero_i), \label{eq34}\\
    \tauzerotoone_i &= \infset{t \geq 0: \distance{\Ubar_1(\Tcal_i + t), \Ubar_2(\Tcal_i + t)} = 1},\label{eq36}\\
    \tauonetozero_i &= \infset{t \geq 0: \distance{\Ubar_1(\Tcal_i + \tauzerotoone_i + t), \Ubar_2(\Tcal_i +\tauzerotoone_i+ t)} = 0}. \label{eq37} 
\end{align}
The sequence $(\tautildezerotoone_i, \tautildeonetozero_i, \Tcaltilde_i)_{i\geq 1}$ is defined similarly by replacing $(\Ubar_1, \Ubar_2)$ with $(\Utilde_1, \Utilde_2)$. The idea is to express the local time in terms of the excursions:
\begin{align}
    \int_0^1 \indicator{\distance{\Ubar_1(s), \Ubar_2(s)} = 0} \drm s &= \sum_{i \geq 1} \indicator{\Tcal_i < 1} \times (\tauzerotoone_i \wedge [1 - \Tcal_i]),\label{integral_excursion_1}\\
    \int_0^1 \indicator{\distance{\Utilde_1(s), \Utilde_2(s)} = 0} \drm s &= \sum_{i \geq 1} \indicator{\Tcaltilde_i < 1} \times (\tautildezerotoone_i \wedge [1 - \Tcaltilde_i]).\label{integral_excursion_2}
\end{align}
Let us define $\Upsilon$ by: 
\begin{equation}
    \Upsilon = \sum_{i \geq 1} \indicator{\Tcaltilde_i < 1} \times \tautildezerotoone_i.
\end{equation}
From now on, we fix a choice of $\beta_2, \beta_3$ as in Lemma \ref{lm:ponctuel_badset}, and we alway write $\badset$ for $\badset_{\beta_2}$. We have the following results.
\begin{lemma}\label{lm:transition-rate-utilde}
    $\distance{\Utilde_1, \Utilde_2}$ is a Markov process with transition rates given in Figure \ref{fig:transition-rate-1}.
\end{lemma}
\begin{proposition}[Excursion coupling]\label{prop:excursioncoupling}
    There is a coupling of $(\Ubar_1, \Ubar_2)$ and $(\Utilde_1, \Utilde_2)$ where 
    \[\distance{\Ubar_1(0), \Ubar_2(0)} = \distance{\Utilde_1(0), \Utilde_2(0)} = 0\]
    such that, almost surely, for all $i \geq 1$,
    \begin{align}
        \tauzerotoone_i &\geq \tautildezerotoone_i,\label{eq14}\\  
        \tauonetozero_i &= \tautildeonetozero_i.\label{eq15}
    \end{align}
\end{proposition}

\begin{lemma}[Fast passage to $1$ from good environment]\label{lm:excursion_estimate}
There exists a constant $\beta > 0$ such that
    \begin{equation*}
        \max_{t\geq 0} \max_{z \notin \badset, u \in z} \esperancewithstartingpoint{z,u,u}{(\tauzerotoone_1 \wedge t) \times \indicator{\Zbar_s \notin \badset,\, \forall s \in [0, t]}} \leq \beta \dfrac{\log^2 L}{L^2}.
    \end{equation*}
\end{lemma}

\begin{lemma}[Anti-concentration, integral form]\label{lm:anticoncentration_integral_form}
    \begin{equation}\label{eq35}
        \esperance{\Upsilon} = \Ocal{1/L}.
    \end{equation}
\end{lemma}
\begin{remark}
    In dimension $2$, the bound on the right-hand side of \eqref{eq35} is replaced by $\Ocal{\dfrac{\log L}{L^2}}$.
\end{remark}

Now we can prove Proposition \ref{prop:adjacent_distribution}.
\begin{proof}[Proof of Proposition \ref{prop:adjacent_distribution}.]
    Thanks to Lemma \ref{lm:localtimecomparision} and equation \eqref{eq22}, we now only have to estimate $\esperance{\int_0^1 \indicator{\distance{\Ubar_1(s), \Ubar_2(s)} = 0} \drm s}$.
    We have 
    \begin{align}\label{eq32}
        &\esperance{\int_0^1 \indicator{\distance{\Ubar_1(s), \Ubar_2(s)} = 0} \drm s} \nonumber\\
        &= \esperance{\int_0^1 \indicator{\distance{\Ubar_1(s), \Ubar_2(s)} = 0} \drm s \left(\indicator{\exists s \in [0,1]: \Zbar_s \in \badset} + \indicator{\forall s \in [0,1]: \Zbar_s \notin \badset}\right)}.
    \end{align}
    For the first term, we write
    \begin{align}
        \esperance{\int_0^1 \indicator{\distance{\Ubar_1(s), \Ubar_2(s)} = 0} \drm s \indicator{\exists s \in [0,1]: \Zbar_s \in \badset}}&\leq \proba{\exists s \in [0,1]: \Zbar_s \in \badset}\nonumber\\
        &\leq \proba{\exists s \in [0,1]: Z_s \in \badset}, \label{eq30}
    \end{align}
    where we have used \eqref{eq23} in the last inequality. Now, we estimate the second term.
    We will prove that 
    \begin{align}\label{eq33}
        &\esperance{\int_0^1 \indicator{\distance{\Ubar_1(s), \Ubar_2(s)} = 0} \drm s \indicator{\forall s \in [0,1]: \Zbar_s \notin \badset}} = \Ocal{\log^2 L\esperance{\Upsilon}}.
    \end{align}
        
    Thanks to \eqref{integral_excursion_1}, \eqref{integral_excursion_2}, and Lemma \ref{lm:anticoncentration_integral_form}, we only need to prove that there exists a constant $\beta$ such that
    \begin{equation}\label{eq40}
        \esperance{\indicator{\forall s \in [0,1]: \Zbar_s \notin \badset}\indicator{\Tcal_i < 1} \times (\tauzerotoone_i \wedge [1 - \Tcal_i]) } \leq \beta \log^2 L \times \esperance{\indicator{\Tcaltilde_i < 1} \times \tautildezerotoone_i}.
    \end{equation}
    Let $\beta_1$ be the constant in Lemma \ref{lm:excursion_estimate}. By the Markov property at time $\Tcal_i$,
    \begin{align}
        &\esperance{\indicator{\forall s \in [0,1]: \Zbar_s \notin \badset}\indicator{\Tcal_i < 1} \times (\tauzerotoone_i \wedge [1 - \Tcal_i])} \nonumber\\
        &= \esperance{ \indicator{\forall s \in [0,\Tcal_i]: \Zbar_s \notin \badset}\indicator{\Tcal_i < 1} \times \esperancewithstartingpoint{\Zbar(\Tcal_i),\Ubar_1(\Tcal_i),\Ubar_2(\Tcal_i)}{\indicator{\forall s \in [0,1 - \Tcal_i]: \Zbar_s \notin \badset} \times \left(\tauzerotoone_1 \wedge (1 - \Tcal_i)\right) }} \nonumber\\
        &\leq \esperance{\indicator{\forall s \in [0,\Tcal_i]: \Zbar_s \notin \badset}\indicator{\Tcal_i < 1}  \dfrac{\beta_1\log^2 L}{L^2}}\nonumber\\
        &\leq \dfrac{\beta_1\log^2 L}{L^2} \proba{\Tcal_i < 1}\nonumber\\
        &\leq \dfrac{\beta_1\log^2 L}{L^2} \proba{\Tcaltilde_i < 1}. \label{eq12}
    \end{align}
    Here we have used Lemma \ref{lm:excursion_estimate} in the first inequality and the fact that $\Tcal_i \geq \Tcaltilde_i$ (due to Proposition \ref{prop:excursioncoupling}) in the last inequality. On the other hand, 
    \begin{align}\label{13}
       \esperance{\indicator{\Tcaltilde_i < 1} \times \tautildezerotoone_i  } &= \esperance{\indicator{\Tcaltilde_i < 1} \esperancewithstartingpoint{\Utilde_1(\Tcaltilde_i), \Utilde_1(\Tcaltilde_i)}{\tautildezerotoone_1} } \nonumber\\
        &= \esperance{\indicator{\Tcaltilde_i < 1} \dfrac{1}{2L^2}}.
    \end{align}
    Here we have used the fact that $\distance{\Utilde_1(\Tcaltilde_i), \Utilde_2(\Tcaltilde_i)} = 0$ by construction, and that the expected time needed for $\distance{\Utilde_1, \Utilde_2}$ to reach $1$ from this time onwards is $\dfrac{1}{2L^2}$, due to Lemma \ref{lm:transition-rate-utilde}. 
    So \eqref{eq12} and \eqref{13} together imply \eqref{eq40}. The equations \eqref{eq31}, \eqref{eq22}, \eqref{eq32}, \eqref{eq33}, \eqref{eq30} together imply that 
    \begin{equation*}
         \normtv{\probawithstartingpoint{\pi_1}{X_t \in \cdot} - \probawithstartingpoint{\pi_2}{X_t \in \cdot}}^2 =  \Ocal{\dfrac{\log^2 L}{L} + \probawithstartingpoint{z \cup \{0\}}{\exists s \in [0,1]: Z_s \in \badset}}.
    \end{equation*}
    Together with the inequality $\sqrt{a + b} \leq \sqrt{a} + \sqrt{b},\, a,b >0$, the inequality above gives us what we want.
\end{proof}
It remains to prove Lemma \ref{lm:transition-rate-utilde}, Proposition \ref{prop:excursioncoupling}, Lemma \ref{lm:excursion_estimate}, and Lemma \ref{lm:anticoncentration_integral_form}. We will need the following result to prove Lemma \ref{lm:transition-rate-utilde} and Proposition \ref{prop:excursioncoupling}.
\begin{lemma}[Strong lumpability of Markov process]\label{lm:lumpability}
    Let $\Xcal$ and $\Ycal$ be two finite sets, and let $f: \Xcal \to \Ycal$ be a surjective function. Let $(X_t)_{t\geq 0}$ be a Markov process with a generator $\Lcal$ on $\Xcal$. Then $(f(X_t))_{t\geq 0}$ is a Markov process for any initial condition of $(X_t)_{t\geq 0}$ if and only if
    \begin{align}\label{eq24}
        \forall y_1, y_2 \in \Ycal, y_1 \neq y_2, \forall x_1, x'_1 \in f^{-1}(y_1),\;  \sum_{x \in f^{-1}(y_2)} \Lcal(x_1, x) = \sum_{x \in f^{-1}(y_2)} \Lcal(x'_1, x).
    \end{align}
    Moreover, if \eqref{eq24} is satisfied, then the generator $\Lcal'$ of $(f(X_t))_{t\geq 0}$ is determined by
    \begin{align}
        \forall y_1, y_2 \in \Ycal, y_1 \neq y_2, \forall x_1\in f^{-1}(y_1), \Lcal'(y_1, y_2) = \sum_{x \in f^{-1}(y_2)} \Lcal(x_1, x).
    \end{align}
\end{lemma}
Lemma \ref{lm:lumpability} is a classical result. For its proof, the readers can see Theorem 2.4 in \cite{Ball1993}, see also Lemma 2.5 in \cite{Levin2017} for the analogous result for discrete-time Markov chains. Now we prove Lemma \ref{lm:transition-rate-utilde} and Proposition \ref{prop:excursioncoupling}.

\begin{proof}[Proof of Lemma \ref{lm:transition-rate-utilde}]
    This is just a direct application of Lemma \ref{lm:lumpability} for the process $(\Utilde_1, \Utilde_2)$ and the function $\distance{\cdot, \cdot}$.
\end{proof}

\begin{proof}[Proof of Proposition \ref{prop:excursioncoupling}]
    The intuition is that $(\Ubar_1, \Ubar_2)$ evolves exactly like $(\Utilde_1, \Utilde_2)$ when $\Ubar_1$ and $\Ubar_2$ are not on a same site. This explains \eqref{eq15}. The only difference is when $\Ubar_1$ and $\Ubar_2$ are on a same site since they might be forced to move together when their neighbors are not blue. This explains \eqref{eq14}.
    
    More precisely, for any initial configuration of $(\Zbar, \Ubar_1, \Ubar_2)$, consider the modified version $(\Zbar^{mod}, \Ubar^{mod}_1, \Ubar^{mod}_2)$ of $(\Zbar, \Ubar_1, \Ubar_2)$ which is stopped at $E:= \{(z, u, u)|z \in \Zcal, u \in \lattice\}$.
    Then $(\Zbar^{mod}, \Ubar^{mod}_1, \Ubar^{mod}_2)$ is also a Markov process, and up to the time that $(\Zbar, \Ubar_1, \Ubar_2)$ reaches $E$, those two processes are the same. The modified version $(\Utilde^{mod}_1, \Utilde^{mod}_2)$ of $(\Utilde_1, \Utilde_2)$ is defined similarly. 
    
    We apply Lemma \ref{lm:lumpability} for the function 
    \begin{align*}
        f: (z, u_1, u_2) \mapsto \distance{u_1, u_2}
    \end{align*}
    and the chain $(\Zbar^{mod}, \Ubar^{mod}_1, \Ubar^{mod}_2)$ to see that $f(\Zbar^{mod}, \Ubar^{mod}_1, \Ubar^{mod}_2)$ is also a Markov process. Let $\Lcal$ be the generator associated with this process. By direct computation, the process with generator $\Lcal$ can be described as in Figure \ref{fig:transition-rate-1}, except that the transition from $0$ to $1$ is suppressed.

    Similarly, we apply Lemma \ref{lm:lumpability} for the function
    \begin{align*}
        \ftilde: (u_1, u_2) \mapsto \distance{u_1, u_2}
    \end{align*}
    and the process $(\Utilde^{mod}_1, \Utilde^{mod}_2)$ to see that $\ftilde(\Utilde^{mod}_1, \Utilde^{mod}_2)$ is also a Markov process. By direct computation, we can see that its generator is also $\Lcal$.

    Our coupling is as follows. Let $\Gcal_t$ be the \sigmaalgebra generated by $\Xi([0,t])$. Recall the definition of $(\Tcal_i)$ in \eqref{eq34}.
    \begin{itemize}
        \item Conditionally on $\Gcal_{\Tcal_i}$, let $\tau_i$ be the first time counted from $\Tcal_i$ onwards that $\Ubar_1$ meets a blue mark and the two walks $\Ubar_1, \Ubar_2$ take different decisions. Then,  
        \begin{align*}
            \tau_i &\leq \tauzerotoone_i,
        \end{align*}
        where the inequality is strict if the other endpoint of the blue mark is a black site. 
        \item Conditionally on $\Gcal_{\Tcal_i + \tauzerotoone_i}$, $\distance{\Ubar_1, \Ubar_2}$ evolves as a Markov process with generator $\Lcal$ until it reaches $0$. This gives us the excursion $\tauonetozero_i$. 
    \end{itemize}
    Note that, by construction, 
    \begin{itemize}
        \item $\tau_i$ is independent of $\Gcal_{\Tcal_i}$, and hence of $(\tau_j, \tauonetozero_j)_{1 \leq j \leq i-1}$.
        \item $\tauonetozero_i$ is independent of $\Gcal_{\Tcal_i + \tauzerotoone_i}$ and hence of $(\tau_j, \tauonetozero_j)_{1 \leq j \leq i-1}$.
        \item For any $i$, $\tau_i \sim \exp(2L^2)$.
        \item For any $i$, $\tauonetozero_i$ is the time of an excursion of a Markov process with generator $\Lcal$ starting at state $1$ until reaching state $0$.
    \end{itemize}
    Hence $(\tau_i, \tauonetozero_i)_{i\geq 1}$ are i.i.d., and they have the same distribution as the excursions $(\tautildezerotoone_i, \tautildeonetozero_i)_{i \geq 1}$. This finishes our proof.
\end{proof}

Now we prove Lemma \ref{lm:excursion_estimate}.
We will need the following result.
\begin{lemma}[Fast transition in good environment]\label{lm:fasttransition}
    There exist constants $\beta, \beta' > 0$ such that
    \begin{equation*}
        \min_{z \notin \badset, u \in z }\probawithstartingpoint{z, u, u}{\tauzerotoone_1 < \dfrac{\beta'\log^2 L}{L^2}} \geq \beta.
    \end{equation*}
\end{lemma}

\begin{proof}[Proof of Lemma \ref{lm:excursion_estimate}]
    Let $z \in \badset$, $u \in z$ arbitrary. Let $\beta, \beta'$ be as in Lemma \ref{lm:fasttransition}. Then
    \begin{align*}
         &\esperancewithstartingpoint{z,u,u}{(\tauzerotoone_1 \wedge t) \times \indicator{\Zbar_s \notin \badset,\, \forall s \in [0, t]}}\\
         &\leq \sum_{i = 0}^{\floor{tL^2/(\beta'\log^2 L)}} \esperancewithstartingpoint{z,u,u}{\tauzerotoone_1 \times \indicator{\Zbar_s \notin \badset,\, \forall s \in [0, t]} \indicator{\dfrac{i \beta' \log^2 L}{L^2} \leq \tauzerotoone_1 < \dfrac{(i+1) \beta' \log^2 L}{L^2} }}\\
         &\;\;\;\; + \esperancewithstartingpoint{z,u,u}{t \times \indicator{\Zbar_s \notin \badset,\, \forall s \in [0, t]} \indicator{\tauzerotoone_1 > t}}\\
         &\leq \left(\sum_{i = 0}^{\floor{tL^2/(\beta'\log^2 L) }} \dfrac{(i+1)\beta' \log^2 L}{L^2} \esperancewithstartingpoint{z,u,u}{ \indicator{\Zbar_s \notin \badset,\, \forall s \in [0, t]} \indicator{\dfrac{i \beta' \log^2 L}{L^2} \leq \tauzerotoone_1 }} \right)\\
         &\;\;\;\; + t\esperancewithstartingpoint{z,u,u}{ \indicator{\Zbar_s \notin \badset,\, \forall s \in [0, t]} \indicator{\tauzerotoone_1 > \floor{tL^2/(\beta' \log^2 L)} (\beta' \log^2L)/ (L^2)}}\\
         &\leq \left( \sum_{i = 0}^{\floor{tL^2/(\beta'\log^2 L) }} \dfrac{(i+1) \beta' \log^2 L}{L^2} (1- \beta)^i \right) + t(1-\beta)^{\floor{tL^2/(\beta'\log^2L)}}\\
         &= \Ocal{\log^2 L/L^2}.
    \end{align*}
\end{proof}
Now we prove Lemma \ref{lm:fasttransition}.
We recall the following classical result about the hitting time of a simple random walk. 
\begin{lemma}[Hitting time of a symmetric simple random walk]\label{lm:hittingtime}
    Let $U$ be a simple random walk on $\Zbb$ where edges have conductance $1$. Let $\tau:= \minset{ t\geq 0: U(t) = 0}$ be the hitting time of the site $0$. Then there exists a constant $\beta> 0$ such that
    \begin{align}
        \max_{u \in \Zbb}\probawithstartingpoint{u}{\tau \leq 2|u|^2} > \beta.
    \end{align}
    The conclusion holds also for a walk in $\Zbb^2$.
\end{lemma}

\begin{proof}[Proof of Lemma \ref{lm:fasttransition}]
    First, consider the case where one of the neighbors of $u$, say $u+1$, is blue. Then there is a probability bounded away from $0$ that a blue mark appears on edge $(u, u+1)$ before time $1/L^2$ and before any blue or black mark appears on edges $(u+1, u+2)$ and $(u-1, u)$. When this happens, with probability $1/2$, the two walks $\Ubar_1, \Ubar_2$ can take different decisions. Therefore, $\tauzerotoone_1 < 1/L^2$ with a probability greater than some constant $\beta_4 > 0$.     

    Coming back to the general case, since $z \notin \badset$, there exists a blue site $u'$ at distance at most $2\beta_2\log L$ from $u$. Let $U_3$ be the walk starting at $u'$, which behaves like $U_1$ (see paragraph ``the perturbation position") when it meets the marks. 
    
    Let $\beta_1$ be a constant that we will choose later. The number of Glauber marks appearing in the interval $\left[0,\dfrac{\beta_1 \log^2 L}{L^2}\right]$ is a Poisson variable with average $\dfrac{\lambda \beta_1 \log^2 L}{L}$. Hence, for $L$ large enough,
    \begin{align*}
        \proba{\text{no Glauber mark appears in $\left[0,\dfrac{\beta_1 \log^2 L}{L^2}\right]$}} \geq 1/2.
    \end{align*}
    Conditionally on this event, $U_3$ and $\Ubar_1$ evolve as two independent random walks on $\lattice$ in $\left[0,\dfrac{\beta_1 \log^2 L}{L^2}\right]$, until they are at distance $1$ from each other. Let 
    \begin{align*}
        \tau:= \infset{t \geq 0| \distance{U_3(t), \Ubar_1(t)} = 1}.
    \end{align*}
    Note that $U_3 - \Ubar_1$ is a simple random walk on $\lattice$ where edges have conductance $2L^2$ until it reaches $1$ or $-1$. So for $\beta$ as in Lemma \ref{lm:hittingtime}, there exists a constant $\beta_1$ such that 
    \begin{align*}
        \proba{\tau < \dfrac{\beta_1 \log^2 L}{L^2} \Bigg|\text{no Glauber mark appears in $\left[0,\dfrac{\beta_1 \log^2 L}{L^2}\right]$}} \geq \beta.
    \end{align*}
    Hence, 
    \begin{align*}
        \proba{\tau < \dfrac{\beta_1 \log^2 L}{L^2}, \text{no Glauber mark appears in $\left[0,\dfrac{\beta_1 \log^2 L}{L^2}\right]$}} \geq \beta/2.
    \end{align*}
    By the strong Markov property at time $\tau$ and the first case, we can deduce that 
    \begin{align*}
        \proba{\tauzerotoone_1 \leq \tau + 1/L^2, \tau < \dfrac{\beta_1 \log^2 L}{L^2}, \text{no Glauber mark appears in $\left[0,\dfrac{ \beta_1 \log^2 L}{L^2}\right]$}} \geq \beta_4 \beta/2.
    \end{align*}
    This finishes our proof.
\end{proof} 

\begin{remark}
    The proof in dimension $d = 2$ is similar. Lemma \ref{lm:hittingtime} is still valid for this dimension. $U_3$ and $\Ubar_1$ still evolve as 2 independent walks until they are at distance $1$ from each others or until $U_3$ is killed when a Glauber mark appears. The number of Glauber marks appearing in the time interval $\left[0,\dfrac{\beta_1 \log^2 L}{L^2}\right]$ is $\Ocal{\log^2L}$. Each time there is a ring among the Glauber clocks, the chance that the mark appearing kills $U_3$ is $\Ocal{1/L^2}$. So the chance that $U_3$ is killed in the time interval $\left[0,\dfrac{ \beta_1 \log^2 L}{L^2}\right]$ is $\Ocal{\dfrac{\log^2 L}{L^2}} = o(1)$. So $U_3$ can still meet $\Ubar_1$ in time interval $\left[0,\dfrac{ \beta_1 \log^2 L}{L^2}\right]$ with a chance bounded away from $0$, and the proof is adapted accordingly.
\end{remark}

It remains to prove Lemma \ref{lm:anticoncentration_integral_form}. It is not hard to prove the result if we replace $(\Utilde_1, \Utilde_2)$ with two independent SRWs. We will prove the results for two independent SRWs and use a comparison argument to transfer the result back to the pair $(\Utilde_1, \Utilde_2)$. 

Let $\Ucaltilde_1$ be a SRW starting from $0$ on $\lattice$, where every edge has conductance $L^2$. Let $\Ucaltilde_2$ be an independent copy of $\Ucaltilde_1$. Let $(\taucheckzerotoone_i, \taucheckonetozero_i, \Tcalcheck_i)_{i\geq 1}$ be defined as in \eqref{eq34}, \eqref{eq36}, \eqref{eq37} when we replace $(\Ubar_1, \Ubar_2)$ by $(\Ucaltilde_1, \Ucaltilde_2)$. For two real random variables $\zeta_1, \zeta_2$, we write $\zeta_1\overset{d}{\leq} \zeta_2$ if $\zeta_2$ dominates stochastically $\zeta_1$. We will need the following results.
\begin{figure}
    \centering
    \includegraphics[width = \textwidth]{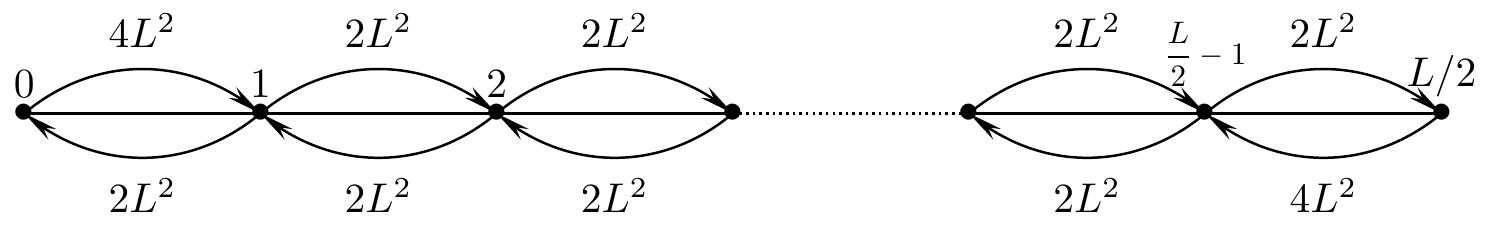}
    \caption{Transition rate of $\distance{\Ucaltilde_1, \Ucaltilde_2}$ when $L$ is even.}
    \label{fig:transition-rate-2}
\end{figure}
\begin{lemma}\label{lm:5}
    \begin{align}
        \taucheckonetozero_1 &\overset{d}{\leq} \tautildeonetozero_1,\label{eq39}\\
        \taucheckzerotoone_1 &\overset{d}{\leq} \tautildezerotoone_1. \label{eq38}
    \end{align}
\end{lemma}
\begin{lemma}\label{lm:6}
    \begin{align}
        \sum_{i=1}^\infty \proba{\Tcalcheck_i < 1} = \Ocal{L}.
    \end{align}
\end{lemma}
Now we prove Lemma \ref{lm:anticoncentration_integral_form}.
\begin{proof}[Proof of Lemma \ref{lm:anticoncentration_integral_form}]
    In this proof, $\esperancewithstartingpoint{0}{\cdot}$ means the expectation taken with respect to the law of the process $\distance{\Utilde_1, \Utilde_2}$ starting from $0$. We have
    \begin{align*}
        \esperance{\Upsilon} &= \sum_{i\geq 1} \esperance{\indicator{\Tcaltilde_i < 1}\times \taucheckzerotoone_i}\\
        &= \sum_{i\geq 1} \esperance{\indicator{\Tcaltilde_i < 1}\times \esperancewithstartingpoint{0}{\tautildezerotoone_1}}\\
        &= \sum_{i\geq 1} \dfrac{1}{2L^2} \proba{\Tcaltilde_i < 1},
    \end{align*}
    where we have used the Markov property for the process $\distance{\Utilde_1, \Utilde_2}$ in the second equality, and the formula of the generator of $\distance{\Utilde_1, \Utilde_2}$ in the third equality. Note that, by Lemma \ref{lm:5}, $\Tcalcheck_i \overset{d}{\leq} \Tcaltilde_i,\, \forall i \geq 1$. Therefore, 
    \begin{align*}
        \sum_{i\geq 1} \dfrac{1}{2L^2} \proba{\Tcaltilde_i < 1} &\leq \dfrac{1}{2L^2}  \sum_{i\geq 1} \proba{\Tcalcheck_i < 1} \\
        &= \dfrac{1}{2L^2} \Ocal{L}\\
        &= \Ocal{L}.
    \end{align*}
    This finishes our proof.
\end{proof}
It remains to prove Lemma \ref{lm:5} and Lemma \ref{lm:6}. We first prove Lemma \ref{lm:5}.
\begin{proof}[Proof of Lemma \ref{lm:5}]
    See Figure \ref{fig:transition-rate-1} and \ref{fig:transition-rate-2} for intuition. A direct application of Lemma \ref{lm:lumpability} shows that $\distance{\Ucaltilde_1, \Ucaltilde_2}$ is a Markov process with transition rates as in Figure \ref{fig:transition-rate-2}. Then we see that 
    \begin{align*}
        \tautildezerotoone_1 &\sim \exp(2L^2),\\
        \taucheckzerotoone_1 &\sim \exp(4L^2).
    \end{align*}
    This proves \eqref{eq38}. Now let $\tautwotoone$ be the (random) time it takes for $\distance{\Utilde_1, \Utilde_2}$ to reach state $1$ from state $2$. By Figure \ref{fig:transition-rate-1} and Figure \ref{fig:transition-rate-2}, it is also the time it takes for $\distance{\Ucaltilde_1, \Ucaltilde_2}$ to reach state $1$ from state $2$. Let $(\zeta_i)_{i\geq 1}, (\tautwotoone_i)_{i\geq 1}$, and $(\xi_i)_{i\geq 1}$ be three independent sequences of \iid random variables such that 
    \begin{itemize}
        \item $(\zeta_i)_{i\geq 1}$ is a sequence of \iid $\exp(L^2)$.
        \item $(\tautwotoone_i)_{i\geq 1}$ is a sequence of independent copies of $\tautwotoone$.
        \item $(\xi_i)_{i\geq 1}$ is a sequence of \iid Bernoulli random variables with parameter $2/3$.
    \end{itemize}
    Then we see that 
    \begin{align*}
        \tautildeonetozero_1 &\overset{d}{=} \dfrac{\zeta_1}{3} + \xi_1 \tautwotoone_1 + \xi_1 \dfrac{\zeta_2}{3} + \xi_1\xi_2\tautwotoone_2 + \dots\\
        &= \sum_{i = 1}^\infty\left[ \left(\dfrac{\zeta_i}{3} \prod_{j = 1}^{i-1} \xi_i\right) + \left(\tautwotoone_i \prod_{j = 1}^{i} \xi_i\right)\right].
    \end{align*}
    This is because to go to state $0$ from state $1$, first, we need to jump out of state $1$, which takes a time $\exp(3L^2)$. Then, with probability $1/3$, we succeed to jump to state $0$, and with probability $2/3$, we jump to state $2$. Then we have to come back to state $1$ and try again. Similarly, let $(\xi'_i)_{i\geq 1}$ be a sequence of \iid Bernoulli random variables with parameter $1/2$, independent of $(\zeta_i)_{i\geq 1}$ and $(\tautwotoone_i)_{i\geq 1}$. Then 
    \begin{align*}
        \taucheckonetozero_1 \overset{d}{=} \sum_{i = 1}^\infty\left[ \left(\dfrac{\zeta_i}{4} \prod_{j = 1}^{i-1} \xi'_i\right) + \left(\tautwotoone_i \prod_{j = 1}^{i} \xi'_i\right)\right].
    \end{align*}
    Since $\xi'_i \overset{d}{\leq} \xi_i$, we conclude that $\taucheckonetozero_1 \overset{d}{\leq} \tautildeonetozero_1$, which finishes our proof.
\end{proof}

Finally, we show Lemma \ref{lm:6}. We will need the following classical result about anticoncentration of SRW on the lattice. 
\begin{lemma}[Local limit theorem]\label{lm:7}
    Let $\Ucaltilde$ be a SRW on $\lattice$ starting from $0$, where every edge has conductance $L^2$. Then, 
    \begin{align*}
        \max_{u \in \lattice} \probawithstartingpoint{0}{\Ucaltilde(t) = u} = \Ocal{\dfrac{1}{L}\left(\dfrac{1}{\sqrt{t} }+1\right)}.
    \end{align*}
\end{lemma}

\begin{proof}[Proof of Lemma \ref{lm:6}]
    Let us define 
    \begin{align*}
        \check{\Upsilon} := \sum_{i \geq 1} \indicator{\Tcalcheck_i < 1} \times \taucheckzerotoone_i.
    \end{align*}
    We see that
    \begin{align*}
        \check{\Upsilon} - \int_0^1 \indicator{\distance{\Ucaltilde_1(s), \Ucaltilde_2(s)} = 0} \drm s = \indicator{\distance{\Ucaltilde_1(1), \Ucaltilde_2(1)} = 0} \times \left(\infset{t \geq 1| \distance{\Ucaltilde_1(t), \Ucaltilde_2(t)} = 1} - 1\right)
    \end{align*}
    Taking the expectation and using the Markov property at time $1$, we get
    \begin{align*}
        \esperance{\check{\Upsilon} - \int_0^1 \indicator{\distance{\Ucaltilde_1(s), \Ucaltilde_2(s)} = 0} \drm s} &= \proba{\distance{\Ucaltilde_1(1), \Ucaltilde_2(1)} = 0} \esperance{\taucheckzerotoone_1}\\
        &\leq 1 \times \dfrac{1}{4L^2}\\
        &= \dfrac{1}{4L^2}.
    \end{align*}
    Note that $\indicator{\distance{\Ucaltilde_1(s), \Ucaltilde_2(s)} = 0} = \indicator{\Ucaltilde_1(s) - \Ucaltilde_2(s) = 0}$, and note also that $(\Ucaltilde_1(s) - \Ucaltilde_2(s))_{s\geq 0}$ is a SRW on the lattice, where edges have conductance $2L^2$. Hence by Lemma \ref{lm:7}, 
    \begin{align*}
        \proba{\Ucaltilde_1(s) - \Ucaltilde_2(s) = 0} = \Ocal{\dfrac{1}{L}\left(\dfrac{1}{\sqrt{s}} + 1\right)}.
    \end{align*}
    Therefore, 
    \begin{align*}
        \esperance{\check{\Upsilon}} &\leq  \dfrac{1}{4L^2} + \esperance{\int_0^1 \indicator{\distance{\Ucaltilde_1(s), \Ucaltilde_2(s)} = 0} \drm s} \\
        &= \dfrac{1}{4L^2} + \int_0^1 \proba{\Ucaltilde_1(s) - \Ucaltilde_2(s) = 0} \drm s\\
        &= \dfrac{1}{4L^2} + \Ocal{\int_0^1 \dfrac{1}{L}\left(\dfrac{1}{\sqrt{s}} + 1\right) \drm s}\\
        &= \Ocal{\dfrac{1}{L}}.
    \end{align*}
    On the other hand, 
    \begin{align*}
        \esperance{\check{\Upsilon}} = \sum_{i \geq 1} \proba{\Tcalcheck_i < 1} \times \dfrac{1}{4L^2}.
    \end{align*}
    The two equations above imply that 
    \begin{align*}
        \sum_{i \geq 1} \proba{\Tcalcheck_i < 1} = \Ocal{L},
    \end{align*}
    which finishes our proof.
\end{proof}
\newpage
\appendixtitleon
\appendixtitletocon
\begin{appendices}
\section{}\label{app:a}

In this appendix, we prove Lemma \ref{lm:anticoncentration}. We will need the following result.
\begin{lemma}[Anticoncentration of SRW, integral form]\label{lm:3}
    Let $\Ucaltilde_1, \Ucaltilde_2$ be two \iid SRWs on $\latticedimensiond$ where edges have conductance $L^2$. Let $\theta$ be a strictly positive number and $\zeta \sim \exp(\theta)$, independent of $(\Ucaltilde_1, \Ucaltilde_2)$. Then, for any $k \in \Zbb_+$,
        \begin{equation*}
            \max_{u_1, u_2 \in \latticedimensiond}\probawithstartingpoint{u_1, u_2}{\distance{\Ucaltilde_1(\zeta), \Ucaltilde_2(\zeta)} \leq k} = \begin{cases}
                \mathcal{O}_{\theta,k} (1/L) &\text{if $d = 1$},\\
                \mathcal{O}_{\theta,k} (\log L/L^2) &\text{if $d = 2$},\\
                \mathcal{O}_{\theta,k} (1/L^2) &\text{if $d \geq 3$}.
        \end{cases} 
    \end{equation*}
\end{lemma}

For $u_1, u_2 \in \lattice$, we write $u_1 \preceq u_2$ if $\distance{u_1, 0} \leq \distance{u_2, 0}$. In higher dimensions, for $u_1, u_2 \in \latticedimensiond$, we write $u_1 \preceq u_2$ if there exists a permutation $\sigma$ of $[d]$ such that $\forall i \in [d], u_1(i) \preceq u_2(\sigma(i))$. We will also need the following result.
\begin{lemma}[Distance between $2$ walks in an $IP(2)$ is bigger than that between 2 independent SRWs]\label{lm:4}
    Consider the lattice $\latticedimensiond$ where every edge has conductance $L^2$. Let $(\Ucal_1, \Ucal_2)$ be an IP(2) and $(\Ucaltilde_1, \Ucaltilde_2)$ be a pair of independent SRWs on $\latticedimensiond$. There is a Markovian coupling of $(\Ucal_1, \Ucal_2)$ and $(\Ucaltilde_1, \Ucaltilde_2)$ such that if the initial condition $(u_1, u_2, \utilde_1, \utilde_2)$ satisfies $(\utilde_2 - \utilde_1) \preceq (u_2 - u_1)$, then almost surely,
    \begin{align*}
        \forall t \geq 0,\, (\Ucaltilde_2(t)- \Ucaltilde_1(t)) \preceq (\Ucal_2(t) - \Ucal_1(t)),
    \end{align*}
    and in particular,
    \begin{align*}
        \distance{\Ucaltilde_2(t), \Ucaltilde_1(t)} \leq \distance{\Ucal_2(t), \Ucal_1(t)}.
    \end{align*}
\end{lemma}
Clearly, Lemma \ref{lm:3} and Lemma \ref{lm:4} together imply Lemma \ref{lm:anticoncentration}.
\begin{proof}[Proof of Lemma \ref{lm:anticoncentration}]
    Let $u_1, u_2 \in \lattice$. Consider the coupling of the IP(2) $(\Ucal_1, \Ucal_2)$ and a pair of independent SRWs $(\Ucaltilde_1, \Ucaltilde_2)$, both starting from $u_1, u_2$, independent of $\zeta$, that satisfies Lemma \ref{lm:4}. Then, by Lemma \ref{lm:4}, 
    \begin{align*}
        \proba{\distance{\Ucal_1(\zeta), \Ucal_2(\zeta)} \leq k} \leq \proba{\distance{\Ucaltilde_1(\zeta), \Ucaltilde_2(\zeta)}\leq k}.
    \end{align*}
    This and Lemma \ref{lm:3} lead to what we want.
\end{proof}

Lemma \ref{lm:3} is just the integral form of Lemma \ref{lm:7}.

\begin{proof}[Proof of Lemma \ref{lm:3}]
    We denote by $\Ucaltilde_1(t,i)$ the $i$-th coordinate of $\Ucaltilde_1(t)$, and similarly for $\Ucaltilde_2$. We see that $(\Ucaltilde_1(\cdot, i) - \Ucaltilde_2(\cdot, i))_{1 \leq i \leq d}$ are $d$ independent SRWs on $\lattice$, where every edge has conductance $2L^2$. Hence, for any $t \geq 0$,
    \begin{align*}
        \proba{\distance{\Ucal_1(t), \Ucal_2(t)} \leq k} &\leq \prod_{i = 1}^d \proba{\left(\Ucaltilde_1(t,i) - \Ucaltilde_2(t,i)\right) \in \{-k, \dots, k\}}\\
        &\leq \prod_{i =1}^d \left( (2k+1) \max_{u \in \lattice}\proba{\Ucaltilde_1(t,i) - \Ucaltilde_2(t,i) = u} \right)\\
        &=  \mathcal{O}_k\left(\dfrac{1}{L^d}\left( \dfrac{1}{t^{d/2}} + 1\right)\right).
    \end{align*}
    Note that
    \begin{align*}
         \proba{\distance{\Ucaltilde_1(\zeta), \Ucaltilde_2(\zeta)} \leq k} = \int_0^\infty  \proba{\distance{\Ucaltilde_1(t), \Ucaltilde_2(t)} \leq k} \theta e^{-\theta t} \drm t.
    \end{align*}
    We split the integral above into three intervals $[0,1/L^2],\, [1/L^2,1],\, [1, \infty)$. First, we estimate the integral on $[0, 1/L^2]$: 
    \begin{align*}
        &\int_0^{1/L^2}  \proba{\distance{\Ucaltilde_1(t), \Ucaltilde_2(t)} \leq k} \theta e^{-\theta t} \drm t \\
        &\leq \int_0^{1/L^2}  \theta e^{-\theta t} \drm t\\
        &\leq \theta/L^2.
    \end{align*}
    Now we estimate the integral on $[1, \infty)$:
    \begin{align*}
        &\int_1^{\infty} \proba{\distance{\Ucaltilde_1(t), \Ucaltilde_2(t)} \leq k} \theta e^{-\theta t} \drm t \\
        &= \mathcal{O}_k\left( \int_1^\infty \dfrac{1}{L^d}\left( \dfrac{1}{t^{d/2}} + 1\right) \theta e^{-\theta t} \drm t \right)\\
        &= \mathcal{O}_k\left( \dfrac{1}{L^d} \int_1^\infty  \theta e^{-\theta t} \drm t \right)\\
        &= \mathcal{O}_k\left( \dfrac{1}{L^d} \right).
    \end{align*}
    Finally, we estimate the integral on $[1/L^2,1]$:
    \begin{align*}
        &\int_{1/L^2}^1 \proba{\distance{\Ucaltilde_1(t), \Ucaltilde_2(t)} \leq k} \theta e^{-\theta t} \drm t\\
        &=  \mathcal{O}_k\left( \int_{1/L^2}^1 \dfrac{1}{L^d}\left( \dfrac{1}{t^{d/2}} + 1\right) \theta e^{-\theta t} \drm t \right)\\
        &= \mathcal{O}_k\left( \dfrac{\theta}{L^d} \int_{1/L^2}^1 \dfrac{1}{t^{d/2}} \drm t \right)\\
        &= \begin{cases}
            \mathcal{O}_k\left( \dfrac{\theta}{L} \left(2\sqrt{t} \Big\vert^1_{1/L^2} \right)\right) = \mathcal{O}_k \left( \dfrac{\theta}{L} \right) &\text{if $d = 1$}, \\
            \mathcal{O}_k\left( \dfrac{\theta}{L^2} \left(\log t \Big\vert^1_{1/L^2} \right)\right) = \mathcal{O}_k \left( \dfrac{\theta \log L}{L^2} \right) &\text{if $d = 2$},\\
            \mathcal{O}_k\left( \dfrac{\theta}{L^d} \left(\dfrac{t^{1-d/2}}{1-d/2} \Bigg\vert^1_{1/L^2} \right)\right) = \mathcal{O}_k \left( \dfrac{\theta}{L^2} \right) &\text{if $d \geq 3$}.
        \end{cases}
    \end{align*}
    The estimates above lead to what we want.
\end{proof}
Now we prove Lemma \ref{lm:4}.
\begin{proof}[Proof of Lemma \ref{lm:4}]
    The intuition is simple: an IP(2) evolves exactly like $2$ independent SRWs when the two particles are not next to each other. The distance between 2 walks in IP(2) is forbidden to jump to $0$, while the distance between two independent SRWs can. This makes the distance between 2 particles in $IP(2)$ always bigger than that of $2$ independent SRWs in distribution.
    
    More formally, note that $Y := (\Ucal_1 - \Ucal_2)$ and $\Ytilde := (\Ucaltilde_1 - \Ucaltilde_2)$ are also Markov processes. We only have to construct a coupling of $Y$ and $\Ytilde$. 

    Let $\Lcal$ (resp. $\tilde{\Lcal}$) be the generator of $Y$ (resp $\Ytilde$). Let $\efrak_1, \dots, \efrak_d$ be the unit vectors in $\latticedimensiond$. 
    Then 
    \begin{align*}
        \tilde{\Lcal}(u, u') = \begin{cases}
            2L^2 &\text{ if $u' = u\pm \efrak_i$ for some $i \in [d]$},\\
            0 &\text{ if not},
        \end{cases}
    \end{align*}
    and 
    \begin{align*}
        \Lcal(u, u') = \begin{cases}
            2L^2 &\text{ if $u' = u\pm \efrak_i$ for some $i \in [d]$ and $u' \neq 0$},\\
            L^2 &\text{ if $(u, u') \in \{(\pm \efrak_i, \mp \efrak_i)| i\in [d]\}$},\\
            0 &\text{ if not},
        \end{cases}
    \end{align*}
    To construct the coupling kernel $\Lcal^{coupling}$, we only need to construct the jump rates from each configuration $(u, \utilde)$ such that $u \preceq \utilde$. Without loss of generality, we can suppose that $\distance{u(i), 0} \leq \distance{\utilde(i), 0},\, \forall i \in [d]$. The kernel is as follows. For any coordinate $i$,
    \begin{itemize}
        \item If $\distance{u(i), 0} < \distance{\utilde(i), 0}$, 
        \begin{align*}
            \Lcal^{coupling}((u, \utilde), (u\pm \efrak_i, \utilde)) = \Lcal^{coupling}((u, \utilde), (u, \utilde \pm \efrak_i)) = 2L^2.
        \end{align*}
        \item If $\distance{u(i), 0} = \distance{\utilde(i), 0}$, then $u(i) = \pm \utilde(i)$, and therefore there exists $\xi \in \{-1,1\}$ such that $\distance{u(i) \pm 1, 0} = \distance{\utilde(i) \pm \xi, 0}$. Then
        \begin{align*}
            \Lcal^{coupling}((u, \utilde), (u\pm \efrak_i, \utilde \pm \xi \efrak_i)) =  2L^2 \indicator{u \pm e_i \neq 0}.
        \end{align*}
        If $u \in \{\pm \efrak_i\}$, then we necessarily have $\utilde \in \{\pm \efrak_i\}$, and in this case,
        \begin{align*}
            \Lcal^{coupling}((u, \utilde), (-u,\utilde)) &= L^2,\\
            \Lcal^{coupling}((u, \utilde), (u, 0)) &= 2L^2.
        \end{align*}
    \end{itemize}
    For any $(u', \utilde')\neq (u, \utilde)$ not listed in the cases above, 
        \begin{align*}
            \Lcal^{coupling}((u, \utilde), (u', \utilde')) = 0.
        \end{align*}
    We can verify that this gives us a coupling kernel for $(Y, \Ytilde)$ that preserves the relation $\preceq$. This finishes our proof.
\end{proof}

\end{appendices}

\bibliographystyle{abbrv}
\bibliography{ref}

\end{document}